\documentclass[11pt]{article}
\usepackage{multicol}
\usepackage{wrapfig}
\usepackage{lipsum}
\usepackage[colorlinks=true,linkcolor=black,citecolor=black,hyperfootnotes=false]{hyperref}
\usepackage{float}
\usepackage{amssymb}
\usepackage{stmaryrd}
\usepackage{mathalfa}
\usepackage{euscript}
\usepackage{extarrows}
\usepackage[normalem]{ulem}
\usepackage{amsmath}
\usepackage{amsthm}

\numberwithin{equation}{section}

\theoremstyle{plain}

\newtheorem{thm}[equation]{Theorem}

\newtheorem{lem}[equation]{Lemma}
\newtheorem{cor}[equation]{Corollary}

\newtheorem{prop}[equation]{Proposition}

\newtheorem*{claim*}{Claim}
\newtheorem*{thm*}{Theorem}
\newtheorem*{thmc*}{Coherence Theorem}

\theoremstyle{definition}
\newtheorem{defn}[equation]{Definition}
\newtheorem{informal-definition}[equation]{Informal Definition}

\newtheorem{rem}[equation]{Remark}

\newtheorem{example}[equation]{Example}

\allowdisplaybreaks
\usepackage{tabularx}
\usepackage{xcolor}
\usepackage{framed}
\usepackage{tikz}
\usetikzlibrary{decorations.markings}
\usetikzlibrary{shapes.geometric}
\usepackage{booktabs}
\usepackage{euscript}
\usepackage[normalem]{ulem}
\usepackage{array}
\usepackage{collectbox}
\newcommand{\bo}{{\raisebox{1pt}{\tiny\mbox{$\Box$}}}}
\usepackage{footnote}
\usepackage{makecell}
\pgfdeclarelayer{edgelayer}
\pgfdeclarelayer{nodelayer}
\pgfsetlayers{edgelayer,nodelayer,main}
\usepackage{array}
\usepackage{collectbox}
\usetikzlibrary{decorations.pathmorphing, patterns,shapes}
\newcommand{\xupdownarrow}[1]{%
  {\left\Updownarrow\vbox to #1{}\right.\kern-\nulldelimiterspace}
}

\makeatletter
\newcommand{\mybox}{%
    \collectbox{%
        \setlength{\fboxsep}{5pt}%
        \fbox{\BOXCONTENT}%
    }%
}
\makeatother
\def\tobar{\mathrel{\mkern3mu  \vcenter{\hbox{$\scriptscriptstyle+$}}%
                    \mkern-12mu{\to}}}

\usepackage{multirow}

\newcommand{\set}[1]{\{#1\}}

\newcommand{\pcat}[1]{\|#1\|}   
 \usetikzlibrary{calc}
\newcolumntype{P}[1]{>{\centering\arraybackslash}p{#1}}
\newcolumntype{M}[1]{>{\centering\arraybackslash}m{#1}}

\title{Categorified cyclic operads} 
\author
{Pierre-Louis Curien$^{1}$ and Jovana Obradovi\' c$^{2}$ \\
\small{\small$^{1}$ \em IRIF,  Universit\' e Paris Diderot and $\pi r^2$ team,  
 Inria, France}\\
\small{\small$^{2}$ \em Institute of Mathematics of the Czech Academy of Sciences}\\
{\small$^{1}$\url{curien@irif.fr}}\enspace\enspace\enspace
\small$^{2}$\url{obradovic@math.cas.cz}}

\date{}

\begin{document} 

\maketitle 
\begin{abstract}
\noindent   In this paper, we introduce a notion of categorified cyclic operad for   set-based cyclic operads with symmetries.  
Our categorification is obtained by  relaxing   defining axioms of cyclic operads   to isomorphisms and by formulating coherence conditions for these isomorphisms. The coherence theorem that we prove has the form ``all  diagrams of  canonical isomorphisms commute''.  Our coherence results come in two flavours, corresponding to the ``entries-only" and ``exchangeable-output" definitions of cyclic operads.  Our proof of  coherence in the entries-only style
 is of syntactic nature and   relies  on the coherence of categorified non-symmetric operads established by Do\v sen and Petri\' c. We obtain the coherence  in the exchangeable-output style   by ``lifting'' the equivalence between entries-only and exchangeable-output cyclic operads, set up by the second author. Finally,  we show  that a generalization
of the structure of profunctors  of B\' enabou provides an example of
categorified cyclic operad, and we exploit the coherence of categorified cyclic
operads in proving that the Feynman category for cyclic operads, due to Kaufmann and Ward, admits an odd version.
\end{abstract}
$^{\ast}$  The work on the final version of this paper was supported by
Praemium Academiae of M. Markl and RVO:67985840. The second author was additionally supported by the grant GA CR P201/12/G028.
\section*{Introduction}
 For the purposes of higher-dimensional category theory and homotopy theory, various concepts of categorification recently   emerged in operad theory, where at least three definitions of  categorified operads have been proposed. In \cite{street1}, Day and Street define  pseudo-operads   by categorifying the original ``monoidal'' definition of operads of Kelly \cite{kelly}, which leads to an algebraic, ``one-line'' characterization  of the form:  a pseudo-operad is a pseudo-monoid in a certain monoidal 2-category. In \cite{dp}, Do\v sen and Petri\' c introduce the notion of  weak Cat-operad by categorifying the  definition of non-symmetric operads with (biased) composition operations $\circ_i$, which leads  to an equational axiomatic definition, in the style of Mac Lane's definition of a monoidal category. In \cite{val},  Dehling and Vallette investigate  higher homotopy (symmetric) operads,  
by means of curved Koszul duality theory.
 \\[0.1cm]
\indent In this paper, we propose a categorification of the    notion of  cyclic operad with symmetries, following 
the  steps  of \cite{dp}, to which we add the categorification of the cyclic structure, and 
the action of the symmetric group, both with  strict and  (like in  \cite{val}) relaxed equivariance.   
We replace sets (of operations of the same arity) with categories, obtaining in this way the intermediate notion of cyclic operad enriched over ${\bf Cat}$, for which we then  relax defining axioms to isomorphisms and exhibit the conditions which make these isomorphisms coherent. The coherence theorem has the form ``all  diagrams made of canonical isomorphisms commute''. \\[0.1cm]
\indent   In the original approach of Getzler and Kapranov \cite[Theorem 2.2]{Getzler:1994pn},  cyclic operads are  seen as enrichments of  operads with simultaneous composition, determined by  adding to the action of permuting the inputs of an operation an  action of interchanging its output  with one of the inputs, in a suitably  compatible way. In \cite[Proposition 42]{opsprops}, Markl gave an adaptation of their definition, by considering underlying operads with partial composition. Both of these definitions are {\em skeletal}, meaning that  the labeling of inputs of operations comes from the {\em skeleton} ${\bf\Sigma}$ of the category ${\bf Bij}$ of finite sets and bijections. The non-skeletal variant of Markl's definition, obtained by passing from  ${\bf\Sigma}$ to ${\bf Bij}$,  has been given in \cite[Definition 3.16]{mo}.  We suggestively refer to these three definitions as the {\em exchangeable-output} definitions of cyclic operads.  The fact that two operations of a cyclic operad can  be composed along  inputs that ``used to be outputs" and outputs that ``used to be inputs"  leads to another point of view on cyclic operads, in  which an operation, instead of having inputs and an (exchangeable) output, now has only ``entries", and can be composed with another operation along any of them.  Such  {\em entries-only} definitions are \cite[Definition 3.2]{mo} (non-skeletal) and \cite[Definition 1.4]{djms} (skeletal).   \\[0.1cm]
\indent The categorified cyclic operads that we introduce are obtained by categorifying   the entries-only definition \cite[Definition 3.2]{mo}, which, thanks to the equivalence \cite[Theorem 2]{mo} of the two presentations of cyclic operads, leads to the appropriate categorification of the exchangeable-output definition  \cite[Definition 3.16]{mo} as well. \\[0.1cm]
\indent 
We relax the axioms in three stages. In the first stage, which makes the most fundamental part of the categorification,    the associativity and commutativity axioms 
of {\em non-unital} entries-only cyclic operads become the {\em sequential associator} and {\em commutator} isomorphisms, with instances $$\beta^{x,\underline{x};y,\underline{y}}_{f,g,h}:(f\, {_x\circ_{\underline{x}}}\, g)\, {_y\circ_{\underline{y}}}\, h\rightarrow f\, {_x\circ_{\underline{x}}}\, (g\, {_y\circ_{\underline{y}}}\, h)\quad \mbox{ and } \quad \gamma^{x,y}_{f,g}:f\, {_x\circ_{{y}}}\, g\rightarrow g\, {_y\circ_{{x}}}\, f,$$ respectively, while the equivariance axiom remains strict.
At first glance,  the coherence of this structure seems easily reducible to   the coherence of symmetric monoidal categories of Mac Lane (see \cite[Section XI.1]{maclane}): all diagrams made of   instances of  the sequential associator and commutator are required to commute. However, in the setting of cyclic operads, where the existence
of operations  is restricted, these instances do not  exist for all possible $f$, $g$ and $h$. As a consequence, the coherence conditions of symmetric monoidal categories  do
not solve our coherence problem. In particular, the hexagon of Mac Lane   is  not  well-defined in this setting.

The coherence conditions that we do take from Mac Lane are the pentagon and the requirement that the commutator isomorphism is involutive. Borrowing the terminology from \cite{dp}, we need two more {\em mixed} coherence conditions  (i.e., coherence conditions that involve both sequential associator and commutator), a hexagon (which is  {\em not} the hexagon of Mac Lane) and a decagon, as well as three more conditions which deal with the action of the symmetric group on   morphisms of categories of operations of the same arity.

The approach we take to treat the coherence problem is of syntactic, term-rewriting spirit,   as  in \cite{maclane} and \cite{dp}, and relies on the coherence result of \cite{dp}. The proof of the coherence theorem consists of three faithful reductions, each restricting the coherence problem to a smaller class of diagrams, in order to finally reach diagrams that correspond to diagrams of canonical isomorphisms of categorified non-symmetric skeletal operads, i.e., weak Cat-operads of \cite{dp}. Intuitively speaking, the first reduction excludes the action of the symmetric group, the second  removes ``cyclicity'', and the third replaces non-skeletality with skeletality.
\\[0.1cm]
\indent 
 In the second and third stages of categorification, we additionally relax equivariance, and we  take units into consideration and relax the unit law, respectively.
The coherence for this notion of categorified  cyclic operads, for which {\em all} the axioms are relaxed, follows easily: we modify the first reduction so that the relaxed equivariance is taken into account, and we add the reduction zero, which removes units. \\[0.1cm]
\indent 
Categorified cyclic operads with strict equivariance 
appear ``in nature''.  We provide an example where the operations are generalized profunctors.
On the other hand, our coherence theorem for  categorified 
cyclic operads for which the equivariance is also relaxed   is precisely the tool needed to prove   that the  Feynman category for cyclic operads, introduced by Kaufmann and Ward in \cite{kauf}, admits an odd version.  \\[0.1cm]
\indent For cyclic operads in  exchangeable-output style,  in the first stage of categorification, 
the two associativity axioms of the underlying operad ${\EuScript O}$ become the {\em sequential associator} and {\em parallel associator} isomorphisms, with instances $$\beta^{x,y}_{f,g,h}:(f\circ_x g)\circ_y h\rightarrow f\circ_x (g\circ_y h)\quad \mbox{ and } \quad \theta^{x,y}_{f,g,h}:(f\circ_x g)\circ_y h\rightarrow (f\circ_y h)\circ_x g, $$ respectively. Therefore, the operadic part of the obtained structure is   the non-skeletal and symmetric counterpart of a weak Cat-operad of \cite{dp}. However,  the correct notion of categorified exchangeable-output cyclic operad should be equivalent with its entries-only counterpart, which, by  ``categorifying'' the equivalence \cite[Theorem 2]{mo} between  the two presentations, implies that   
 an axiom of the extra structure (accounting for the input-output exchange) should also be relaxed. This leads to a third isomorphism, called the {\em exchange}, whose instances are $$\alpha^{z,x;v}_{f,g}:D_z(f\circ_x g)\rightarrow D_{zv}(g)\circ_v D_{xz}(f),$$ where $D_{z}(X):{\EuScript O}(X)\rightarrow {\EuScript O}(X)$ is the endofunctor that ``exchanges the input $z\in X$ with the output'', and $D_{zy}(X):{\EuScript O}(X)\rightarrow {\EuScript O}(X\backslash\{z\}\cup\{y\})$ is the functor that ``exchanges the input $z\in X$ with the output and then renames it to $y$''. Likewise, in order to further relax the equivariance and unit laws, one simply has to further adapt  \cite[Theorem 2]{mo}. \\[0.1cm]
\indent  By ``categorifying''    the equivalence between non-skeletal and skeletal operads, established in \cite[Theorem 1.61]{mss}, and extending it to the corresponding structures of categorified cyclic operads, the non-skeletal notions of   categorified cyclic operads that we introduce can be straightforwardly coerced to skeletal notions. In this way,   categorifications of \cite[Proposition 42]{opsprops} and  \cite[Definition 1.4]{djms}   are obtained.

\paragraph{Layout.}  In Section \ref{s1}, we recall the entries-only definition \cite[Definition 3.2]{mo}    and the\linebreak exchangeable-output definition  \cite[Definition 3.16]{mo} of cyclic operads, both without the structure of units.  In  Section \ref{s2}, we introduce the principal  categorification of   entries-only   cyclic operads,   by relaxing the associativity and commutativity axioms, while keeping the equivariance strict.  The largest part of the section is devoted to the proof of the coherence theorem of this notion of  categorified   entries-only   cyclic operads. In Section \ref{s3}, we further relax equivariance, add units and relax the unit law.   Section \ref{s4} shows alternative presentations (exchangeable-output, skeletal) of our categorified notion. Section \ref{s5} is devoted to  examples.

\paragraph{Notation and conventions.} {\em About finite sets and bijections.} In this paper, we shall use two different notions of union. In the category {\bf Set} of sets and functions, for   sets $X$ and $Y$, $X+Y$ will denote the coproduct (disjoint union) of $X$ and $Y$ (constructed in the usual way by tagging $X$ and $Y$, by, say, $1$ and $2$). If $\sigma:X\rightarrow X'$ and $\tau: Y\rightarrow Y'$ are bijections,  $\sigma+\tau:X+Y\rightarrow X'+Y'$ will denote the canonically defined bijection between the corresponding disjoint unions. In the category {\bf Bij} of   finite sets and bijections, we shall denote the {\em ordinary} union of {\em already disjoint} sets $X$ and $Y$ with $X\cup Y$. For a bijection $\sigma:X'\rightarrow X$ and  $Y\subseteq X$, we shall  denote with $\sigma|^{Y}$ the corestriction of $\sigma$ on $\sigma^{-1}(Y)$.  If $\sigma(x')=x$,  we shall denote with $\sigma[{y/x}]$ the bijection defined in the same way as $\sigma$, except that the pair $(x',x)\in\sigma$ is replaced with $(y,y)$.   If $\tau:Y'\rightarrow Y$ is a bijection such that $X'\cap Y'=X\cap Y=\emptyset$,  then $\sigma\cup \tau:X'\cup Y'\rightarrow X\cup Y$ denotes the bijection defined  as $\sigma$ on $X'$ and as $\tau$ on $Y'$.  If $\kappa:X\backslash\{x\}\cup \{x'\}\rightarrow X$ is identity on $X\backslash\{x\}$ and $\kappa(x')=x$, we say that $\kappa$ renames $x$ to $x'$ (notice the contravariant nature of this convention). If a bijection $\kappa:X\rightarrow X$ renames $x$ to $y$ and $y$ to $x$, we say that it exchanges $x$ and $y$. \\[0.1cm]
  {\em About categories and functors.} We shall denote the singleton category with ${\bf 1}$ and the category of Abelian groups with ${\bf Ab}$.  For a functor ${\EuScript C}:{\bf Bij}^{op}\rightarrow {\bf Set}$ (resp. ${\EuScript C}:{\bf Bij}^{op}\rightarrow {\bf Cat}$) and a bijection $\sigma:Y\rightarrow X$, we shall use the notation $(-)^{\sigma}$ for  ${\EuScript C}(\sigma)(-)$.
   
\section{Cyclic operads}\label{s1}
This section is a reminder on the two biased definitions of cyclic operads with symmetries. These are the definitions whose categorifications we introduce in the following   sections. 
\subsection{The entries-only definition}
We recall below   \cite[Definition 3.2]{mo}. We omit the structure of units.  
 \begin{defn}\label{entriesonly}
An  {\em entries-only cyclic operad} is a functor ${\EuScript C}:{\bf Bij}^{op}\rightarrow {\bf Set}$,    together with a family of functions
$${{_{x}\circ_{y}}}:{\EuScript C}(X)\times {\EuScript C}(Y)\rightarrow {\EuScript C}(X\backslash\{x\}\cup Y\backslash\{y\}) ,$$  called {\em partial compositions}, indexed by arbitrary non-empty finite sets $X$ and $Y$ and elements $x\in X$ and $y\in Y$, such that $X\backslash\{x\}\cap Y\backslash \{y\}=\emptyset $.
These data must satisfy the  axioms given below, where, for each of the equalities, it is assumed that both hand sides are well defined.  \\[0.2cm]
\hypertarget{(A1)}{{\em Sequential associativity.}} For $f\in {\EuScript C}(X)$, $g\in {\EuScript C}(Y)$ and $h\in {\EuScript C}(Z)$,  the following   equality holds:\\[0.25cm]
\indent \hypertarget{A1}{{\texttt{(A1)}}} $(f\, {_{x}\circ_{y}}\,\, g)\,\,{_{u}\circ_z}\, h =  f {_{x}\circ_{y}}\,\, (g\,\,{_{u}\circ_{z}}\, h)$, where $x\in X$, $y,u\in Y,$ $z\in Z$.\\[0.25cm]
\hypertarget{(CO)}{{\em Commutativity.}} For  $f\in{\EuScript C}(X)$, $g\in {\EuScript C}(Y)$, $x\in X$ and $y\in Y$, the following equality holds:\\[0.25cm]
\indent  {\texttt{(CO)}} $f\, {_x\circ_y} \,\, g=g\, {_y\circ_x} \,\, f$.\\[0.25cm]
\hypertarget{(EQ)}{{\em Equivariance.}} For bijections $\sigma_1:X'\rightarrow X$ and $\sigma_2:Y'\rightarrow Y$, and $f\in{\EuScript C}(X)$ and $g\in {\EuScript C}(Y)$, the following equality holds:\\[0.25cm] 
\indent  {\texttt{(EQ)}} $f^{\sigma_1}\,\,{_{{ {\sigma_1^{-1}}(x)}}\circ_{\sigma_2^{-1}(y)}}\,\, g^{\sigma_2}=(f {_x\circ_y} \,\, g)^{\sigma}$,
where    $\sigma=\sigma_1|^{X\backslash\{x\}}\cup \sigma_2|^{Y\backslash\{y\}}$.\\[0.25cm]
For a finite set $X$, the elements of  ${\EuScript C}(X)$ are called the   operations of ${\EuScript C}$ with   entries indexed by the elements of $X$. An entries-only cyclic operad ${\EuScript C}$ is  constant-free if ${\EuScript C}(\emptyset)={\EuScript C}(\{x\})=\emptyset$, for all singletons $\{x\}$.  \hfill$\square$
\end{defn}

 \subsection{The exchangeable-output definition}
We now recall  \cite[Definition 3.16]{mo}, again  leaving out   the structure of units. 
\begin{defn}\label{exoutput}
An {\em exchangeable-output   cyclic operad}  is an operad ${\EuScript O}:{\bf Bij}^{op}\rightarrow {\bf Set}$ (defined as in \cite[Definition 2.3]{mo}, with units omitted), enriched with actions  $$D_{x}:{\EuScript O}(X)\rightarrow {\EuScript O}(X), $$ defined for all  $x\in X$ and subject to the  axioms given below (with $f\in {\EuScript O}(X)$):
\\[0.1cm]
{\em Inverse.} For $x\in X$, \\[0.25cm]
\indent \hypertarget{[DIN]}{\texttt{[DIN]}} $D_{x}(D_{x}(f))=f$.\\[0.25cm]
{\em Equivariance.} For $x\in X$ and an arbitrary bijection $\sigma:Y\rightarrow X$, \\[0.25cm]
\indent \hypertarget{[DEQ]}{\texttt{[DEQ]}} $ D_x(f)^{\sigma} = D_{\sigma^{-1}(x)}(f^{\sigma})$.\\[0.25cm]
{\em Exchange.}
For $x,y\in X$ and a bijection $\sigma:X\rightarrow X$ that exchanges $x$ and $y$,\\[0.25cm]
\indent \hypertarget{[DEX]}{\texttt{[DEX]}}  $D_{x}(f)^{\sigma}=D_x(D_y(f)).$\\[0.25cm]
{\em Compatibility with operadic compositions.} For $g\in {\EuScript O}(Y)$, the following equality holds: \\[0.25cm]
\indent \hypertarget{[DC1]}{\texttt{[DC1]}}  $D_{y}(f\circ_{x}g)=D_{y}(f)\circ_{x}g$, where $y\in X\backslash\{x\}$, and  \\[0.25cm] 
\indent \hypertarget{[DC2]}{\texttt{[DC2]}}  $D_{y}(f\circ_{x}g)=D_{y}(g)^{\sigma_1}\circ_{v}D_{x}(f)^{\sigma_2}$, where $y\in Y$, $\sigma_1:Y\backslash\{y\}\cup\{v\}\rightarrow Y$ renames \linebreak 
\phantom{\indent \texttt{(iiii)}}$y$ to  $v$ and $\sigma_2:X\backslash\{x\}\cup\{y\}\rightarrow X$ renames $x$ to $y$.\\[0.25cm]
For a finite set $X$, the elements of  ${\EuScript O}(X)$ are called the {\em operations of} ${\EuScript O}$  with   inputs indexed by the elements of $X$.  An exchangeable-output cyclic operad ${\EuScript O}$ is  constant-free if ${\EuScript O}(\emptyset)=\emptyset$.  \hfill$\square$
\end{defn}
\section{Categorified   cyclic operads}\label{s2}
This section deals with our principal categorification of entries-only cyclic operads. The categorification is made by relaxing the axioms \hyperlink{(A1)}{\texttt{(A1)}} and \hyperlink{(CO)}{\texttt{(CO)}} of Definition \ref{entriesonly}. In this section, the axiom \hyperlink{(EQ)}{\texttt{(EQ)}} remains strict. In Section \ref{ajaja}, we introduce the categorified notion and exhibit important properties. in Section \ref{blibli}, we state the coherence theorem. Sections \ref{frst} through \ref{cohe} are dedicated to its proof.
\subsection{The definition and properties}\label{ajaja}
The quest for  coherence led us to the following definition.   Below, we use Latin letters  for operations of a categorified cyclic operad, and Greek letters for morphisms between them.

\begin{defn}\label{cat}
A {\em categorified entries-only cyclic operad} is a functor ${\EuScript C}:{\bf Bij}^{\it op}\rightarrow {\bf Cat}$, together with\\[-0.6cm]
\begin{itemize}
\item a family of bifunctors $$_x\circ _y:{\EuScript C}(X)\times{\EuScript C}(Y)\rightarrow {\EuScript C}(X\backslash\{x\}\cup Y\backslash\{y\}),$$   called {partial compositions}, indexed by arbitrary non-empty finite sets $X$ and $Y$ and elements $x\in X$ and $y\in Y$, such that $X\backslash\{x\}\cap  Y\backslash\{y\}=\emptyset$,   which are subject to the equivariance axiom \hyperlink{(EQ)}{\texttt{(EQ)}} of Definition \ref{entriesonly},
and
\item
two natural isomorphisms, $\beta$ and $\gamma$, called the {\em sequential associator} and the {\em  commutator}, whose respective components  $$\beta^{x,\underline{x};y,\underline{y}}_{f,g,h}:(f\,{_x\circ _{\underline{x}}}\,g)\,{_y\circ _{\underline{y}}}\,h\rightarrow f\,{_x\circ _{\underline{x}}}\,(g\,{_y\circ _{\underline{y}}}\,h)\mbox{\quad and\quad} \gamma^{x,y}_{f,g}:f\,{_x\circ _y}\,g\rightarrow g\,{_y\circ _x}\,f\,,$$ are natural in $f$, $g$ and $h$, and are subject to the following coherence conditions:

  - \hypertarget{b-pentagon}{{\tt ($\beta$-\texttt{pentagon})}}
\begin{center}
\begin{tikzpicture}
    \node (E) at (0,-0.6) {\small $((f\,{_x\circ _{\underline{x}}}\,g)\,{_y\circ _{\underline{y}}}\,h)\,{_z\circ _{\underline{z}}}\,k$};
    \node (F) at (-3.7,-2.6) {\small $(f\,{_x\circ _{\underline{x}}}\,(g\,{_y\circ _{\underline{y}}}\,h))\,{_z\circ _{\underline{z}}}\,k$};
    \node (A) at (3.7,-2.6) {\small $(f\,{_x\circ _{\underline{x}}}\,g)\,{_y\circ _{\underline{y}}}\,(h\,{_z\circ _{\underline{z}}}\,k)$};
    \node (Asubt) at (-2.85,-5) {\small $f\,{_x\circ _{\underline{x}}}\,((g\,{_y\circ _{\underline{y}}}\,h)\,{_z\circ _{\underline{z}}}\,k)$};
    \node (P4) at (2.85,-5) {\small $f\,{_x\circ _{\underline{x}}}\,(g\,{_y\circ _{\underline{y}}}\,(h\,{_z\circ _{\underline{z}}}\,k))$};
    \draw[->] (E)--(F) node [midway,above,xshift=-1cm,yshift=-0.1cm] {\scriptsize   $\beta^{x,\underline{x};y,\underline{y}}_{f,g,h}\,{_z\circ _{\underline{z}}}\, 1_k$};
    \draw[->] (E)--(A) node [midway,above,xshift=0.7cm,yshift=-0.1cm] {\scriptsize     $\beta^{y,\underline{y};z,\underline{z}}_{f{_x\circ _{\underline{x}}}g,h,k}$};
 \draw[->] (F)--(Asubt) node [midway,left] {\scriptsize     $\beta^{x,\underline{x};z,\underline{z}}_{f,g {_y\circ _{\underline{y}}} h,k}$};
    \draw[->] (A)--(P4) node [midway,right]  {\scriptsize   $\beta^{x,\underline{x};y,\underline{y}}_{f,g,h{_z\circ _{\underline{z}}}k}$};
    \draw[->] (Asubt)--(P4) node [midway,above] {\scriptsize     $1_f\,{_x\circ _{\underline{x}}}\,\beta^{y,\underline{y};z,\underline{z}}_{g,h,k}$};
   \end{tikzpicture}
\end{center}

  - \hypertarget{bg-hexagon}{{\tt ($\beta \gamma$-\texttt{hexagon})}}
 
\begin{center}
\begin{tikzpicture}
    \node (E) at (0,0) {\small $(f\,{_x\circ _{\underline{x}}}\,g)\,{_y\circ _{\underline{y}}}\,h$};
    \node (G) at (4.2,0) {\small $f\,{_x\circ _{\underline{x}}}\,(g\,{_y\circ _{\underline{y}}}\,h)$};
    \node (F) at (8.4,0) {\small $(g\,{_y\circ _{\underline{y}}}\,h)\,{_{\underline{x}}\circ _{{x}}}\,f$};
    \node (A) at (0,-2) {\small $(g\,{_{\underline{x}}\circ _{{x}}}\,f)\,{_y\circ _{\underline{y}}}\,h$};
    \node (Asubt) at (4.2,-2) {\small $h\,{_{{\underline{y}}}\circ _y}\,(g\,{_{\underline{x}}\circ _{{x}}}\,f)$};
    \node (P4) at (8.4,-2) {\small $(h\,{_{\underline{y}}\circ _{{y}}}\,g) \,{_{\underline{x}}\circ _{{x}}}\,f$};
    \draw[->] (E)--(G) node [midway,above] {\scriptsize     $\beta^{x,\underline{x};y,\underline{y}}_{f,g,h}$};
    \draw[->] (G)--(F) node [midway,above] {\scriptsize    $\gamma^{x,\underline{x}}_{f,g{_y\circ _{\underline{y}}}h}$};
    \draw[->] (E)--(A) node [midway,left] {\scriptsize    $\gamma^{x,\underline{x}}_{f,g}\,{_y\circ _{\underline{y}}}\,{1_h}$};
 \draw[->] (F)--(P4) node [midway,right] {\scriptsize    $\gamma^{y,\underline{y}}_{g,h}\,{_{\underline{x}}\circ _{{x}}}\,{1_f}$};
    \draw[->] (A)--(Asubt) node [midway,above]  {\scriptsize    $\gamma^{y,\underline{y}}_{g{_{\underline{x}}\circ _{{x}}}f,h}$};
    \draw[->] (P4)--(Asubt) node [midway,above] {\scriptsize    $\beta^{\underline{y},y;\underline{x},x}_{h,g,f}$};
   \end{tikzpicture}
\end{center}
 
 - \hypertarget{bg-decagon}{{\tt ($\beta \gamma$-{\texttt{decagon}})}}
 
\begin{center}
\begin{tikzpicture}
    \node (E) at (-6.3,0) {\small $((f\,{_x\circ _{\underline{x}}}\,g)\,{_y\circ _{\underline{y}}}\,h)\,{_z\circ _{\underline{z}}}\,k$};
 \node (E1) at (-3.3,1.6) {\small $(h\,{_{\underline{y}}\circ _y}\,(f\,{_x\circ _{\underline{x}}}\,g))\,{_z\circ _{\underline{z}}}\,k$};
 \node (E2) at (2,1.6) {\small $h\,{_{\underline{y}}\circ _y}\,((f\,{_x\circ _{\underline{x}}}\,g)\,{_z\circ _{\underline{z}}}\,k)$};
    \node (G) at (5,0) {\small $((f\,{_x\circ _{\underline{x}}}\,g)\,{_z\circ _{\underline{z}}}\,k)\,{_y\circ _{\underline{y}}}\,h$};
    \node (F) at (5,-1.7) {\small $(f\,{_x\circ _{\underline{x}}}\,(g\,{_z\circ _{\underline{z}}}\,k))\,{_y\circ _{\underline{y}}}\,h$};
    \node (A) at (-6.3,-1.7) {\small $(f\,{_x\circ _{\underline{x}}}\,(g\,{_y\circ _{\underline{y}}}\,h))\,{_z\circ _{\underline{z}}}\,k$};
    \node (Asubt) at (-6.3,-3.5) {\small $f\,{_x\circ _{\underline{x}}}\,((g\,{_y\circ _{\underline{y}}}\,h)\,{_z\circ _{\underline{z}}}\,k)$};
  \node (Asubt1) at (-3.3,-5.1) {\small $f\,{_x\circ _{\underline{x}}}\,((h\,{_{\underline{y}}\circ _y}\,g)\,{_z\circ _{\underline{z}}}\,k)$};
  \node (Asubt2) at (2,-5.1) {\small $f\,{_x\circ _{\underline{x}}}\,(h\,{_{\underline{y}}\circ _y}\,(g\,{_z\circ _{\underline{z}}}\,k))$};
    \node (P4) at (5,-3.5) {\small $f\,{_x\circ _{\underline{x}}}\,((g\,{_z\circ _{\underline{z}}}\,k)\,{_y\circ _{\underline{y}}}\,h)$};
    \draw[->] (E)--(E1) node [midway,above,yshift=-0.15cm,xshift=-0.9cm] {\scriptsize  $\gamma^{{y,{\underline{y}}}}_{f{_x\circ _{\underline{x}}}g,h}\,{_z\circ _{\underline{z}}}\, 1_k$};
   \draw[->] (E1)--(E2) node [midway,above] {\scriptsize $\beta^{\underline{y},y;z,\underline{z}}_{h,f{_x\circ _{\underline{x}}}g,k}$};
   \draw[->] (E2)--(G) node [midway,above,xshift=0.8cm,yshift=-0.15cm] {\scriptsize  $\gamma^{\underline{y},y}_{h,(f{_x\circ _{\underline{x}}}g){_z\circ _{\underline{z}}}k}$};
    \draw[->] (G)--(F) node [midway,right] {\scriptsize   $\beta^{x,\underline{x};z,\underline{z}}_{f,g,k}\,{_y\circ _{\underline{y}}}\,1_h$};
    \draw[->] (E)--(A) node [midway,left] {\scriptsize  $\beta^{x,\underline{x};y,\underline{y}}_{f,g,h}\,{_z\circ _{\underline{z}}}\,1_k$};
 \draw[->] (F)--(P4) node [midway,right] {\scriptsize  $\beta^{x,\underline{x};y,\underline{y}}_{f,g{_z\circ _{\underline{z}}}k,h}$};
    \draw[->] (A)--(Asubt) node [midway,left]  {\scriptsize  $\beta^{x,\underline{x};z,\underline{z}}_{f,g{_y\circ _{\underline{y}}}h,k}$};
    \draw[->] (Asubt)--(Asubt1) node [midway,below,yshift=0.15cm,xshift=-1.4cm] {\scriptsize $1_f\,{_x\circ _{\underline{x}}}\,(\gamma^{y,\underline{y}}_{g,h}\,{_z\circ _{\underline{z}}}\, 1_k)$};
    \draw[->] (Asubt1)--(Asubt2) node [midway,above] {\scriptsize  $1_f\,{_x\circ _{\underline{x}}}\, \beta^{\underline{y},y;z,\underline{z}}_{h,g,k}$};
   \draw[->] (Asubt2)--(P4) node [midway,below,xshift=1.1cm,yshift=0.15cm] {\scriptsize  $1_f \,{_x\circ _{\underline{x}}}\,  \gamma^{\underline{y},y}_{h,g{_z\circ _{\underline{z}}} k}$};
   \end{tikzpicture}
\end{center}

- \hypertarget{g-involution}{{\tt ($\gamma$-{\texttt{involution}})}}   $\gamma^{\underline{x},x}_{g,f}\circ\gamma^{x,\underline{x}}_{f,g}={1}_{f{_x\circ _{\underline{x}}}g}$, where $1_{(-)}$ denotes the identity morphism for $(-)$, \\[0.2cm] as well as the following conditions which involve the action of ${\EuScript C}(\sigma)$, where $\sigma:Y\rightarrow X$, on the morphisms of ${\EuScript C}(X)$:\\[0.2cm]
- \hypertarget{bs}{{\tt ($\beta\sigma$)}} if the equality $((f\,{_x\circ _{\underline{x}}}\,\, g)\,{_y\circ _{\underline{y}}}\, h)^{\sigma}=(f^{\sigma_1}\,{_{x'}\circ _{\underline{x}'}}\,\, g^{\sigma_2})\,{_{y'}\circ _{\underline{y}'}}\, h^{\sigma_3}$ holds  by   \hyperlink{(EQ)}{\texttt{(EQ)}}, then $$(\beta^{x,\underline{x};y,\underline{y}}_{f,g,k})^{\sigma}=\beta^{x'\!, \underline{x}';y'\!,\underline{y}'}_{f^{\sigma_1},g^{\sigma_2},h^{\sigma_3}},$$ 
- \hypertarget{gs}{{\tt ($\gamma\sigma$)}} if the equality  $(f\,{_x\circ _{\underline{x}}}\,\, g)^{\sigma}=f^{\sigma_1}\,{_{x'}\circ _{\underline{x}'}}\,\, g^{\sigma_2}$ holds  by  \hyperlink{(EQ)}{\texttt{(EQ)}}, then $$(\gamma^{x,\underline{x}}_{f,g})^{\sigma}=\gamma^{x'\!,\underline{x}'}_{f^{\sigma_1},g^{\sigma_2}},$$

- \hypertarget{(EQ-mor)}{{\tt(EQ-mor)}} if the equality $(f\,{_x\circ _{\underline{x}}}\,\, g)^{\sigma}=f^{\sigma_1}\,{_{x'}\circ _{\underline{x}'}}\,\, g^{\sigma_2}$ holds  by  \hyperlink{(EQ)}{\texttt{(EQ)}}, and if $\varphi:f\rightarrow f'$ and  \\ 
\phantom{-\,{\tt(EQ-mor)}}    $\psi:g\rightarrow g'$, then $$(\varphi \,{_{x}\circ_{\underline{x}}} \,\psi)^{\sigma}=\varphi^{\sigma_1}\,{_{x'}\circ_{\underline{x}'}}\,\psi^{\sigma_2}.$$
\end{itemize}
We lift some of the vocabulary of Section \ref{s1} to the categorified setting: an object of ${\EuScript C}(X)$ is an operation, the 
elements of $X$ are its   entries; a categorified entries-only cyclic operad ${\EuScript C}$ is constant-free if ${\EuScript C}(\emptyset)={\EuScript C}(\{x\})=\emptyset$ (the empty category), for all singletons $\{x\}$.    \hfill$\square$
\end{defn}

\begin{rem}\label{oj}
The nodes of the diagrams of Definition \ref{cat} can be viewed as formal expressions built over operations $f,g,\dots$ and their entries $x,\underline{x},y,\underline{y},\dots$. For each diagram, the rules  for assembling correctly these expressions are determined by the predicate ``is an entry of''. 
For example, in \hyperlink{bg-decagon}{\em{\tt ($\beta\gamma$-${\tt{decagon}}$)}}, the legitimacy of all the nodes in the diagram witnesses that $x$ is  entry of $f$, $\underline{x}$, $y$ and $z$ are entries  of $g$, $\underline{y}$ is the entry of $h$ and $\underline{z}$ is the entry of $k$. From the tree-wise perspective, these data can be encoded by the unrooted tree
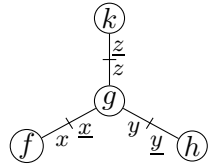
\begin{center}
\begin{tikzpicture}
 \node (f) [circle,fill=none,draw=black,minimum size=4mm,inner sep=0.1mm]  at (-1.1,1.1) {$f$};
\node (g) [circle,fill=none,draw=black,minimum size=4mm,inner sep=0.1mm]  at (0,1.7) {$g$};
\node (k) [circle,fill=none,draw=black,minimum size=4mm,inner sep=0.1mm]  at (0,2.8) {$k$};
\node (h) [circle,fill=none,draw=black,minimum size=4mm,inner sep=0.1mm]  at (1.1,1.1) {$h$};
\node (i1) [label={[xshift=0.3cm, yshift=0.4cm,]{\footnotesize $\underline{x}$}},label={[xshift=0.0cm, yshift=0.3cm,]{\footnotesize ${x}$}},circle,fill=none,draw=none,minimum size=0mm,inner sep=0mm]  at (-0.625,0.65) {};
\node (i2) [label={[xshift=-0.3cm, yshift=0.425cm,]{\footnotesize $y$}},label={[xshift=0cm, yshift=0.175cm,]{\footnotesize $\underline{y}$}},circle,fill=none,draw=none,minimum size=0mm,inner sep=0mm]  at (0.625,0.65) {};
\node (i3) [label={[xshift=0cm, yshift=-0.4cm]{\footnotesize $z$}},label={[xshift=0cm, yshift=-0.15cm,]{\footnotesize $\underline{z}$}},circle,fill=none,draw=none,minimum size=0mm,inner sep=0mm]  at (0.12,2.3) {};
\draw (f)--(g);
\draw (k)--(g);
\draw (g)--(h);
\draw (0.58,1.475)--(0.49,1.33);
\draw (-0.58,1.475)--(-0.49,1.33);
\draw (-0.09,2.25)--(0.09,2.25);
\end{tikzpicture}
\end{center}
This tree also illustrates that the morphism, say, $\beta^{\underline{x},x;y,\underline{y}}_{g,f,h}$ does not exist (for these particular $f$, $g$ and $h$), since its codomain is not well-formed. This exemplifies the difference with the setting of symmetric monoidal categories, where an instance of the sequential associator exists for  any (ordered) triple  of objects. The trees corresponding to \hyperlink{b-pentagon}{{\tt ($\beta$-${\tt{pentagon}}$)}},  \hyperlink{bg-hexagon}{{\tt ($\beta\gamma$-${\tt{hexagon}}$)}} and \hyperlink{g-involution}{{\tt ($\gamma$-${\tt{involution}}$)}} are 
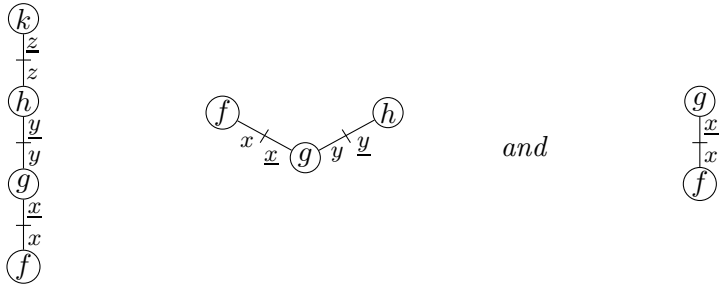
\begin{center}
\begin{tabular}{m{4.5cm} m{3.55cm} m{1.25cm} m{1.5cm}}
\begin{tikzpicture}
 \node (e) [circle,fill=none,draw=none,minimum size=4mm,inner sep=0.1mm]  at (-0.3,0) {};
 \node (f) [circle,fill=none,draw=black,minimum size=4mm,inner sep=0.1mm]  at (0,0) {$f$};
\node (g) [circle,fill=none,draw=black,minimum size=4mm,inner sep=0.1mm]  at (1.1,0) {$g$};
\node (k) [circle,fill=none,draw=black,minimum size=4mm,inner sep=0.1mm]  at (2.2,0) {$h$};
\node (h) [circle,fill=none,draw=black,minimum size=4mm,inner sep=0.1mm]  at (3.3,0) {$k$};
\node (i1) [label={[xshift=0.17cm, yshift=-0.1cm]{\footnotesize $\underline{x}$}},label={[xshift=-0.08cm, yshift=-0.035cm]{\footnotesize ${x}$}},circle,fill=none,draw=none,minimum size=0mm,inner sep=0mm]  at (0.5,0) {};
\node (i2) [label={[xshift=-0.08cm, yshift=-0.035cm]{\footnotesize $y$}},label={[xshift=0.17cm, yshift=-0.1cm]{\footnotesize $\underline{y}$}},circle,fill=none,draw=none,minimum size=0mm,inner sep=0mm]  at (1.6,0) {};
\node (i3) [label={[xshift=-0.08cm, yshift=-0.035cm]{\footnotesize $z$}},label={[xshift=0.17cm, yshift=-0.1cm]{\footnotesize $\underline{z}$}},circle,fill=none,draw=none,minimum size=0mm,inner sep=0mm]  at (2.7,0) {};
\draw (f)--(g);
\draw (g)--(k);
\draw (k)--(h);
\draw (0.55,-0.1)--(0.55,0.1);
\draw (1.65,-0.1)--(1.65,0.1);
\draw (2.75,-0.1)--(2.75,0.1);
\end{tikzpicture} &\begin{tikzpicture}
 \node (e) [circle,fill=none,draw=none,minimum size=4mm,inner sep=0.1mm]  at (-0.3,0) {};
 \node (f) [circle,fill=none,draw=black,minimum size=4mm,inner sep=0.1mm]  at (0,0) {$f$};
\node (g) [circle,fill=none,draw=black,minimum size=4mm,inner sep=0.1mm]  at (1.1,0) {$g$};
\node (k) [circle,fill=none,draw=black,minimum size=4mm,inner sep=0.1mm]  at (2.2,0) {$h$};
\node (i1) [label={[xshift=0.17cm, yshift=-0.1cm]{\footnotesize $\underline{x}$}},label={[xshift=-0.08cm, yshift=-0.035cm]{\footnotesize ${x}$}},circle,fill=none,draw=none,minimum size=0mm,inner sep=0mm]  at (0.5,0) {};
\node (i2) [label={[xshift=-0.08cm, yshift=-0.035cm]{\footnotesize $y$}},label={[xshift=0.17cm, yshift=-0.1cm]{\footnotesize $\underline{y}$}},circle,fill=none,draw=none,minimum size=0mm,inner sep=0mm]  at (1.6,0) {};
\draw (f)--(g);
\draw (g)--(k);
\draw (0.55,-0.1)--(0.55,0.1);
\draw (1.65,-0.1)--(1.65,0.1);
\end{tikzpicture}
& and & 
\begin{tikzpicture}
 \node (e) [circle,fill=none,draw=none,minimum size=4mm,inner sep=0.1mm]  at (-0.3,0) {};
 \node (f) [circle,fill=none,draw=black,minimum size=4mm,inner sep=0.1mm]  at (0,0) {$f$};
\node (g) [circle,fill=none,draw=black,minimum size=4mm,inner sep=0.1mm]  at (1.1,0) {$g$};
\node (i1) [label={[xshift=0.17cm, yshift=-0.1cm]{\footnotesize $\underline{x}$}},label={[xshift=-0.08cm, yshift=-0.035cm]{\footnotesize ${x}$}},circle,fill=none,draw=none,minimum size=0mm,inner sep=0mm]  at (0.5,0) {};
\draw (f)--(g);
\draw (0.55,-0.1)--(0.55,0.1);
\end{tikzpicture}
\end{tabular}
\end{center}
respectively.  In \S \ref{trrrre}, we shall introduce a formal tree-wise representation of  the operations of a categorified cyclic operads, based on this  intuition. Until then, we shall continue to omit the data about the ``origin of entries'' whenever possible.
\end{rem}

\begin{rem}\label{eqdis}
Observe that, for a categorified cyclic operad ${\EuScript C}$ and a finite set $X$, both the objects and the morphisms of ${\EuScript C}(X)$ enjoy  equivariance: at the level of objects, this is ensured by  \hyperlink{(EQ)}{\texttt{(EQ)}}, and at the level of morphisms, by \hyperlink{(EQ-mor)}{{\tt(EQ-mor)}}.
\end{rem}

In the remainder of the section, we shall work with a fixed categorified entries-only cyclic operad ${\EuScript C}$. In the remark  that follows, we list  the different kinds of equalities on objects and morphisms of ${\EuScript C}(X)$  which are  implicitly imposed by the structure of   ${\EuScript C}$.
 \begin{rem}\label{functoriality} For an arbitrary finite set $X$, the following equalities hold in  ${\EuScript C}(X)$:
\begin{itemize}
\item[1.] the categorical equations:
 $$\varphi\circ 1_f=\varphi=1_g\circ\varphi, \mbox{ for }   \varphi:f\rightarrow g, \quad \quad
 (\varphi\circ\phi)\circ\psi=\varphi\circ(\phi\circ\psi);$$
\item[2.] the equations imposed by the bifunctoriality of ${_x\circ_{\underline{x}}}$: 
$$1_{f} \, {_x\circ_{\underline{x}}}\, 1_g = 1_{f\, {_x\circ_{\underline{x}}}\,g},\quad \quad
(\varphi_2\circ\varphi_1)\, {_x\circ_{\underline{x}}}\, (\psi_2\circ\psi_1)=(\varphi_2 \, {_x\circ_{\underline{x}}}\,\psi_2)\circ(\varphi_1 \, {_x\circ_{\underline{x}}}\,\psi_1);$$
\item[3.] the naturality equations for $\beta$ and $\gamma$: 
\begin{itemize}
\item[a)] $\beta^{x,\underline{x};y,\underline{y}}_{{f}_2,{g}_2,{h}_2}\circ ((\varphi\,{_x\circ_{\underline{x}}}\,\phi)\,{_y\circ_{\underline{y}}}\,\psi)=(\varphi\,{_x\circ_{\underline{x}}}\,(\phi\,{_y\circ_{\underline{y}}}\,\psi)) \circ \beta^{x,\underline{x};y,\underline{y}}_{{f}_1,{g}_1,{h}_1}$, 
\item[b)] $\gamma^{x,y}_{f_2,g_2}\circ (\varphi\,{_x\circ_{y}}\,\phi)=(\phi\,{_y\circ_{x}}\,\varphi)\circ \gamma^{x,y}_{f_1,g_1}$; 
\end{itemize}
\item[4.] the equations imposed by the functoriality of ${\EuScript C}$: 
$${\EuScript C}(1_X)=1_{{\EuScript C}(X)},\quad\quad  (f^{\sigma})^{\tau}=f^{\sigma\circ\tau}, \quad\quad (\varphi^{\sigma})^{\tau}=\varphi^{\sigma\circ\tau};$$
\item[5.] the equations imposed by the functoriality of ${\EuScript C}(\sigma)$: 
$$1_f^{\sigma}=1_{f^{\sigma}}, \quad\quad (\varphi\circ\psi)^{\sigma}=\varphi^{\sigma}\circ\psi^{\sigma}.$$
\end{itemize}
\end{rem}
 
\subsubsection{Parallel associator in   ${\EuScript C}$}\label{pass}    
We   introduce an important abbreviation:  we define a natural isomorphism $\vartheta$, called {\em parallel associator}, by taking 
\begin{equation}\label{theta}\vartheta^{x,\underline{x};y,\underline{y}}_{f,g,h}\,= \,\gamma^{\underline{x},x}_{g,f{_y\circ _{\underline{y}}}h}\circ\beta^{\underline{x},x;y,\underline{y}}_{g,f,h}\circ(\gamma^{x,\underline{x}}_{f,g}\,{_y\circ _{\underline{y}}}\, 1_h)\,:\,(f\,{_x\circ _{\underline{x}}}\,g)\,{_y\circ _{\underline{y}}}\,h\,\longrightarrow\, (f\,{_y\circ _{\underline{y}}}\,h)\,{_x\circ _{\underline{x}}}\,g \end{equation}

\noindent
for its components.  Here are first observations about the natural isomorphism $\vartheta$.
 
\begin{rem}\label{thetaappears} The natural isomorphism $\vartheta$ appears  in \hyperlink{bg-hexagon}{\em{\tt ($\beta \gamma$-\texttt{hexagon})}} and \hyperlink{bg-decagon}{\em{\tt ($\beta \gamma$-{\texttt{decagon}})}}.
\begin{itemize}
\item[1.] An  isomorphism with the same source and target as $\vartheta^{x,\underline{x};y,\underline{y}}_{f,g,h}$ could be introduced as the composition 
$(\gamma^{\underline{y},y}_{h,f}\,{_x\circ _{\underline{x}}}\, 1_g)\circ({{\beta}^{\underline{y},y;x,\underline{x}}_{h,f,g}})^{-1}\circ \gamma^{y,\underline{y}}_{f{_x\circ _{\underline{x}}}g,h}$,
 which is as ``natural" as the composition \eqref{theta}. With this in mind,   \hyperlink{bg-hexagon}{\em{\tt ($\beta \gamma$-\texttt{hexagon})}} can be read as: the two  possible (and equally natural) definitions  of $\vartheta^{x,\underline{x};y,\underline{y}}_{f,g,h}$ are equal.
\item[2.]  By using  explicitly the abbreviations $\vartheta^{y,\underline{y};z,\underline{z}}_{f{_x\circ _{\underline{x}}}g,h,k}$ and $1_f\,{_x\circ _{\underline{x}}}\,\vartheta^{y,\underline{y};z,\underline{z}}_{g,h,k}$ for  the top  and  the bottom horizontal sequence of arrows of \hyperlink{bg-decagon}{\em{\tt ($\beta \gamma$-{\texttt{decagon}})}}, respectively,  \hyperlink{bg-decagon}{\em{\tt ($\beta \gamma$-{\texttt{decagon}})}} is turned into a hexagon (which corresponds to  the mixed hexagon of \cite{dp}). \end{itemize}
\end{rem}
\begin{lem}\label{thetainverse} The following equalities hold in ${\EuScript C}(X)$:
\begin{itemize}
\item[1.] \hypertarget{vt-involution}{{\em\texttt{($\vartheta$-involution)}}}\enspace $\vartheta^{y,\underline{y};x,\underline{x}}_{f,h,g}\circ\vartheta^{x,\underline{x};y,\underline{y}}_{f,g,h}=1_{(f{_x\circ _{\underline{x}}}\,g){_y\circ _{\underline{y}}}\,h}$,
\item[2.] \hypertarget{bvt-pentagon}{{\em\texttt{($\beta\vartheta$-pentagon)}}} \enspace $\vartheta^{y,\underline{y};x\underline{x}}_{f,h,g{_z\circ _{\underline{z}}}k}\circ \beta^{x,\underline{x};z,\underline{z}}_{f{_y\circ _{\underline{y}}}h,g,k}\circ (\vartheta^{x,\underline{x};y,\underline{y}}_{f,g,h}\,{_z\circ _{\underline{z}}}\, 1_k)=(\beta^{x,\underline{x};z,\underline{z}}_{f,g,k}\,{_y\circ _{\underline{y}}}\, 1_h) \circ \vartheta^{y,\underline{y};z,\underline{z}}_{f{_x\circ _{\underline{x}}}g,h,k}$,
\item[3.] \hypertarget{vt-hexagon}{{\em\texttt{($\vartheta$-hexagon)}}}  $$\vartheta^{x,\underline{x};y,\underline{y}}_{f\,{_z\circ _{\underline{z}}}\,k,g,h}\circ (\vartheta^{x,\underline{x};z,\underline{z}}_{f,g,k}\,{_y\circ _{\underline{y}}}\,1_h)\circ \vartheta^{y,\underline{y};z,\underline{z}}_{f{_x\circ _{\underline{x}}}g,h,k}=(\vartheta^{y,\underline{y};z,\underline{z}}_{f,h,k}\,{_x\circ _{\underline{x}}}\,1_g)\circ \vartheta^{x,\underline{x};z,\underline{z}}_{f{_y\circ _{\underline{y}}}h,g,k}\circ (\vartheta^{x,\underline{x};y,\underline{y}}_{f,g,h}\,{_z\circ _{\underline{z}}}\, 1_k).$$
\end{itemize}
\end{lem}

\begin{proof}  
\hyperlink{vt-involution}{{\texttt{($\vartheta$-involution)}}} follows by the commutation of the (outer part of) diagram 
\begin{center}
\resizebox{13cm}{!}{\begin{tikzpicture}
    \node (1) at (-1,0) {\small $(f\,{_x\circ _{\underline{x}}}\,g)\,{_y\circ _{\underline{y}}}\,h$};
    \node (2) at (3,2.2) {\small $(g\,{_{\underline{x}}\circ _{x}}\,f)\,{_y\circ _{\underline{y}}}\,h$};
    \node (3) at (7,2.2) {\small $g\,{_{\underline{x}}\circ _{x}}\,(f\,{_y\circ _{\underline{y}}}\,h)$};
    \node (4) at (11,0) {\small $(f\,{_y\circ _{\underline{y}}}\,h)\,{_{x}\circ _{\underline{x}}}\,g$};
    \node (11) at (3,-2.2) {\small $h \,{_{\underline{y}}\circ _{y}}\, (f\,{_x\circ _{\underline{x}}}\,g)$};
    \node (12) at (7,-2.2) {\small $(h \,{_{\underline{y}}\circ _{y}}\, f)\,{_x\circ _{\underline{x}}}\,g$};
 \node (m1) at (3,-0) {\small $(g\,{_{\underline{x}}\circ _{x}}\,f)\,{_y\circ _{\underline{y}}}\,h$};
 \node (m2) at (7,-0) {\small $g\,{_{\underline{x}}\circ _{x}}\,(f\,{_y\circ _{\underline{y}}}\,h)$};
    \draw[->] (1)--(2) node [midway,above,xshift=-0.5cm] {\scriptsize   $\gamma^{x,\underline{x}}_{f,g}\,{_y\circ _{\underline{y}}}\,1_h$};
    \draw[->] (2)--(3) node [midway,above] {\scriptsize   $\beta^{\underline{x},x;y,\underline{y}}_{g,f,h}$};
    \draw[->] (3)--(4) node [midway,above,xshift=0.2cm] {\scriptsize $\gamma^{\underline{x},x}_{g,f{_y\circ _{\underline{y}}}h}$};
 \draw[->] (4)--(m2) node [midway,above,yshift=-0.05cm] {\scriptsize  $\gamma^{x,\underline{x}}_{f{_y\circ _{\underline{y}}}h,g}$};
 \draw[->] (m2)--(m1) node [midway,above,yshift=-0.05cm] {\scriptsize  $({\beta^{\underline{x},x;y,\underline{y}}_{g,f,h}})^{-1}$};
 \draw[->] (m1)--(1) node [midway,above,yshift=-0.05cm] {\scriptsize  $\gamma^{\underline{x},x}_{g,f}\,{_y\circ _{\underline{y}}}\, 1_h$};
 \draw[->] (4)--(12) node [midway,below,xshift=0.6cm] {\scriptsize  $\gamma^{y,\underline{y}}_{f,h}\,{_y\circ _{\underline{y}}}\, {1_g}$};
    \draw[->] (12)--(11) node [midway,above]  {\scriptsize   $\beta^{\underline{y},y;x,\underline{x}}_{h,f,g}$};
    \draw[->] (11)--(1) node [midway,below,xshift=-0.2cm] {\scriptsize  $\gamma^{\underline{y},y}_{h,f{_{x}\circ _{\underline{x}}}g}$};
   \end{tikzpicture}}
\end{center}

\noindent
in which the upper hexagon commutes by \hyperlink{g-involution}{{\tt ($\gamma$-{\texttt{involution}})}}  and the lower hexagon commutes as an instance of \hyperlink{bg-hexagon}{{\tt ($\beta \gamma$-\texttt{hexagon})}}. \\[0.1cm]
\indent For \hyperlink{bvt-pentagon}{\texttt{($\beta\vartheta$-pentagon)}}, consider the diagram 
\begin{center}\resizebox{15.5cm}{!}{
\begin{tikzpicture}
    \node (E) at (0,0.2) {\small $((f\,{_x\circ _{\underline{x}}}\,g)\,{_y\circ _{\underline{y}}}\,h)\,{_z\circ _{\underline{z}}}\,k$};
  \node (Er) at (0,4.15) {$(h\,{_{\underline{y}}\circ _{y}}\,(f\,{_x\circ _{\underline{x}}}\,g))\,{_z\circ _{\underline{z}}}\,k$};
    \node (F) at (-3,-2) {\small $((f\,{_y\circ _{\underline{y}}}\,h)\,{_x\circ _{\underline{x}}}\,g)\,{_z\circ _{\underline{z}}}\,k$};
    \node (A) at (3,-2) {\small $((f\,{_x\circ _{\underline{x}}}\,g)\,{_z\circ _{\underline{z}}}\,k)\,{_y\circ _{\underline{y}}}\,h$};
    \node (Asubt) at (-2.5,-4.2) {\small $(f\,{_y\circ _{\underline{y}}}\,h)\,{_x\circ _{\underline{x}}}\,(g\,{_z\circ _{\underline{z}}}\,k)$};
    \node (P4) at (2.5,-4.2) {\small $(f\,{_x\circ _{\underline{x}}}\,(g\,{_z\circ _{\underline{z}}}\,k))\,{_y\circ _{\underline{y}}}\,h$};
   \node (A') at (8,-1) {\small $h\,{_{\underline{y}}\circ _{y}}\,((f\,{_x\circ _{\underline{x}}}\,g)\,{_z\circ _{\underline{z}}}\,k)$};
  \node (F') at (-8,-1) {\small $((h\,{_{\underline{y}}\circ _{y}}\,f)\,{_x\circ _{\underline{x}}}\,g)\,{_z\circ _{\underline{z}}}\,k$};
  \node (Asubt') at (-6.6,-6.7) {\small $(h\,{_{\underline{y}}\circ _{y}}\,f)\,{_x\circ _{\underline{x}}}\,(g\,{_z\circ _{\underline{z}}}\,k)$};
\node (P4') at (6.6,-6.7) {\small $h\,{_{\underline{y}}\circ _{y}}\,(f\,{_x\circ _{\underline{x}}}\,(g\,{_z\circ _{\underline{z}}}\,k))$};
\draw[->] (A)--(A') node [midway,below,xshift=0.4cm] {\scriptsize  $\gamma^{y,\underline{y}}_{(f{_x\circ _{\underline{x}}}g){_z\circ _{\underline{z}}}k,h}$};
 \draw[->] (F)--(F') node [midway,below,xshift=-0.4cm] {\scriptsize  $(\gamma^{y,\underline{y}}_{f,h}\,{_x\circ _{\underline{x}}}\,1_g)\,{_z\circ _{\underline{z}}}\,1_k$};
 \draw[->] (E)--(Er) node [midway,left,yshift=-0.25cm] {\scriptsize   $\gamma^{y,\underline{y}}_{f{_x\circ _{\underline{x}}}g,h}\,{_z\circ _{\underline{z}}}\, 1_k$};
    \draw[->] (E)--(F) node [midway,above,xshift=-1cm,yshift=-0.15cm] {\scriptsize   $\vartheta^{x,\underline{x};y,\underline{y}}_{f,g,h}\,{_z\circ _{\underline{z}}}\, 1_k$};
    \draw[->] (E)--(A) node [midway,above,xshift=0.7cm,yshift=-0.15cm] {\scriptsize    $\vartheta^{y,\underline{y};z,\underline{z}}_{f{_x\circ _{\underline{x}}}g,h,k}$};
 \draw[->] (F)--(Asubt) node [midway,left] {\scriptsize    $\beta^{x,\underline{x};z,\underline{z}}_{f{_y\circ _{\underline{y}}}h,g,k}$};
    \draw[->] (A)--(P4) node [midway,right]  {\scriptsize     $\beta^{x,\underline{x};z,\underline{z}}_{f,g,k}\,{_y\circ _{\underline{y}}}\, 1_h$};
    \draw[->] (Asubt)--(P4) node [midway,above] {\scriptsize    $\vartheta^{y,\underline{y};x\underline{x}}_{f,h,g{_z\circ _{\underline{z}}}k}$};
 \draw[->] (Asubt)--(Asubt') node [midway,left,yshift=0.2cm,xshift=0.1cm] {\scriptsize  $\gamma^{y,\underline{y}}_{h,f} \,{_x\circ _{\underline{x}}}\, 1_{g{_z\circ _{\underline{z}}}k}$};
 \draw[->] (F')--(Asubt') node [midway,left] {\scriptsize  $\beta^{x,\underline{x};z,\underline{z}}_{h{_{\underline{y}}\circ _{y}}f,g,k}$};
\draw[->] (P4')--(P4) node [midway,right,yshift=0.08cm,xshift=-0.1cm] {\scriptsize  $\gamma^{\underline{y},y}_{h,f{_x\circ _{\underline{x}}}(g{_z\circ _{\underline{z}}}k)}$};
\draw[->] (Asubt')--(P4') node [midway,above] {\scriptsize  $\beta^{\underline{y},y;x,\underline{x}}_{h,f,g{_z\circ _{\underline{z}}}k}$};
\draw[->] (A')--(P4') node [midway,right] {\scriptsize  $1_h\,{_{\underline{y}}\circ _{y}}\,\beta^{x,\underline{x};z,\underline{z}}_{f,g,k}$};
 \draw[->] (Er)--(F') node [midway,left,xshift=-0.15cm] {\scriptsize  $\beta^{\underline{y},y;x,\underline{x}}_{h,f,g}\,{_z\circ _{\underline{z}}}\, 1_k$};
 \draw[->] (Er)--(A') node [midway,right,xshift=0.2cm] {\scriptsize  $\beta^{\underline{y},y;z,\underline{z}}_{h,f{_x\circ _{\underline{x}}}g,k}$};
   \end{tikzpicture}}
\end{center}
 whose ``inner" pentagon is \hyperlink{bvt-pentagon}{\texttt{($\beta\vartheta$-pentagon)}} and whose ``outer" pentagon commutes as an instance of  \hyperlink{b-pentagon}{{\tt ($\beta$-\texttt{pentagon})}}. The claim follows by the  commutations of all the diagrams ``between" the two pentagons (two naturality squares for $\beta$ and three squares expressing the definition of $\vartheta$).\\[0.1cm]
\indent We use an analogous diagram for proving \hyperlink{vt-hexagon}{{\texttt{($\vartheta$-hexagon)}}}. The ``inner" hexagon in the diagram
\begin{center}\resizebox{16.2cm}{!}{
\begin{tikzpicture}
    \node (E) at (-2.5,0) {\small $((f\,{_x\circ _{\underline{x}}}\,g)\,{_y\circ _{\underline{y}}}\,h)\,{_z\circ _{\underline{z}}}\,k$};
   \node (El) at (-6.5,2.4) {\small $((g\,{_{\underline{x}}\circ _{x}}\,f)\,{_y\circ _{\underline{y}}}\,h)\,{_z\circ _{\underline{z}}}\,k$};
   \node (Elh) at (6.5,2.4) {\small $(g\,{_{\underline{x}}\circ _{x}}\,(f\,{_y\circ _{\underline{y}}}\,h))\,{_z\circ _{\underline{z}}}\,k$};
    \node (Al) at (-8.2,-1.7) {\small $((g\,{_{\underline{x}}\circ _{x}}\,f)\,{_z\circ _{\underline{z}}}\,k)\,{_y\circ _{\underline{y}}}\,h$};
  \node (Ar2) at (8.2,-1.7) {\small $g\,{_{\underline{x}}\circ _{x}}\,((f\,{_y\circ _{\underline{y}}}\,h)\,{_z\circ _{\underline{z}}}\,k)$};
    \node (G) at (2.5,0) {\small $((f\,{_y\circ _{\underline{y}}}\,h)\,{_x\circ _{\underline{x}}}\,g)\,{_z\circ _{\underline{z}}}\,k$};
    \node (F) at (3,-1.7) {\small $((f\,{_y\circ _{\underline{y}}}\,h)\,{_z\circ _{\underline{z}}}\,k)\,{_x\circ _{\underline{x}}}\,g$};
    \node (A) at (-3,-1.7) {\small $((f\,{_x\circ _{\underline{x}}}\,g)\,{_z\circ _{\underline{z}}}\,k)\,{_y\circ _{\underline{y}}}\,h$};
    \node (Asubt) at (-2.5,-3.5) {\small $((f\,{_z\circ _{\underline{z}}}\,k)\,{_x\circ _{\underline{x}}}\,g)\,{_y\circ _{\underline{y}}}\,h$};
    \node (Asubtl) at (-6.5,-5.9) {\small $(g\,{_{\underline{x}}\circ _{x}}\,(f\,{_z\circ _{\underline{z}}}\,k))\,{_y\circ _{\underline{y}}}\,h$};
    \node (P4) at (2.5,-3.5) {\small $((f\,{_z\circ _{\underline{z}}}\,k)\,{_y\circ _{\underline{y}}}\,h)\,{_x\circ _{\underline{x}}}\,g$};
 \node (P4r) at (6.5,-5.9) {\small $g\,{_{\underline{x}}\circ _{x}}\,((f\,{_z\circ _{\underline{z}}}\,k)\,{_y\circ _{\underline{y}}}\,h)$};
    \draw[->] (E)--(G) node [midway,above] {\scriptsize     $\vartheta^{x,\underline{x};y,\underline{y}}_{f,g,h}\,{_z\circ _{\underline{z}}}\, 1_k$};
\draw[->] (E)--(El) node [midway,above,xshift=1.2cm,yshift=-0.1cm] {\scriptsize   $(\gamma^{x,\underline{x}}_{f,g}\,{_y\circ _{\underline{y}}}\, 1_h)\,{_z\circ _{\underline{z}}}\,  1_k$};
\draw[->] (El)--(Al) node [midway,left] {\scriptsize   $\vartheta^{y,\underline{y};z,\underline{z}}_{g{_{\underline{x}}\circ _{x}}f,h,k}$};
\draw[->] (El)--(Elh) node [midway,above] {\scriptsize   $\beta^{\underline{x},x;y,\underline{y}}_{g,f,h}\,{_z\circ _{\underline{z}}}\, 1_k$};
\draw[->] (Elh)--(G) node [midway,above,xshift=-0.95cm,yshift=-0.1cm] {\scriptsize   $\gamma^{\underline{x},x}_{g,f{_y\circ _{\underline{y}}}h}\,{_z\circ _{\underline{z}}}\,1_k$};
\draw[->] (Elh)--(Ar2) node [midway,right] {\scriptsize   $\beta^{\underline{x},x;z,\underline{z}}_{g,f{_y\circ _{\underline{y}}}h,k}$};
\draw[->] (A)--(Al) node [midway,above,sloped] {\scriptsize   $(\gamma^{x,\underline{x}}_{f,g}\,{_z\circ _{\underline{z}}}\,  1_k)\,{_y\circ _{\underline{y}}}\, 1_h$};
\draw[->] (Ar2)--(P4r) node [midway,right] {\scriptsize   $1_g\,{_{\underline{x}}\circ _{{x}}}\, \vartheta^{y,\underline{y};z,\underline{z}}_{f,h,k}$};
    \draw[->] (G)--(F) node [midway,right,yshift=-0.1cm] {\scriptsize      $\vartheta^{x,\underline{x};z,\underline{z}}_{f{_y\circ _{\underline{y}}}h,g,k}$};
    \draw[->] (E)--(A) node [midway,left,yshift=-0.1cm] {\scriptsize     $\vartheta^{y,\underline{y};z,\underline{z}}_{f{_x\circ _{\underline{x}}}g,h,k}$};
 \draw[->] (F)--(Ar2) node [midway,above] {\scriptsize  $\gamma^{x,\underline{x}}_{(f{_y\circ _{\underline{y}}}h){_z\circ _{\underline{z}}}k,g}$};
 \draw[->] (F)--(P4) node [midway,right,yshift=-0.1cm] {\scriptsize     $\vartheta^{y,\underline{y};z,\underline{z}}_{f,h,k}\,{_x\circ _{\underline{x}}}\,1_g$};
    \draw[->] (A)--(Asubt) node [midway,left,yshift=-0.1cm]  {\scriptsize     $\vartheta^{x,\underline{x};z,\underline{z}}_{f,g,k}\,{_y\circ _{\underline{y}}}\,1_h$};
    \draw[->] (Asubt)--(P4) node [midway,above] {\scriptsize    $\vartheta^{x,\underline{x};y,\underline{y}}_{f\,{_z\circ _{\underline{z}}}\,k,g,h}$};
    \draw[->] (Asubtl)--(Asubt) node [midway,above,xshift=-0.7cm] {\scriptsize  $\gamma^{\underline{x},x}_{g,f{_z\circ _{\underline{z}}}k}\,\,{_y\circ _{\underline{y}}}\, 1_h$};
   \draw[->] (Al)--(Asubtl) node [midway,left] {\scriptsize  $\beta^{\underline{x},x;z,\underline{z}}_{g,f,k}\,{_y\circ _{\underline{y}}}\,  1_h$};
    \draw[->] (P4)--(P4r) node [midway,above,xshift=0.6cm] {\scriptsize     $\gamma^{x,\underline{x}}_{(f{_z\circ _{\underline{z}}}k){_y\circ _{\underline{y}}}h,g}$};
 \draw[->] (Asubtl)--(P4r) node [midway,above] {\scriptsize    $\beta^{\underline{x},x;y,\underline{y}}_{g,f{_z\circ _{\underline{z}}}k,h}$};
   \end{tikzpicture}}
\end{center}  is  \hyperlink{vt-hexagon}{{\texttt{($\vartheta$-hexagon)}}}, the ``outer" hexagon is obtained from \hyperlink{bg-decagon}{\em{\tt ($\beta \gamma$-{\texttt{decagon}})}} by using explicitly the abbreviations for the parallel associator (see Remark \ref{thetaappears}(2)), and the claim follows by  the commutations of all the diagrams ``between" the two hexagons (the four naturality squares for $\vartheta$ and two squares which express  the definition of $\vartheta$).
\end{proof}
 \subsection{Canonical diagrams and the coherence theorem}\label{blibli}
The coherence theorem that we shall prove has the form:   {\em all diagrams of canonical arrows commute in ${\EuScript C}(X)$.}
In order to formulate it rigorously, we first specify  what  a diagram of canonical arrows  is exactly, by means of a syntax. Denoting with $\underline{\EuScript C}$   the underlying ({\bf Set}-valued) functor of ${\EuScript C}$, we essentially describe  the free categorified entries-only cyclic operad built over $\underline{\EuScript C}$.
 However, since the purpose of the syntax is solely to distinguish the canonical arrows of ${\EuScript C}(X)$, the formalism will be left without any equations.
 \subsubsection{The syntax ${\tt Free}_{\underline{\EuScript C}}$}
 Let \begin{equation}\label{pp}P_{\underline{\EuScript C}}=\{a\,|\, a\in\underline{\EuScript C}(X) \mbox{ for some finite set } X \}\end{equation} be the collection of {\em parameters} of $\underline{\EuScript C}$,  let $V$ range over {\em variables}   $x,y,z,\underline{x},\underline{y},\underline{z},\dots $, and let  $\Sigma$ range over 
  bijections of finite sets. \\[0.1cm]
\indent Our syntax $\tt{Free}_{\underline{\EuScript C}}$ 
contains two kinds of typed expressions, the {\em object terms} and the {\em arrow terms}  (as all the other formal systems that we shall introduce in the remaining of the section). \\[0.1cm]
\indent The syntax of object terms is obtained from raw (i.e., not  yet typed) object terms
 \begin{center}\vspace{0.1cm}\mybox{
$\displaystyle {\cal W}::= \underline{a} \,\,|\,\, ({\cal W}{_x\bo_y}{\cal W})\,\,|\,\, {\cal W}^{\sigma}$ }\vspace{0.1cm}
\end{center}
where $a\in P_{\underline{\EuScript C}}$, $x,y\in V$, and $\sigma\in\Sigma$, by typing them as ${\cal W}:X$, where $X$ ranges over finite sets. The assignment of types is done by the following  rules:  
 \begin{center} 
\mybox{  $\displaystyle \frac{a\in {\EuScript C}(X)}{\underline{a}:X}\quad \quad\frac{{\cal W}_1\!:\!X\quad {\cal W}_2\!:\!Y\quad x\in X\enspace y\in Y\quad X\backslash\{x\}\cap Y\backslash\{y\}=\emptyset}{{\cal W}_1\,{_x\bo_y}\,{\cal W}_2:X\backslash\{x\}\cup Y\backslash\{y\}}\quad \quad\frac{{\cal W}\!:\!X\quad \sigma:Y\rightarrow X}{{\cal W}^{\sigma}:Y}$  } \end{center} \vspace{0.02cm}

\begin{rem}
The notation $_x\bo_y$ (rather then $_x\circ_y$) for the syntax of partial composition operations  is chosen merely to avoid confusion with the symbol $\circ$, used to denote the (usual) composition of morphisms in a category.
\end{rem}
 To the syntax of object terms we   add the syntax  of   arrow terms, obtained from raw arrow terms   
 \begin{center}\vspace{0.1cm}\mybox{
$\displaystyle \Phi::=\left\{\begin{array}{l}
1_{\cal W}\,\,|\,\, \beta^{x,\underline{x};y,\underline{y}}_{{\cal W}_1,{\cal W}_2,{\cal W}_3} \,\,|\,\, {\beta_{{\cal W}_1,{\cal W}_2,{\cal W}_3}^{x,\underline{x};y,\underline{y}\enspace ^{_{-1}}}}\,\,|\,\, \gamma^{x,y}_{{\cal W}_1,{\cal W}_2}\\[0.25cm]  
 {\varepsilon_1}_{\underline{a}}^{\sigma}\,\,|\,\, {\varepsilon_1}_{\underline{a}}^{\sigma \, ^ {_{-1}}}\,\,|\,\, {{\varepsilon}_2}_{\cal W} \,\,|\,\, {{\varepsilon}_2}^{_{-1}}_{\cal W} \,\,|\,\, {\varepsilon_3}_{\cal W}^{\sigma,\tau}\,\,|\,\,{\varepsilon_3}_{\cal W}^{\sigma,\tau\,^{_{-1}}}\,\,|\,\,{\varepsilon_4}^{x,y;x',y'}_{{\cal W}_1,{\cal W}_2;\sigma}\,\,|\,\, {\varepsilon_4}^{x,y;x',y'\,^{_{-1}}}_{{\cal W}_1,{\cal W}_2;\sigma}\\[0.25cm]
\Phi\circ\Phi\,\,|\,\, \Phi \,{_x\bo_y}\,  \Phi\,\,|\,\, \Phi^{\sigma},\end{array} \right. $}\vspace{0.1cm}\end{center} by assigning them types in the form of ordered pairs $({\cal W}_1,{\cal W}_2)$ of object terms, denoted by ${\cal W}_1\rightarrow {\cal W}_2$,   as follows: 
\begin{center}\mybox{\begin{tabular}{c}
{\small $\displaystyle \frac{}{1_{\cal W}:{\cal W}\rightarrow {\cal W}}$} \\[0.6cm] {\small $\displaystyle\frac{}{\beta^{x,\underline{x};y,\underline{y}}_{{\cal W}_1,{\cal W}_2,{\cal W}_3}:({\cal W}_1 {_x\bo_{\underline{x}}}{\cal W}_2){_y\bo_{\underline{y}}}{\cal W}_3\rightarrow {\cal W}_1 {_x\bo_{\underline{x}}}({\cal W}_2{_y\bo_{\underline{y}}}{\cal W}_3)}$} \\[0.6cm]

{\small $\displaystyle\frac{}{{\beta_{{\cal W}_1,{\cal W}_2,{\cal W}_3}^{x,\underline{x};y,\underline{y}\enspace ^{_{-1}}}}:{\cal W}_1 {_x\bo_{\underline{x}}}({\cal W}_2{_y\bo_{\underline{y}}}{\cal W}_3)\rightarrow ({\cal W}_1 {_x\bo_{\underline{x}}}{\cal W}_2){_y\bo_{\underline{y}}}{\cal W}_3}$}\\[0.6cm]

{\small $\displaystyle\frac{}{\gamma^{x,y}_{{\cal W}_1,{\cal W}_2}:{\cal W}_1 \,{_x\bo_y}\,{\cal W}_2\rightarrow{\cal W}_2\,{_y\bo_x}\,{\cal W}_1}$}\\[0.6cm]

{\small $\displaystyle\frac{}{{\varepsilon_1}^{\sigma}_{\underline{a}}:\underline{a}^{\sigma}\rightarrow \underline{a^{\sigma}}}\quad\quad \frac{}{{\varepsilon_1}^{\sigma\, ^ {_{-1}}}_{\underline{a}}:\underline{a^{\sigma}}\rightarrow \underline{a}^{\sigma}}$} \quad\quad {\small  $\displaystyle\frac{}{{{\varepsilon}_2}_{\cal W}:{\cal W}^{{\it id}_X}\rightarrow {\cal W}}\quad\quad \frac{}{{{\varepsilon}_2}^{_{-1}}_{\cal W}:{\cal W}\rightarrow {\cal W}^{{\it id}_X}}$}\\[0.6cm]

{\small $\displaystyle\frac{}{{\varepsilon_3}_{\cal W}^{\sigma,\tau}:({\cal W}^{\sigma})^{\tau}\rightarrow {\cal W}^{\sigma\circ\tau}}\quad\quad \frac{}{{\varepsilon_3}_{\cal W}^{\sigma,\tau \, ^ {_{-1}}}:{\cal W}^{\sigma\circ\tau}\rightarrow ({\cal W}^{\sigma})^{\tau}}$}\\[0.65cm]

{\small $\displaystyle\frac{\substack{\sigma\,:\,Z\,\rightarrow\, X\backslash\{x\}\cup Y\backslash\{y\}  \\[0.05cm] \sigma_1:\,\sigma^{-1}[X\backslash\{x\}]\cup \{x'\}\,\rightarrow\, X\enspace\enspace  \sigma_1|^{X\backslash\{x\}}=\sigma|^{X\backslash\{x\}}\enspace\enspace \sigma_1(x')=x \\[0.05cm] \sigma_2\,:\,\sigma^{-1}[Y\backslash\{y\}]\cup \{y'\}\,\rightarrow\, Y\enspace\enspace  \sigma_2|^{Y\backslash\{y\}}=\sigma|^{Y\backslash\{y\}}\enspace\enspace \sigma_2(y')=y\\[0.1cm]} }{{\varepsilon_4}^{x,y;x',y'}_{{\cal W}_1,{\cal W}_2;\sigma}:({\cal W}_1\,{_x\bo_y}\,{\cal W}_2)^{\sigma}\rightarrow {\cal W}_1^{\sigma_1} \,{_{x'}\bo_{y'}}\,{\cal W}_2^{\sigma_2}}$}\\[0.65cm]

{\small $\displaystyle\frac{\substack{ \sigma_1\,:\,X'\,\rightarrow\, X \enspace\enspace \sigma_1(x')=x \\[0.05cm] \sigma_2\,:\,Y'\,\rightarrow\, Y \enspace\enspace \sigma_2(y')=y\\[0.05cm] \sigma\,:\,X'\backslash\{x'\}\cup Y'\backslash\{y'\}\,\rightarrow\, X\backslash\{x\}\cup Y\backslash\{y\} \enspace \sigma=\sigma_1|^{X'\backslash\{x'\}}\cup \sigma_2|^{Y'\backslash\{y'\}}  \\[0.1cm]} }{{\varepsilon_4}^{x,y;x',y' \, ^ {_{-1}}}_{{\cal W}_1,{\cal W}_2;\sigma}:{\cal W}_1^{\sigma_1} \,{_{x'}\bo_{y'}}\,{\cal W}_2^{\sigma_2}\rightarrow ({\cal W}_1\,{_x\bo_y}\,{\cal W}_2)^{\sigma}}$}\\[0.65cm]

{\small $\displaystyle\frac{\Phi_1:{\cal W}_1\rightarrow {\cal W}_2\quad \Phi_2:{\cal W}_2\rightarrow {\cal W}_3}{\Phi_2\circ\Phi_1:{\cal W}_1\rightarrow {\cal W}_3}$ \quad\quad $\displaystyle\frac{\Phi_1:{\cal W}_1\rightarrow {\cal W}'_1\quad \Phi_2:{\cal W}_2\rightarrow {\cal W}'_2}{\Phi_1 \, {_x\bo_y}\, \Phi_2:{\cal W}_1\, {_x\bo_y}\,{\cal W}_2\rightarrow {\cal W}'_1 \,{_x\bo_y}\,{\cal W}'_2}$ }\\[0.65cm]

{\small   $\displaystyle\frac{\Phi:{\cal W}_1\rightarrow {\cal W}_2 \quad  \sigma:Y\rightarrow X}{\Phi^{\sigma}:{\cal W}_1^{\sigma}\rightarrow {\cal W}_2^{\sigma}}$}
\end{tabular}
 }
\end{center}
\noindent where it is also (implicitly) assumed that all the object terms that appear in the types of the arrow terms are well-formed. Given an arrow term $\Phi:{\cal U}\rightarrow {\cal V}$, we   call the object term ${\cal U}$ {\em the source of } $\Phi$ and the object term ${\cal V}$ {\em the target of } $\Phi$.  
\begin{rem}
Observe that, for all well-typed arrow terms $\Phi:{\cal U}\rightarrow {\cal V}$ of $ {\tt{Free}}_{\underline{\EuScript C}}$, the object terms ${\cal U}$ and ${\cal V}$ have the same type.\end{rem}
\begin{rem}
\label{indices} Notice that the type of an  arrow term $\Phi$ of ${\tt Free}_{\underline{\EuScript C}}$ is determined completely by $\Phi$ only, that is, by the {main symbol of} $\Phi$, the {indices of} $\Phi$ and their {order of appearance in} $\Phi$. For example, if $\Phi=\gamma^{x,y}_{{\cal W}_1,{\cal W}_2}$, then the main symbol $\gamma$ of $\Phi$, together the indices ${\cal W}_1,{\cal W}_2$ (the subscript of $\gamma$) and $x,y$ (the superscript of  $\gamma$), appearing in this particular order in the specification of $\Phi$, unambiguously determine ${\cal W}_1 \, _x\bo_y\, {\cal W}_2$ as the source of $\Phi$. This allows us to speak about the source and the target of $\Phi$ without having to explicitly declare what those object terms are. This property of arrow terms will hold for all syntactic systems that we shall consider in the remaining of the paper. 
\end{rem}
\indent The collection of object terms of type $X$, together with the collection of arrow terms  whose source and target have type $X$, will be denoted by    $ {\tt{Free}}_{\underline{\EuScript C}}(X)$. We shall use the same convention for the other syntaxes that will be introduced in the sequel.

\subsubsection{The interpretation of ${\tt{Free}}_{\underline{\EuScript C}}$ in  ${\EuScript C}$}
The semantics of  $\tt{Free}_{\underline{\EuScript C}}$ in  ${\EuScript C}$ is what distinguishes canonical arrows (or $\beta\gamma\sigma$-arrows) of ${\EuScript C}(X)$: they will be precisely the interpretations of the arrow terms of ${\tt{Free}_{\underline{\EuScript C}}}(X)$. \\[0.1cm]
\indent  The interpretation function $$[[-]]_X:{\tt Free}_{\underline{\EuScript C}}(X)\rightarrow {\EuScript C}(X)$$ is defined recursively in the obvious way: it maps $\underline{a}$ to $a$, ${\cal W}_1\,{_x\bo_y}\,{\cal W}_2$ to $[[{\cal W}_1]]_{X_1}\,{_x\circ_y}\,[[{\cal W}_2]]_{X_2}$, etc., and the $\varepsilon_i$'s and their inverses to ${\it id}$ (reflecting the fact that the axiom \hyperlink{(EQ)}{\texttt{(EQ)}} remains strict in the transition from Definition \ref{entriesonly} to Definition \ref{cat}).

\begin{lem}\label{int1} The interpretation function $[[-]]_X:{\tt Free}_{\underline{\EuScript C}}(X)\rightarrow {\EuScript C}(X)$ is well-defined, in the sense that, for an arrow term $\Phi:{\cal U}\rightarrow {\cal V}$ of ${\tt Free}_{\underline{\EuScript C}}(X)$, we have that  $[[\Phi]]_X$ is a morphism from $[[{\cal U}]]_X$ to  $[[{\cal V}]]_X$.
\end{lem}
\begin{proof} The claim holds thanks to the  axiom \hyperlink{(EQ)}{\texttt{(EQ)}} for ${\EuScript C}$.
\end{proof}
A canonical diagram in ${\EuScript C}(X)$ is a pair of parallel morphisms (i.e., morphisms that share the same source and  target) arising as interpretations of two arrow terms of the same type of ${\tt{Free}}_{\underline{\EuScript C}}$.  
\subsubsection{The coherence theorem}\label{coht}
We can now state precisely the coherence theorem for ${\EuScript C}$.  
\begin{thm}[Coherence Theorem] \label{coherence-theorem}
For any finite set $X$ and for any pair of arrow terms  $\Phi,\Psi:{\cal W}_1\rightarrow {\cal W}_2$ of the same type in  ${\tt Free}_{\underline{\EuScript C}}(X)$, the equality  $[[\Phi]]_X=[[\Psi]]_X$ holds in ${\EuScript C}(X)$.
\end{thm}

We prove this theorem in the remaining of Section 2. We make three faithful reductions, each restricting the coherence problem to a smaller class of diagrams, in such a way that the coherence problem is ultimately reduced to the coherence of weak Cat-operads of \cite{dp}.
\subsection{The first reduction: getting rid of symmetries}\label{frst}
Our first goal is to cut down the coherence problem of    ${\EuScript C}$ to the problem of commutation of all {\em diagrams of} $\beta \gamma$-{\em arrows} of ${\EuScript C}$. We introduce first the syntax of these diagrams.
\subsubsection{The syntax $\underline{\tt{Free}}_{\underline{\EuScript C}}$}
The syntax $\underline{\tt{Free}}_{\underline{\EuScript C}}$ is obtained by removing the  term constructor $(-)^{\sigma}$ from the list of raw object and raw arrow terms of ${\tt {Free}}_{\underline{\EuScript C}}$, as well as 
 the $\varepsilon_i$'s and their inverses, and by restricting the typing system of ${\tt {Free}}_{\underline{\EuScript C}}$ accordingly.
We call the syntax $\underline{\tt{Free}}_{\underline{\EuScript C}}$ the $\beta \gamma$-{\em reduction} of $\tt{Free}_{\underline{\EuScript C}}$. In passing from $\tt{Free}_{\underline{\EuScript C}}$ to $\underline{\tt{Free}}_{\underline{\EuScript C}}$, we shall switch from 
calligraphic to italic letters for object terms, and from uppercase Greek to lowercase Greek letters for arrow terms.
 
\begin{rem} Notice that, in the syntax $\underline{\tt{Free}}_{\underline{\EuScript C}}$, for an   arrow term $\varphi:U\rightarrow V$,  the parameters and variables that appear in $U$ are {\em exactly}  the parameters and variables that appear in $V$.
 
\end{rem}
\subsubsection{The interpretation of $\underline{\tt Free}_{\underline{\EuScript C}}$ in   ${\EuScript C}$}
The semantics of $\underline{\tt Free}_{\underline{\EuScript C}}$ in   ${\EuScript C}$ is what distinguishes   $\beta \gamma$-{\em arrows} of ${\EuScript C}(X)$ among $\beta \gamma\sigma$-{arrows}   of ${\EuScript C}(X)$. The interpretation function $$[-]_{X}:\underline{\tt Free}_{\underline{\EuScript C}}(X)\rightarrow {\EuScript C}(X)$$ is defined simply as the appropriate restriction of the interpretation function $[[-]]_{X}$.

\begin{lem}\label{wd2}
The interpretation function $[-]_{X}:\underline{\tt Free}_{\underline{\EuScript C}}(X)\rightarrow {\EuScript C}(X)$ is well-defined, in the sense that, for an arrow term $\varphi:U\rightarrow W$ of $\underline{\tt Free}_{\underline{\EuScript C}}(X)$, we have that $[\varphi]_X$ is a morphism from $[U]_X$ to $[V]_{X}$.
 \end{lem}

\subsubsection{An auxiliary typing system for the raw arrow terms of $\underline{\tt Free}_{\underline{\EuScript C}}$}\label{aux}
As will become clear in Remark \ref{why-auxiliary-type} below, in order to pass from ${\tt Free}_{\underline{\EuScript C}}$ to $\underline{\tt Free}_{\underline{\EuScript C}}$, we shall need  an intermediate, slightly  more permissive typing system {\uwave{\tt Free}}$_{{\underline{\EuScript C}}}$  for the raw arrow terms of $\underline{\tt Free}_{\underline{\EuScript C}}$. 
It is the same as $\underline{\tt Free}_{\underline{\EuScript C}}$, except for the composition rule for arrow terms, where we add a degree of freedom by allowing the composition not only along the {\em same} typed object term, but also along  $\alpha${\em -equivalent ones}\footnote{The terminology ``$\alpha$-equivalence'' comes from $\lambda$-calculus, where it is used to formalize  the intuition that the names of bound variables do  not matter: the
function $f(x)$ is the same as the function $f(y)$.}. Intuitively, two object terms are said to be $\alpha$-equivalent if they have the same type, say $X$, and one can be obtained from the other one only by renaming the variables that do not appear in $X$ (i.e. that get ``used up''  in the construction of  those object terms by the syntax encoding partial composition operations). Thereby, $\alpha$-equivalence formalizes the instances of the axiom \hyperlink{(EQ)}{\texttt{(EQ)}} (see Remark \ref{eqdis}) for which the ``outer'' bijection is the identity.

\indent In order to rigorously define $\alpha$-equivalence on object terms of $\underline{\tt Free}_{\underline{\EuScript C}}$, we  
introduce some terminology. For a parameter $a\in \underline{\EuScript C}(X)$ of $P_{\underline{\EuScript C}}$, we say that $X$ is the set of {\em free variables} of $a$, and we write ${\it FV}(a)=X$.  For an   object term $W:Y$, we shall denote with $P_{\underline{\EuScript C}}(W)$ the set of all parameters of $P_{\underline{\EuScript C}}$ that appear in $W$.  
The $\alpha$-equivalence  $\equiv$ on  object terms of  $\underline{\tt{Free}}_{\underline{\EuScript C}}$ is the smallest congruence (with respect to $_x\bo_y$) generated by the rule
\begin{center}\mybox{\begin{tabular}{c}
$\displaystyle\frac{\substack{ W_1:X\enspace W_2:Y \enspace  x\in X \enspace    y\in Y  \enspace   X\backslash\{x\}\cap Y\backslash\{y\}=\emptyset \enspace x',y'\not\in X\backslash\{x\}\cup Y\backslash\{y\} \enspace x'\neq y' \\[0.1cm] a\in  P_{\underline{\EuScript C}}(W_1) \enspace\enspace {\it{FV}}(a)=X_1 \enspace\enspace x\in X_1\cap X \\[0.1cm] b\in  P_{\underline{\EuScript C}}(W_2) \enspace\enspace {\it{FV}}(b)=Y_1 \enspace\enspace y\in Y_1\cap Y \\[0.1cm] \tau_1:X_1\backslash\{x\}\cup \{x'\}\rightarrow X_1\enspace\enspace \tau_1|_{X_1\!\backslash\{x\}}={\it id}_{X_1\!\backslash\{x\}} \enspace\enspace \tau_1(x')=x \\[0.1cm]  \tau_2:Y_1\backslash\{y\}\cup \{y'\}\rightarrow Y_1\enspace\enspace \tau_2|_{Y_1\!\backslash\{y\}}={\it id}_{Y_1\!\backslash\{y\}} \enspace\enspace \tau_2(y')=y \vspace{0.1cm}}}{\vspace{0.3cm}{W}_1\,{_x\bo_y}\,{W}_2\equiv{W}_1[\underline{a^{\tau_1}}/\underline{a}] \,{_{x'}\bo_{y'}}\,{W}_2[\underline{b^{\tau_2}}/\underline{b}]}$  \end{tabular}}
\end{center}
 where ${W}_1[\underline{a^{\tau_1}}/\underline{a}]$  (resp. ${W}_2[\underline{b^{\tau_2}}/\underline{b}]$) denotes the result of the  substitution of the parameter $\underline{a^{\tau_1}}$ (resp. $\underline{b^{\tau_2}}$)  for the parameter $\underline{a}$ (resp. $\underline{b}$)   in   $W_1$ (resp. $W_2$).
 \begin{example}\label{gggg}  
Returning to the syntax ${\tt{Free}}_{\underline{\EuScript C}}$, which encompasses terms of the form ${\cal W}^{\sigma}$,  observe that, fixing $a=[[\underline{a}]]_X$ and $b=[[\underline{b}]]_Y$,  by  \hyperlink{(EQ)}{\texttt{(EQ)}}, we have 
$$\begin{array}{rcl}
[[\underline{a}\,{_x\bo_y}\,\underline{b}]]_{X\backslash\{x\}\cup Y\backslash\{y\}}&=&a \,{_x\circ_y}\, b\\[0.15cm]
&=&(a \,{_x\circ_y}\, b)^{{\it id}_{X\backslash\{x\}\cup Y\backslash\{y\}}}\\[0.15cm]
&=&a^{\tau_1} \,{_{x'}\circ_{y'}}\, b^{\tau_2}\\[0.15cm]
&=&[[\underline{a^{\tau_1}}\,{_{x'}\bo_{y'}}\,\underline{b^{\tau_2}}]]_{X\backslash\{x\}\cup Y\backslash\{y\}},
\end{array}$$
where $\tau_1:X\backslash\{x\}\cup \{x'\}\rightarrow X$ and $\tau_2:Y\backslash\{y\}\cup \{y'\}\rightarrow Y$ are evident bijections. 
The first and the last object term  in this sequence of equalities of interpretations are  object terms of $\underline{\tt{Free}}_{\underline{\EuScript C}}$ and they are $\alpha$-equivalent.  \hfill$\square$ 
\end{example}

\indent  The substitution of parameters of object terms canonically induces substitution of parameters of arrow terms of $\underline{\tt Free}_{\underline{\EuScript C}}$.  For an arrow term $\varphi:U\rightarrow V$ of $\underline{{\tt Free}}_{{\underline{\EuScript C}}}$, $\underline{a}\in P_{\underline{\EuScript C}}(U)$ and $\underline{a^{\tau}}\not\in P_{\underline{\EuScript C}}(U)$, such that $U[\underline{a^{\tau}}/\underline{a}]$ (and thus also $V[\underline{a^{\tau}}/\underline{a}]$) is well-typed,  the arrow term $\varphi[\underline{a^{\tau}}/\underline{a}]:U[\underline{a^{\tau}}/\underline{a}]\rightarrow V[\underline{a^{\tau}}/\underline{a}]$  is defined straightforwardly by modifying the indices of $\varphi$ as dictated by  $U[\underline{a^{\tau}}/\underline{a}]$. 
\begin{example}\label{ex1} If $\varphi=\beta^{x,\underline{x};y,\underline{y}}_{W_1,W_2,W_3}$, where $x\in X_1$, $a\in {\underline{\EuScript C}}(X_1)$ and $\underline{a}\in P_{\underline{\EuScript C}}(W_1)$,  then  $$\beta^{x,\underline{x};y,\underline{y}}_{W_1,W_2,W_3}[\underline{a^{\tau}}/\underline{a}]=\beta^{x',\underline{x};y,\underline{y}}_{W_1[\underline{a^{\tau}}/\underline{a}],W_2,W_3},$$ where $x'=\tau^{-1}(x)$.  \hfill$\square$ 
\end{example}
We shall need the following property of the ``interpretation of   substitution''.

 \begin{lem}\label{intofsub} Let $W$ be an object term   of $\underline{{\tt Free}}_{{\underline{\EuScript C}}}(X)$ and let $x\in X$. Let $\underline{a}\in P_{\underline{\EuScript C}}(W)$ be such that $x\in {\it FV}(a)$, and suppose that $\tau: {\it FV}(a)\backslash\{x\}\cup\{x'\}\rightarrow {\it FV}(a)$ renames $x$ to $x'$. We then have 
$$[\,W[\underline{a^{\tau}}/\underline{a}]\,]=[\,W\,]^{\sigma},$$ where $\sigma:X\backslash\{x\}\cup\{x'\}\rightarrow X$ renames $x$ to $x'$. 
Additionally, for any arrow term $\varphi$ of $\underline{{\tt Free}}_{{\underline{\EuScript C}}}(X)$ that has $W$ as the source, we have 
$$[\,\varphi[\underline{a^{\tau}}/\underline{a}]\,]=[\,\varphi\,]^{\sigma}.$$ 
\end{lem}
\begin{proof} By easy inductions, thanks to  \hyperlink{(EQ)}{\texttt{(EQ)}},  \hyperlink{bs}{{\tt ($\beta\sigma$)}}, \hyperlink{gs}{\tt{($\gamma\sigma$)}}, \hyperlink{(EQ-mor)}{{\tt(EQ-mor)}} and Remark \ref{functoriality}(5).
\end{proof}

 \begin{lem}\label{int}
If $W_1\equiv W_2$, then $[W_1]_X=[W_2]_X$.
\end{lem}
\begin{proof}
By induction on the proof of $W_1\equiv W_2$ and Lemma \ref{intofsub}.
\end{proof}
\indent We now specify the syntax  {\uwave{\tt Free}}$_{{\underline{\EuScript C}}}$. The object terms and the   raw arrow terms of  {\uwave{\tt Free}}$_{{\underline{\EuScript C}}}$ are exactly the object terms and the raw arrow terms of $\underline{\tt Free}_{\underline{\EuScript C}}$. The type of an arrow term $\varphi$  of {\uwave{\tt Free}}$_{{\underline{\EuScript C}}}$ is again
a  pair of    object terms, which we shall denote with $\vdash \varphi:U\rightarrow V$. The typing rules for arrow terms are the same as the typing rules for arrow terms of $\underline{\tt Free}_{\underline{\EuScript C}}$, except for the composition rule, for which we now set:
\begin{center}
\mybox{
\begin{tabular}{c}
{\small $\displaystyle\frac{\vdash \varphi_1:W_1\rightarrow W_2\quad \vdash\varphi_2:W'_2\rightarrow W_3\quad W_2\equiv W'_2}{\vdash\varphi_2\circ\varphi_1:W_1\rightarrow W_3}$} 
\end{tabular}
}
\end{center}
\indent The interpretation  of {\uwave{\tt Free}}$_{{\underline{\EuScript C}}}(X)$ in ${\EuScript C}(X)$, is defined (and denoted) exactly as the interpretation $[-]_X$.  In particular, the interpretation of the ``relaxed'' composition is defined by $[\varphi_2\circ\varphi_1]_X=[\varphi_2]_X\circ [\varphi_1]_X$. The following lemma is a direct consequence of Lemma \ref{int}.

\begin{lem}
The interpretation function $[-]_X:$\,{\uwave{\tt Free}}$_{{\underline{\EuScript C}}}(X)\rightarrow {\EuScript C}(X)$ is well-defined. 
\end{lem}

The construction from the following lemma shows the transition from {\uwave{\tt Free}}$_{{\underline{\EuScript C}}}$  to $\underline{\tt Free}_{\underline{\EuScript C}}$, which will be an important step of the first reduction.

\begin{lem}\label{construction}
If $\vdash \varphi:U\rightarrow V$ is an arrow term of {\uwave{\tt Free}}$_{{\underline{\EuScript C}}}(X)$   and if $U\equiv U'$, then there exists an arrow term $\varphi^{U'}:U'\rightarrow V'$ of $\underline{\tt Free}_{\underline{\EuScript C}}(X)$, such that $$V\equiv V'\quad\quad\mbox{ and }\quad\quad [\varphi]_X=[\varphi^{U'}]_{X}.$$
\end{lem}
\begin{proof} Let $W_t(\psi)$ denote the target of an arrow term $\psi$ in $\underline{\tt Free}_{\underline{\EuScript C}}$.  The proof goes by induction on the structure of $\varphi$.\\[-0.6cm]
\begin{itemize}
\item If $\varphi=1_U$, then $\varphi^{U'}=1_{U'}$. We conclude by  \hyperlink{(EQ)}{\texttt{(EQ)}} and Remark \ref{functoriality}(5),
%(a)
for $\sigma={\it id}_X$.\\[-0.6cm]
\item Suppose that $\varphi=\beta^{x,\underline{x};y,\underline{y}}_{W_1,W_2,W_3}$. The source of $\varphi$ is then $U=(W_1\,{_x\bo_{\underline{x}}}\, W_2)\,{_y\bo_{\underline{y}}}\, W_3$. If the parameters $a_1\in {P}_{\underline{\EuScript C}}(W_1)$, $a_{21},a_{22}\in {P}_{\underline{\EuScript C}}(W_2)$  and $a_{3}\in {P}_{\underline{\EuScript C}}(W_3)$  are such that $x\in {\it FV}(a_1),$ $\underline{x},y\in {\it FV}(a_2)$  and $\underline{y}\in {\it FV}(a_3)$,
then  $U'=(W'_1 {_{x'}\bo_{\underline{x}'}}W'_2){_{y'}\bo_{\underline{y}'}}W'_3$, where 
 $$W_1[\underline{a^{\tau_1}_1}/\underline{a_1}]\equiv W'_1,\enspace W_2[\underline{a_{21}^{\tau_{21}}}/\underline{a_{21}}][\underline{a_{22}^{\tau_{22}}}/\underline{a_{22}}]\equiv W'_2 \enspace \mbox{ and } \enspace W_3[\underline{a_3^{\tau_3}}/\underline{a_3}]\equiv W'_3  $$ and  $\tau_1$, $\tau_{21}$, $\tau_{22}$ and $\tau_3$ rename $x$ to $x'$, $\underline{x}$ to $\underline{x}'$, $y$ to $y'$ and $\underline{y}$ to $\underline{y}'$.      We set $$\varphi^{U'}=\beta^{x',\underline{x}';y',\underline{y}'}_{W'_1,W'_2,W'_3}. $$ We conclude by  \hyperlink{(EQ)}{\texttt{(EQ)}} and \hyperlink{bs}{{\tt{($\beta\sigma$)}}}, for $\sigma={\it id}_X$.\\[-0.6cm]
\item  If $\varphi=\beta^{x,\underline{x};y,\underline{y}\enspace ^{_{-1}}}_{W_1,W_2,W_3}$, we proceed analogously as in the previous case.  \\[-0.6cm]
\item Suppose that $\varphi=\gamma^{x,y}_{W_1,W_2}$. The source of $\varphi$ is then $U=W_1\,{_x\bo_{{y}}}\, W_2$. If the parameters 
$a_1\in {P}_{\underline{\EuScript C}}(W_1)$  and $a_{2}\in {P}_{\underline{\EuScript C}}(W_2)$  are such that $x\in {\it FV}(a_1)$ and $y\in {\it FV}(a_2)$,
then $U'=W'_1 {_{x'}\bo_{y'}}W'_2$, where $$W_1[\underline{a^{\tau_1}_1}/\underline{a_1}]\equiv W'_1\enspace \mbox{ and } \enspace W_2[\underline{a^{\tau_2}_2}/\underline{a_2}]\equiv W'_2$$ and $\tau_1$ and $\tau_{2}$  rename $x$ to $x'$ and $y$ to $y'$, respectively.  We set $$\varphi^{U'}=\gamma^{x',y'}_{W'_1,W'_2}$$ and conclude by   \hyperlink{(EQ)}{\texttt{(EQ)}} and \hyperlink{bg}{{\tt{($\gamma\sigma$)}}}, for $\sigma={\it id}_X$.\\[-0.6cm]
\item   Suppose that $\vdash\varphi_1:U\rightarrow W$, $\vdash\varphi_2:W'\rightarrow V$ and that  $W\equiv W'$, and let $\varphi=\varphi_2\circ\varphi_1$. By the induction hypothesis for $\varphi_1$ and $U'$, there exists an arrow term  $$\varphi_1^{U'}:U'\rightarrow W_t(\varphi_1^{U'}),$$ such that $W_t(\varphi_1^{U'})\equiv W$ and  $[\varphi_1]_X=[\varphi^{U'}_1]_{X}$. Since  $W\equiv W'$, by the transitivity of $\equiv$,  we get  $W_t(\varphi_1^{U'})\equiv W'$. By the induction hypothesis for $\varphi_2$ and $W_t(\varphi_1^{U'})$, there exists an arrow term $$\varphi_2^{W_t(\varphi_1^{U'})}:W_t(\varphi_1^{U'})\rightarrow W_t(\varphi_2^{W_t(\varphi_1^{U'})}),$$ such that $W_t(\varphi_2^{W_t(\varphi_1^{U'})})\equiv V$ and $[\varphi_2]_X=[\varphi_2^{W_t(\varphi_1^{U'})}]_X$.  We define $$\varphi^{U'}=\varphi_2^{W_t(\varphi_1^{U'})}\circ\varphi_1^{U'}.$$
\item   Suppose that $\vdash\varphi_1:U_1\rightarrow V_1$, $\vdash\varphi_2:U_2\rightarrow V_2$, and let $\varphi=\varphi_1\,{_x\bo_y}\,\varphi_2$. In this case, the source of $\varphi$ is $U=U_1\,{_x\bo_y}\, U_2$, and we have two possibilities for the shape of $U'$.
\begin{itemize} 
\item
$U'=U'_1\,{_{x'}\bo_{y'}}\, U'_2$, where, assuming that $a_1\in {P}_{\underline{\EuScript C}}(U_1)$ and $a_{2}\in {P}_{\underline{\EuScript C}}(U_2)$ are such that $x\in {\it FV}(a_1)$ and $y\in {\it FV}(a_2)$, $U_1[\underline{a^{\tau_1}_1}/\underline{a_1}]\equiv U'_1$  and $U_2[\underline{a^{\tau_2}_2}/\underline{a_2}]\equiv U'_2$. Since $\underline{a^{\tau_1}_1}\in P_{\underline{\EuScript C}}(U'_1)$ and $\underline{a^{\tau_2}_2}\in P_{\underline{\EuScript C}}(U'_2)$,
this means that, symmetrically, we have $U'_1[\underline{a_1}/\underline{a^{\tau_1}_1}]\equiv U_1$  and $U'_2[\underline{a_2}/\underline{a^{\tau_2}_2}]\equiv U_2$. By the induction hypothesis for $\varphi_1$ and $U'_1[\underline{a_1}/\underline{a^{\tau_1}_1}]$, as well as $\varphi_2$ and $U'_2[\underline{a_2}/\underline{a^{\tau_2}_2}]$, we get arrow terms $$\varphi_1^{U'_1[\underline{a_1}/\underline{a^{\tau_1}_1}]}:U'_1[\underline{a_1}/\underline{a^{\tau_1}_1}]\rightarrow W_t(\varphi_1^{U'_1[\underline{a_1}/\underline{a^{\tau_1}_1}]})$$ and  $$\varphi_2^{U'_2[\underline{a_2}/\underline{a^{\tau_2}_2}]}:U'_2[\underline{a_2}/\underline{a^{\tau_2}_2}]\rightarrow W_t(\varphi_2^{U'_2[\underline{a_2}/\underline{a^{\tau_2}_2}]}),$$ such that  $W_t(\varphi_1^{U'_1[\underline{a_1}/\underline{a^{\tau_1}_1}]})\equiv V_1$, $[\varphi_1]_X=[\varphi_1^{U'_1[\underline{a_1}/\underline{a^{\tau_1}_1}]}]_X$, $W_t(\varphi_2^{U'_2[\underline{a_2}/\underline{a^{\tau_2}_2}]})\equiv V_2$ and $[\varphi_2]_X=[\varphi_2^{U'_2[\underline{a_2}/\underline{a^{\tau_2}_2}]}]_X$. By means of substitution on arrow terms, we define $$\varphi^{U'}=\varphi_1^{U'_1[\underline{a_1}/\underline{a^{\tau_1}_1}]}[\underline{a^{\tau_1}_1}/\underline{a_1}]\,{_{x'}\bo_{y'}}\,\varphi_2^{U'_2[\underline{a_2}/\underline{a^{\tau_2}_2}]}[\underline{a^{\tau_2}_2}/\underline{a_2}].$$  

\item $U'=U'_1\,{_{x}\bo_{y}}\, U'_2$, where $U_1\equiv U'_1$ and $U_2\equiv U'_2$. We define $\varphi^{U'}=\varphi_1^{U'_1} \,{_{x}\bo_{y}}\,\varphi_2^{U'_2}$. 
\end{itemize}
We conclude by Lemma \ref{int}.
\end{itemize}
\vspace{-0.6cm}
\end{proof}
\subsubsection{The first reduction}\label{fr1}
We make the first reduction  in two steps. We first define a (non-deterministic) rewriting algorithm $\leadsto$ on  ${\tt Free}_{\underline{\EuScript C}}(X)$  with outputs in {\uwave{\tt Free}}$_{{\underline{\EuScript C}}}$, in such a way that the interpretation of a   term of ${\tt Free}_{\underline{\EuScript C}}$  matches the interpretations of (all) its ``normal forms'' relative to $\leadsto$. We then use Lemma \ref{construction} to move from {\uwave{\tt Free}}$_{{\underline{\EuScript C}}}$ to ${\underline{\tt Free}}_{{\underline{\EuScript C}}}$, while preserving the equality of interpretations from the first step.  This allows us to reduce the proof of the coherence theorem, which  concerns all $\beta\gamma\sigma$-diagrams, to the consideration of parallel $\beta\gamma$-arrows in ${\EuScript C}(X)$ only.\\[0.1cm]
\indent  We   first define the rewriting algorithm $\leadsto$ on   object terms of  ${\tt Free}_{\underline{\EuScript C}}$. The algorithm $\leadsto$ takes an object term ${\cal W}$ of ${\tt Free}_{\underline{\EuScript C}}$ and returns (non-deterministically) an object term $W$ of {\uwave{\tt Free}}$_{{\underline{\EuScript C}}}$, which we denote  by ${\cal W}\leadsto W$, in the way specified by the following rules:
\begin{center}
\mybox{ 
\begin{tabular}{c}

{ $\displaystyle \frac{}{\underline{a}\leadsto\underline{a}}$\quad\quad $\displaystyle\frac{{\cal W}_1\leadsto W_1\quad\quad {\cal W}_2\leadsto W_2}{{\cal W}_1\,{_x\bo_y}\,{\cal W}_2\leadsto W_1\,{_x\bo_y}\,W_2}$}\\[0.7cm] 
{$\displaystyle\frac{}{ \underline{a}^{\sigma}\leadsto\underline{a^{\sigma}}}$\quad\quad $\displaystyle\frac{{\cal W}\leadsto W}{{\cal W}^{{\it id}_X}\leadsto W}$ \quad\quad $\displaystyle\frac{{\cal W}^{\sigma\circ\tau}\leadsto W}{({\cal W}^{\sigma})^{\tau}\leadsto W}$}\\[0.7cm]
{ $\displaystyle\frac{\substack{\sigma\,:\,Z\,\rightarrow\, X\backslash\{x\}\cup Y\backslash\{y\}  \quad x',y'\not\in X\backslash\{x\}\cup Y\backslash\{y\}\quad x'\neq y'\\[0.05cm] \sigma_1:\,\sigma^{-1}[X\backslash\{x\}]\cup \{x'\}\,\rightarrow\, X\enspace\enspace  \sigma_1|^{X\backslash\{x\}}=\sigma|^{X\backslash\{x\}}\enspace\enspace \sigma_1(x')=x \\[0.05cm] \sigma_2\,:\,\sigma^{-1}[Y\backslash\{y\}]\cup \{y'\}\,\rightarrow\, Y\enspace\enspace  \sigma_2|^{Y\backslash\{y\}}=\sigma|^{Y\backslash\{y\}}\enspace\enspace \sigma_2(y')=y\\[0.1cm]{\mbox{$ {\cal W}_1^{\sigma_1}\leadsto W_1 \quad  {\cal W}_2^{\sigma_2}\leadsto W_2$}}}}{({\cal W}_1\,{_x\bo_y}\,{\cal W}_2)^{\sigma}\leadsto W_1\,{_{x'}\bo_{y'}}\,W_2}$}
\end{tabular}}
\end{center}

\noindent The formal system defined above  obviously has a termination property, in the sense  that for all object terms ${\cal W}$ of ${\tt Free}_{\underline{\EuScript C}}$ there exists an object term $W$ of {\uwave{\tt Free}}$_{{\underline{\EuScript C}}}$, such that ${\cal W}\leadsto W$.  Notice also that the last rule is non-deterministic, as it involves a choice of $x'$ and $y'$.     In what follows,  for an  object term ${\cal W}$ of  ${\tt Free}_{\underline{\EuScript C}}$, we shall say that the outputs of the algorithm $\leadsto$ applied on  ${\cal W}$ are
 {\em  normal forms of} ${\cal W}$. We shall denote the collection of all normal forms of ${\cal W}$ with ${\tt NF}({\cal W})$. 

\indent The formal system $({\tt Free}_{\underline{\EuScript C}},\leadsto)$  satisfies the following confluence-like property.
\begin{lem}\label{confluence}
If $W_1,W_2\in {\tt NF}({\cal W})$, then $W_1\equiv W_2$.
\end{lem} 
\begin{proof}  Suppose that   $({\cal W}_1\,{_x\bo_y}\,{\cal W}_2)^{\sigma}\leadsto W_1\,{_{x'}\bo_{y'}}\,W_2$ is obtained from ${\cal W}_1^{\sigma_1}\leadsto W_1$ and  ${\cal W}_2^{\sigma_2}\leadsto W_2$, and   $({\cal W}_1\,{_x\bo_y}\,{\cal W}_2)^{\sigma}\leadsto W'_1\,{_{x''}\bo_{y''}}\,W'_2$  from ${\cal W}_1^{\tau_1}\leadsto W'_1$ and  ${\cal W}_2^{\tau_2}\leadsto W'_2$. Let $a\in P_{\underline{\EuScript C}}(W_1)$ and $b\in P_{\underline{\EuScript C}}(W_2)$ be such that ${\it FV}(a)=X_1$, ${\it FV}(b)=Y_1$, $x\in X_1$ and $y\in Y_1$, and let $\kappa_1:X_1\backslash\{x'\}\cup\{x''\}\rightarrow X_1$ be the  renaming of   $x'$ to $x''$ and $\kappa_2:Y_1\backslash\{y'\}\cup\{y''\}\rightarrow Y_1$ the renaming of  $y'$ to $y''$. It is then easy to show that ${\cal W}^{\tau_1}_1\leadsto {W}_1[\underline{a^{\kappa_1}}/\underline{a}]$ and ${\cal W}^{\tau_2}_2\leadsto {W}_2[\underline{a^{\kappa_2}}/\underline{a}]$. By the definition of $\equiv$ and the induction hypothesis for  ${\cal W}^{\tau_1}_1$ (which reduces to both $W'_1$ and ${W}_1[\underline{a^{\kappa_1}}/\underline{a}]$) and ${\cal W}^{\tau_2}_2$ (which reduces to both $W'_2$ and ${W}_2[\underline{b^{\kappa_2}}/\underline{b}]$), we  then have 
$$ W_1\,{_{x'}\bo_{y'}}\,W_2\equiv {W}_1[\underline{a^{\kappa_1}}/\underline{a}] \,{_{x''}\bo_{y''}}\,{W}_2[\underline{b^{\kappa_2}}/\underline{b}]={W}'_1 \,{_{x''}\bo_{y''}}\,{W}'_2.$$

\end{proof}

\begin{lem}\label{eq_int_obj}
For an arbitrary object term ${\cal W}:X$ of ${\tt Free}_{\underline{\EuScript C}}$ and an arbitrary $W\in {\tt NF}({\cal W})$, the equality $[[{\cal W}]]_X=[W]_X$ holds in ${\EuScript C}(X)$.
\end{lem}

\begin{proof}
By induction on the structure of ${\cal W}$.\\[-0.6cm]
\begin{itemize}
\item If ${\cal W}=\underline{a}$, we trivially have $[[\underline{a}]]_X=a=[\underline{a}]_{X}$. \\[-0.6cm]
\item If ${\cal W}={\cal W}_1{_x\bo_y}{\cal W}_2$, where ${\cal W}_1:X$ and ${\cal W}_2:Y$, then, for any $W_1\in {\tt NF}({\cal W}_1)$ and $W_2\in {\tt NF}({\cal W}_2)$,    $W_1\,{_x\bo_y}\, W_2\in {\tt NF}({\cal W})$. Hence, by Lemma \ref{confluence}, we have that   $W\equiv W_1\,{_x\bo_y}\, W_2$. By the induction hypothesis for ${\cal W}_1$ and ${\cal W}_2$, we have that $[[{\cal W}_1]]_X=[W_1]_{X}$ and  $[[{\cal W}_2]]_Y=[W_2]_{Y}$, and, by Lemma \ref{int}, we get 
$$\begin{array}{rcl}
[[\,{\cal W}_1{_x\bo_y}{\cal W}_2\,]]_{X\backslash\{x\}\cup Y\backslash\{y\}}&=&[[{\cal W}_1]]_X\,{_x\circ_y}\,[[{\cal W}_2]]_Y\\[0.1cm]
&=&[W_1]_{X}\,{_x\circ_y}\,[W_2]_{Y}\\[0.1cm]
&=&[W_1\,{_x\bo_y}\,W_2]_{X\backslash\{x\}\cup Y\backslash\{y\}}\\[0.1cm]
&=&[W]_{X\backslash\{x\}\cup Y\backslash\{y\}}.\\[-0.2cm]
\end{array}$$  
\item Suppose that ${\cal W}={\cal V}^{\sigma}$, where ${\cal V}:X$ and $\sigma:Y\rightarrow X$. We proceed by case analysis relative to the shape of ${\cal V}$ (and $\sigma$).
\begin{itemize}
\item If ${\cal V}=\underline{a}$, for some $a\in P_{\underline{\EuScript C}}$, then $[[\underline{a}^{\sigma}]]_Y=[[\underline{a}]]_X^{\sigma}=[\underline{a}]_X^{\sigma}=a^{\sigma}=[\underline{a^{\sigma}}]_Y.$ 
 
\item If $\sigma={\it id}_X$, and if    $V\in {\tt NF}({\cal V})$, then $V\in  {\tt NF}({\cal W})$, and, by Lemma \ref{confluence}, we have that $W\equiv V$. By the induction hypothesis for ${\cal V}$ and Lemma \ref{int}, we get 
$$[[{\cal V}^{{\it id}_X}]]_X=[[{\cal V}]]^{{\it id}_X}_X=[[{\cal V}]]_X=[V]_{X}=[W]_{X}.$$
\item If ${\cal V}={\cal V}_1{_x\bo_y}{\cal V}_2$,  $V_1\in {\tt NF}({\cal V}^{\sigma_1}_1)$ and  $V_2\in {\tt NF}({\cal V}^{\sigma_2}_2)$, then $V_1 {_{x'}\bo_{y'}} V_2\in {\tt NF}({\cal W})$, and, by Lemma \ref{confluence}, $W\equiv V_1 {_{x'}\bo_{y'}} V_2$. By the induction hypothesis for ${\cal V}^{\sigma_1}_1$ and ${\cal V}^{\sigma_2}_2$, \hyperlink{(EQ)}{\texttt{(EQ)}} and Lemma \ref{int}, we get 
 $$ 
[[({\cal V}_1\,{_x\bo_y}\,{\cal V}_2)^{\sigma}]]_Y= [[{\cal V}^{\sigma_1}_1]]_{Y_1}\,{_{x'}\bo_{y'}}\,[[{\cal V}^{\sigma_2}_2]]_{Y_2}  = [V_1]_{Y_1}\,{_{x'}\bo_{y'}}\, [V_2]_{Y_2} = [V_1\,{_{x'}\bo_{y'}}\, V_2]_{Y}= [W]_{Y}.
 $$  
\item If ${\cal V}={\cal U}^{\tau}$, and if $U\in {\tt NF}({\cal U}^{\tau\circ\sigma})$, then $U\in {\tt NF}({\cal W})$, and, by Lemma \ref{confluence}, $W\equiv U$. By the induction hypothesis for ${\cal U}^{\tau\circ\sigma}$ and Lemma \ref{int}, we get$$
[[({\cal U}^{\tau})^{\sigma}]]_Y=([[{\cal U}]]^{\tau}_X)^{\sigma}=[[{\cal U}]]^{\tau\circ\sigma}_X=[[{\cal U}^{\tau\circ\sigma}]]_Y=[U]_Y=[W]_Y .$$
\end{itemize}
\end{itemize}
\vspace{-0.65cm}
\end{proof}
\indent We move on to the first reduction of   arrow terms of ${{\tt Free}}_{\underline{\EuScript C}}$. Our first step is to  define a (non-deterministic) rewriting algorithm $\leadsto$, which ``normalizes''  arrow terms  of ${{\tt Free}}_{\underline{\EuScript C}}$. In the table below, in the rules defining the reductions $$(\beta^{x,\underline{x};y,\underline{y}}_{{\cal W}_1,{\cal W}_2,{\cal W}_3})^{\sigma}\leadsto  \beta^{x',\underline{x}';y',\underline{y}'}_{W_1,{W_2},{W_3}}\quad\mbox{ and }\quad (\beta^{x,\underline{x};y,\underline{y}\enspace ^{_{-1}}}_{{\cal W}_1,{\cal W}_2,{\cal W}_3})^{\sigma}\leadsto  \beta^{x',\underline{x}';y',\underline{y}'\enspace ^{_{-1}}}_{W_1,{W_2},{W_3}},$$
it is assumed that 
\begin{itemize}
\item[$\diamond$] $\sigma:U\rightarrow X\backslash\{x\}\cup Y\backslash\{\underline{x},y\}\cup Z\backslash\{\underline{y}\}$,  
\item[$\diamond$] $\sigma_1: \sigma^{-1}[X\backslash\{x\}]\cup\{x'\}\rightarrow X$, \enspace $\sigma_1|^{X\backslash\{x\}}=\sigma|^{X\backslash\{x\}}$, \enspace $\sigma_1(x')=x$,
\item[$\diamond$] $\sigma_2: \sigma^{-1}[Y\backslash\{\underline{x},y\}]\cup\{\underline{x}',y'\}\rightarrow Y$, \enspace $\sigma_2|^{Y\backslash\{\underline{x},y\}}=\sigma|^{Y\backslash\{\underline{x},y\}}$, \enspace $\sigma_2(\underline{x}')=\underline{x}$, \enspace $\sigma_2(y')=y$, and
\item[$\diamond$] $\sigma_3: \sigma^{-1}[Z\backslash\{\underline{y}\}]\cup\{{\underline{y}}'\}\rightarrow Z$, \enspace $\sigma_3|^{Z\backslash\{\underline{y}\}}=\sigma|^{Z\backslash\{\underline{y}\}}$, \enspace $\sigma_3(\underline{y}')=\underline{y}$.
\end{itemize} 
Also, in the rule defining the reduction $$(\gamma^{x,y}_{{\cal W}_1,{\cal W}_2})^{\sigma}\leadsto \gamma^{x',y'}_{W_1,{W_2}},$$
it is assumed that 
\begin{itemize}
\item[$\diamond$] $\sigma:Z\rightarrow X\backslash\{x\}\cup Y\backslash\{y\}$,  
\item[$\diamond$]   $\sigma_1: \sigma^{-1}[X\backslash\{x\}]\cup\{x'\}\rightarrow X$, \enspace $\sigma_1|^{X\backslash\{x\}}=\sigma|^{X\backslash\{x\}}$, \enspace $\sigma_1(x')=x$, and
\item[$\diamond$] $\sigma_2: \sigma^{-1}[Y\backslash\{{y}\}]\cup\{{{y}}'\}\rightarrow Z$, \enspace $\sigma_2|^{Y\backslash\{{y}\}}=\sigma|^{Y\backslash\{{y}\}}$, \enspace $\sigma_2({y}')={y}$.
\end{itemize}

\begin{center}
\mybox{
\begin{tabular}{c}
{\small $\displaystyle\frac{ U\in {\tt NF}({\cal U})}{1_{\cal U}\leadsto 1_{U}}$}\\[0.55cm]
 {\small $\displaystyle\frac{W_i\in {\tt NF}({\cal W}_i)\enspace\enspace i\in\{1,2,3\}}{ \beta^{x,\underline{x};y,\underline{y}}_{{\cal W}_1,{\cal W}_2,{\cal W}_3}\leadsto  \beta^{x,\underline{x};y,\underline{y}}_{W_1,{W_2},{W_3}}}$\quad\quad $\displaystyle\frac{W_i\in {\tt NF}({\cal W}_i)\enspace\enspace i\in\{1,2,3\}}{ \beta^{x,\underline{x};y,\underline{y}\enspace ^{_{-1}}}_{{\cal W}_1,{\cal W}_2,{\cal W}_3}\leadsto  \beta^{x,\underline{x};y,\underline{y}\enspace ^{_{-1}}}_{W_1,{W_2},{W_3}}}$}\\[0.55cm]
{\small  $\displaystyle\frac{W_i\in {\tt NF}({\cal W}_i)\enspace\enspace i\in\{1,2\}}{ \gamma^{x,y}_{{\cal W}_1,{\cal W}_2}\leadsto \gamma^{x,y}_{W_1,{W_2}}} $}\\[0.55cm] 
{\small $\displaystyle\frac{}{{\varepsilon_1}_{\underline{a}}^{\sigma}\leadsto 1_{\underline{a^{\sigma}}}}$\quad\quad $\displaystyle\frac{}{{\varepsilon_1}_{\underline{a}}^{\sigma\,^{_{-1}}}\leadsto 1_{\underline{a^{\sigma}}}}$} \quad\quad {\small $\displaystyle \frac{W\in {\tt NF}({\cal W})}{{\varepsilon_2}_{\cal W}\leadsto 1_{W}}$\quad\quad $\displaystyle \frac{W\in {\tt NF}({\cal W})}{{\varepsilon_2}^{_{-1}}_{\cal W}\leadsto 1_{W}}$} \\[0.55cm]
{\small $\displaystyle\frac{W\in {\tt NF}({\cal W}^{\sigma\circ\tau})}{{\varepsilon_3}_{\cal W}^{\sigma,\tau}\leadsto 1_{W}}$\quad\quad $\displaystyle\frac{W\in {\tt NF}({\cal W}^{\sigma\circ\tau})}{{\varepsilon_3}_{\cal W}^{\sigma,\tau\,^{_{-1}}}\leadsto 1_{W}}$}\\[0.55cm]
{\small $\displaystyle\frac{W\in {\tt NF}(({\cal W}_1 {_x\bo_y }{\cal W}_2)^{\sigma})}{ {\varepsilon_4}^{x,y;x',y'}_{{\cal W}_1,{\cal W}_2;\sigma}\leadsto 1_W}$\quad\quad $\displaystyle\frac{W_1\in {\tt NF}({\cal W}_1^{\sigma_1})\quad W_2\in {\tt NF}({\cal W}_2^{\sigma_2})}{ {\varepsilon_4}^{x,y;x',y'\, ^{_{-1}}}_{{\cal W}_1,{\cal W}_2;\sigma_1,\sigma_2}\leadsto 1_W}$}\\[0.55cm]
{\small $\displaystyle\frac{\Phi_1\leadsto \varphi_1\quad \Phi_2\leadsto \varphi_2 }{\Phi_2\circ\Phi_1\leadsto \varphi_2\circ\varphi_1}$ \quad\quad $\displaystyle\frac{\Phi_1\leadsto \varphi_1\quad \Phi_2\leadsto \varphi_2 }{\Phi_1\,{_x\bo_y}\,\Phi_2\leadsto \varphi_1\,{_x\bo_y}\, \varphi_2 }$}\\[0.55cm] 
\small {$\displaystyle\frac{}{1_{\underline{a}}^{\sigma}\leadsto 1_{\underline{a^{\sigma}}}}$}\\[0.55cm]

\small{$\displaystyle\frac{W_i\in {\tt NF}({\cal W}^{\sigma_i}_i)}{(\beta^{x,\underline{x};y,\underline{y}}_{{\cal W}_1,{\cal W}_2,{\cal W}_3})^{\sigma}\leadsto  \beta^{x',\underline{x}';y',\underline{y}'}_{W_1,{W_2},{W_3}}}$\quad\enspace $\displaystyle\frac{W_i\in {\tt NF}({\cal W}^{\sigma_i}_i)}{(\beta^{x,\underline{x};y,\underline{y}\enspace ^{_{-1}}}_{{\cal W}_1,{\cal W}_2,{\cal W}_3})^{\sigma}\leadsto  \beta^{x',\underline{x}';y',\underline{y}'\enspace ^{_{-1}}}_{W_1,{W_2},{W_3}}}$}\\[0.65cm]

\small{$\displaystyle\frac{W_i\in {\tt NF}({\cal W}^{\sigma_i}_i)\enspace\enspace i\in\{1,2\}}{ (\gamma^{x,y}_{{\cal W}_1,{\cal W}_2})^{\sigma}\leadsto \gamma^{x',y'}_{W_1,{W_2}}}$}\\[0.55cm]
{\small $\displaystyle\frac{}{({\varepsilon_1}_{\underline{a}}^{\sigma})^{\kappa}\leadsto {1}_{\underline{a^{\sigma\circ\kappa}}}}$\quad\quad $\displaystyle\frac{}{({\varepsilon_1}_{\underline{a}}^{\sigma\,^{_{-1}}})^{\kappa}\leadsto {1}_{\underline{a^{\sigma\circ\kappa}}}}$}\quad\quad {\small $\displaystyle\frac{W\in {\tt NF}({\cal W}^{\kappa})}{({\varepsilon_2}_{\cal W})^{\kappa}\leadsto 1_{W}}$\quad\quad $\displaystyle\frac{W\in {\tt NF}({\cal W}^{\kappa})}{({\varepsilon_2}^{_{-1}}_{\cal W})^{\kappa}\leadsto 1_{W}}$} \\[0.55cm]
{\small $\displaystyle\frac{W\in {\tt NF}({\cal W}^{\sigma\circ\tau\circ\kappa})}{({\varepsilon_3}_{\cal W}^{\sigma,\tau})^{\kappa}\leadsto 1_{W}}$ \quad\quad $\displaystyle\frac{W\in {\tt NF}({\cal W}^{\sigma\circ\tau\circ\kappa})}{({\varepsilon_3}_{\cal W}^{\sigma,\tau\,^{_{-1}}})^{\kappa}\leadsto 1_{W}}$}\\[0.55cm] 
{\small $\displaystyle\frac{W\in {\tt NF}(({\cal W}_1 {_x\bo_y }{\cal W}_2)^{\sigma\circ\kappa})}{ ({\varepsilon_4}^{x,y;x',y'}_{{\cal W}_1,{\cal W}_2;\sigma})^{\kappa}\leadsto 1_W}$\quad\quad $\displaystyle\frac{W_1\in {\tt NF}({\cal W}_1 ^{\sigma_1\circ\kappa_1})\quad W_2\in {\tt NF}({\cal W}_2 ^{\sigma_2\circ\kappa_2})}{ ({\varepsilon_4}^{x,y;x',y'\,^{_{-1}}}_{{\cal W}_1,{\cal W}_2;\sigma})^{\kappa}\leadsto 1_W}$}\\[0.55cm]

{\small $\displaystyle\frac{\Phi\leadsto \varphi}{\Phi^{{\it id}_X}\leadsto \varphi}$ \quad\quad $\displaystyle\frac{\Phi^{\sigma\circ\tau}\leadsto \varphi}{(\Phi^{\sigma})^{\tau}\leadsto \varphi}$\quad\quad $\displaystyle\frac{\Phi_1^{\sigma_1}\leadsto \varphi_1\quad \Phi_2^{\sigma_2}\leadsto \varphi_2}{(\Phi_1\,{_x\bo_y}\,\Phi_2)^{\sigma}\leadsto \varphi_1 \,{_{x'}\bo_{y'}}\, \varphi_2}$\quad\quad $\displaystyle\frac{\Phi_1^{\sigma}\leadsto \varphi_1\quad \Phi_2^{\sigma}\leadsto \varphi_2}{(\Phi_2\circ\Phi_1)^{\sigma}\leadsto \varphi_2\circ\varphi_1}$}  
\end{tabular} }
\end{center}
We make   first observations about this rewriting algorithm.
\begin{rem} \label{why-auxiliary-type}
Notice that, if $\Phi:{\cal U}\rightarrow {\cal V}$ and if $\Phi\leadsto \varphi$, then $\vdash\varphi:U\rightarrow V$, for some $U\in {\tt NF}({\cal U})$ and $V\in {\tt NF}({\cal V})$.
Also, in the rule defining $(\Phi_2\circ\Phi_1)^{\sigma}\leadsto \varphi_2\circ\varphi_1$, the arrow  term $\varphi_2\circ\varphi_1$ is   not  well-typed  in  $\underline{{\tt Free}}_{\underline{\EuScript C}}$ in general (thus motivating the design of the intermediate system {\uwave{\tt Free}}$_{{\underline{\EuScript C}}}$).  
\end{rem}
\noindent As it was the case for the algorithm on object terms, this formal system is  terminating. 
Therefore, the algorithm gives us, for each arrow term $\Phi:{\cal U}\rightarrow {\cal V}$, the set ${\tt NF}(\Phi)$ of normal forms of $\Phi$, which are arrow terms of  {\uwave{\tt Free}}$_{{\underline{\EuScript C}}}$.  Here is the most important property of these normal forms.
\begin{lem}\label{kh}
For arbitrary arrow term $\Phi$ of ${\tt{Free}_{\underline{\EuScript C}}}(X)$ and     $\varphi\in {\tt NF}(\Phi)$, we have $[[\Phi]]_X=[\varphi]_X$.
\end{lem}
\begin{proof}
By induction on the structure of $\Phi$ and Lemma \ref{eq_int_obj}.
\end{proof}
\subsubsection{The first reduction}
  Suppose that,  for all object terms ${\cal W}$ of ${{\tt Free}}_{\underline{\EuScript C}}$, a normal form ${\tt red_1}({\cal W})\in {\tt NF}({\cal W})$ in {\uwave{\tt Free}}$_{{\underline{\EuScript C}}}$ has been fixed, and that, {independently of that choice}, for all arrow terms $\Phi$ of ${{\tt Free}}_{\underline{\EuScript C}}$ a normal form ${\tt red_1}({\Phi})\in  {\tt NF}({\Phi})$ in {\uwave{\tt Free}}$_{{\underline{\EuScript C}}}$ has been fixed. \\[0.1cm]
\indent We define the { first reduction function} $\tt{Red_1}:{\tt Free}_{\underline{\EuScript C}}\rightarrow \underline{\tt Free}_{\underline{\EuScript C}}$ by
$${\tt{Red_1}}({\cal W})={\tt red_1}({\cal W})  \enspace\enspace\mbox{ and }\enspace \enspace{\tt{Red_1}}(\Phi)={\tt red_1}(\Phi)^{{\tt red_1}({\cal U})},$$ 
where $\Phi:{\cal U}\rightarrow {\cal V}$. Observe that, in the definition of ${\tt Red_1}(\Phi)$, we used the construction of Lemma \ref{construction}, which indeed turns ${\tt red_1}(\Phi)$ (which is an arrow term of {\uwave{\tt Free}}$_{{\underline{\EuScript C}}}$) into an arrow term of $\underline{\tt Free}_{\underline{\EuScript C}}$. Also, for an arrow term $\Phi:{\cal U}\rightarrow {\cal V}$ of ${\tt Free}_{\underline{\EuScript C}}$, we have that ${\tt Red}_1(\Phi)$ is the arrow term with source ${\tt Red}_1({\cal U})$ and target $V$, where, in general, $V\neq {\tt Red_1}({\cal V})$. However, the following important property holds.
\begin{lem}\label{vazno}
For any two arrow terms $\Phi,\Psi: {\cal U}\rightarrow {\cal V}$  of the same type in ${{\tt Free}}_{\underline{\EuScript C}}$, ${\tt Red_1}(\Phi)$ and ${\tt Red_1}(\Psi)$ are arrow terms  of the same type  in  $\underline{\tt Free}_{\underline{\EuScript C}}$.
\end{lem}
\begin{proof}
That  ${\tt Red_1}(\Phi)$ and ${\tt Red_1}(\Psi)$  have the same source is clear by the definition. We prove the equality   $W_t({\tt Red_1}(\Phi))=W_t({\tt Red_1}(\Psi))$  by induction on the proof of   $W_t({\tt Red_1}(\Phi))\equiv W_t({\tt Red_1}(\Psi))$.  Suppose that $W_t({\tt Red_1}(\Phi))=W_1 \,{_x\bo_y}\, W_2$ and  $W_t({\tt Red_1}(\Psi))=W_1[\underline{a^{\tau_1}}/\underline{a}]\,{_{x'}\bo_{y'}}\, W_2[\underline{b^{\tau_2}}/\underline{b}]$.  If, moreover,  at least one of $\tau_1$ and $\tau_2$ is  not the identitiy, i.e., if, say, $x'\neq x$, then, by Remark \ref{indices}, it cannot be the case that  ${\tt Red_1}(\Phi)$ and ${\tt Red_1}(\Psi)$ have the same source. 
\end{proof}

The following proposition, which is essential for the proof of the coherence theorem,  is  simply a consequence of Lemma \ref{eq_int_obj} and Lemma \ref{kh}.
\begin{prop}\label{t1} For an arbitrary object term ${\cal W}$ and an arbitrary arrow term $\Phi$   of  ${{\tt Free}}_{\underline{\EuScript C}}$, the following equalities of interpretations hold:  
 $$[[{\cal W}]]_X=[{\tt{Red_1}}({\cal W})]_{X} \enspace \mbox{ and } \enspace [[\Phi]]_X=[{\tt{Red_1}}(\Phi)]_X .$$
\end{prop}
\subsection{The second reduction: getting rid of the cyclicity}
Our second goal is to cut down the problem of commutation of all $\beta \gamma$-diagrams of ${\EuScript C}(X)$ to the problem of  commutation of all $\beta \vartheta$-diagrams of ${\EuScript C}(X)$ (see Section \ref{pass}).    This reduction is based on a transition from unrooted to rooted trees, and is easier to ``visualize''  using a tree representation of our  syntax, which we introduce next.
\subsubsection{Unrooted trees }\label{section_trees}
We first recall  the formalism of unrooted trees, introduced in \cite[Section 1.2.1]{co} for the definition of the free cyclic operad   built over  $\underline{\EuScript C}$. We   omit the part of the syntax of unrooted trees which accounts for units. Also, as the purpose of the formalism is to provide a representation of the terms of $\underline{\tt{Free}}_{\underline{\EuScript C}}$, which do not involve symmetries,   the unrooted trees will {\em not} be quotiented by $\alpha$-equivalence, as it was the case in \cite[Section 1.2.1]{co}.\\[0.1cm]
\indent A {\em   corolla} is a term  $a(x,y,z,\dots) ,$  where $a\in\underline{\EuScript C}(X)$ and $X=\{x,y,z,\dots\}$. We call the elements of $X$ the {\em free variables} of $a(x,y,z,\dots)$, and we  write   ${\it FV}(a)=X$  to denote this set. \\
\indent  A {\em  graph} ${\cal G}$ is a non-empty, finite set  of corollas with mutually disjoint free variables, together with an involution $\sigma$ on the set $$V({\cal G})=\bigcup_{i=1}^{k}{\it FV}(a_i)$$  of all variables, or {\em half-edges}, occurring in ${\cal G}$. We write   $${\cal G}=\{a_1(x_1,\dots ,x_n), \dots, a_k(y_1,\dots y_m);\sigma\} . $$ 

\indent We shall denote  with ${\it Cor}({\cal G})$  the set of all corollas of ${\cal G}$, and we shall refer to a  corolla  by  its parameter. 
The set of edges ${\it Edge}({\cal G})$ of   ${\cal G}$ consists of unordered pairs $(x,y)$ of distinct half-edges such that $\sigma(x)=y$.  Next,  ${\it FV}({\cal G})$ will denote the set of free variables of ${\cal G}$, i.e., of the  fixpoints of $\sigma$.   Finally,  ${\it FCor}({\cal G})$ will denote the set  of corollas $f$ of ${\cal G}$ for which   ${\it FV}(f)\cap {\it FV}({\cal G})\neq \emptyset$. 
 
\indent  A  graph is an unrooted tree if it is connected and if it does not contain loops, multiple edges and cycles.  \\
\indent To give some intuition, here is an example.\begin{example}\label{ex2} The graph ${\cal G}=\{a(x_1,x_2,x_3,x_4,x_5),b(y_1,y_2,y_3,y_4);\tau\}$, where $\tau(x_1)=y_1$, $\tau(x_2)=y_2$ and $\tau$ is identity otherwise, is {\em not} an unrooted tree, since it has two edges between $a$ and $b$, which can be visualized as
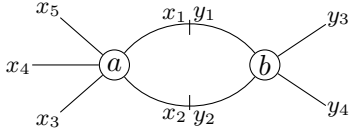
\begin{center}
\begin{tikzpicture}
 \node (f) [circle,fill=none,draw=black,minimum size=4mm,inner sep=0mm]  at (-1,0) {$a$};
\node (g) [circle,fill=none,draw=black,minimum size=4mm,inner sep=0mm]  at (1,0) {$b$};
\node (a) [label={[xshift=-0.08cm, yshift=-0.29cm,]{\footnotesize $x_5$}},circle,fill=none,draw=none,minimum size=2mm,inner sep=0mm]  at (-1.8,0.7) {};
\node (b) [label={[xshift=-0.08cm,yshift=-0.37cm,]{\footnotesize $x_4$}},circle,fill=none,draw=none,minimum size=2mm,inner sep=0mm]  at (-2.2,0) {};
\node (c) [label={[xshift=-0.08cm, yshift=-0.37cm,]{\footnotesize $x_3$}},circle,fill=none,draw=none,minimum size=2mm,inner sep=0mm]  at (-1.8,-0.7) {};
\node (d) [label={[xshift=0.08cm, yshift=-0.33cm,]{\footnotesize $y_3$}},circle,fill=none,draw=none,minimum size=2mm,inner sep=0mm]  at (1.9,0.6) {};
\node (e) [label={[xshift=0.08cm, yshift=-0.33cm,]{\footnotesize $y_4$}},circle,fill=none,draw=none,minimum size=2mm,inner sep=0mm]  at (1.9,-0.6) {};
\node (i) [label={[xshift=-0.2cm, yshift=-0.17cm,]{\footnotesize $x_1$}},label={[xshift=0.2cm, yshift=-0.17cm,]{\footnotesize $y_1$}},circle,fill=none,draw=none,minimum size=0mm,inner sep=0mm]  at (0,0.6) {};
\node (j) [label={[xshift=-0.2cm, yshift=-0.4cm,]{\footnotesize $x_2$}},label={[xshift=0.2cm, yshift=-0.45
cm,]{\footnotesize $y_2$}},circle,fill=none,draw=none,minimum size=2mm,inner sep=0mm]  at (0,-0.6) {};
\path[out=-50,in=230] (f) edge (g);
\path[out=50,in=130] (f) edge (g);
\draw (f)--(a);
\draw (f)--(b);
\draw (f)--(c);
\draw (g)--(d);
\draw (g)--(e);
\draw (0,0.6)--(0,0.4);
\draw (0,-0.6)--(0,-0.4);
\end{tikzpicture}
\end{center}

The graph  ${\cal T}=\{a(x_1,x_2,x_3,x_4,x_5),b(y_1,y_2,y_3,y_4),c(z_1,z_2,z_3);\sigma\},$ where  $\sigma(x_5)=y_{2}$, $\sigma(y_{3})=z_1$ and $\sigma$ is identity otherwise, is an unrooted tree. It can be visualized as 
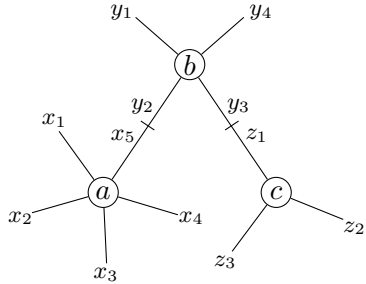
\begin{center}
\begin{tikzpicture}
 \node (f) [circle,fill=none,draw=black,minimum size=4mm,inner sep=0.1mm]  at (-1.15,0) {$a$};
\node (g) [circle,fill=none,draw=black,minimum size=4mm,inner sep=0.1mm]  at (0,1.7) {$b$};
\node (h) [circle,fill=none,draw=black,minimum size=4mm,inner sep=0.1mm]  at (1.15,0) {$c$};
\node (a) [label={[xshift=0cm, yshift=-0.27cm]{\footnotesize $x_1$}},circle,fill=none,draw=none,minimum size=2mm,inner sep=0mm]  at (-1.8,0.9) {};
\node (b) [label={[xshift=-0.05cm, yshift=-0.37cm,]{\footnotesize $x_2$}},circle,fill=none,draw=none,minimum size=2mm,inner sep=0mm]  at (-2.2,-0.3) {};
\node (c) [label={[xshift=-0.015cm, yshift=-0.37cm,]{\footnotesize $x_3$}},circle,fill=none,draw=none,minimum size=2mm,inner sep=0mm]  at (-1.1,-1.05) {};
\node (d) [label={[xshift=-0.1cm, yshift=-0.35cm,]{\footnotesize $y_1$}},circle,fill=none,draw=none,minimum size=2mm,inner sep=0mm]  at (-0.8,2.4) {};
\node (e) [label={[xshift=0.15cm, yshift=-0.35cm,]{\footnotesize $y_4$}},circle,fill=none,draw=none,minimum size=2mm,inner sep=0mm]  at (0.8,2.4) {};
\node (i) [label={[xshift=0.17cm, yshift=-0.27cm,]{\footnotesize $x_4$}},circle,fill=none,draw=none,minimum size=0mm,inner sep=0mm]  at (-0.15,-0.3) {};
\node (i1) [label={[xshift=-0.26cm, yshift=-0.125cm,]{\footnotesize $x_5$}},label={[xshift=0.0cm, yshift=0.25cm,]{\footnotesize $y_2$}},circle,fill=none,draw=none,minimum size=0mm,inner sep=0mm]  at (-0.625,0.65) {};
\node (i2) [label={[xshift=0.025cm, yshift=0.25cm,]{\footnotesize $y_3$}},label={[xshift=0.265cm, yshift=-0.12cm,]{\footnotesize $z_1$}},circle,fill=none,draw=none,minimum size=0mm,inner sep=0mm]  at (0.625,0.65) {};
\node (k) [label={[xshift=0.05cm, yshift=-0.4cm,]{\footnotesize $z_2$}},circle,fill=none,draw=none,minimum size=2mm,inner sep=0mm]  at (2.14,-0.4) {};
\node (l) [label={[xshift=-0.03cm, yshift=-0.38cm,]{\footnotesize $z_3$}},circle,fill=none,draw=none,minimum size=2mm,inner sep=0mm]  at (0.5,-0.87) {};
\draw (h)--(k);
\draw (h)--(l);
\draw (f)--(g);
\draw (f)--(i);
\draw (g)--(h);
\draw (f)--(a);
\draw (f)--(b);
\draw (f)--(c);
\draw (g)--(d);
\draw (g)--(e);
\draw (0.65,0.95)--(0.47,0.83);
\draw (-0.65,0.95)--(-0.47,0.83);
\end{tikzpicture} \\[-0.4cm] \hfill $\square$
\end{center}
 \end{example}
\indent Let ${\cal T}$, ${\cal T}_1$ and ${\cal T}_2$  be   unrooted trees with involutions $\sigma$, $\sigma_1$ and $\sigma_2$, respectively. We say that ${\cal T}_1$ and ${\cal T}_2$ make a decomposition of ${\cal T}$ if ${\it Cor}({\cal T}_1)\cap {\it Cor}({\cal T}_2)=\emptyset$, ${\it Cor}({\cal T}_1)\cup {\it Cor}({\cal T}_2)={\it Cor}({\cal T})$, and there exist $x\in {\it FV}({\cal T}_1)$ and $y\in {\it FV}({\cal T}_2)$ such that   
\[
    \sigma(z) = \left\{\begin{array}{ll}
        \sigma_1(z), & \text{if } z\in V({\cal T}_1)\backslash\{x\}\\
        \sigma_2(z), & \text{if } z\in V({\cal T}_2)\backslash\{y\}\\
        y, & \text{if } z=x \\
x, & \text{if } z=y.
        \end{array}  \right.
  \]
We write ${\cal T}=\{{\cal T}_1\,(xy)\,{\cal T}_2\}$ (or, equivalently, ${\cal T}=\{{\cal T}_2\,(yx)\,{\cal T}_1\}$).  \\
\indent The set of  {\em  subtrees}  of a tree ${\cal T}$ {with involution} ${\sigma}$ is obtained by the following recursive definiton:
\begin{itemize}
\item[$\diamond$] if $a(x_1,\dots,x_n)\in {\it Cor}({\cal T})$, then $\{a(x_1,\dots,x_n);{\it id}_{X}\}$, where $X=\{x_1,\dots,x_n\}$, is a subtree of ${\cal T}$, to which we refer simply as $a$,
\item[$\diamond$] if trees ${\cal T}_1$ and ${\cal T}_2$, with involutions $\sigma_1$ and $\sigma_2$, respectively, are subtrees of ${\cal T}$,
and if $x\in FV({\cal T}_1)$ and $y\in FV({\cal T}_2)$ are such that $\sigma(x)=y$, then $\{{\cal T}_2\,(yx)\,{\cal T}_1\}$ 
is a subtree of ${\cal T}$.
\end{itemize}

\indent We shall denote with   $\underline{\tt{T}}_{\underline{\EuScript C}}$ (resp. $\underline{\tt{T}}_{\underline{\EuScript C}}(X)$) the collection of  unrooted trees  whose corollas belong to $P_{\underline{\EuScript C}}$ (resp. whose corollas belong to $P_{\underline{\EuScript C}}$ and whose set of free variables is $X$). 
 \subsubsection{A tree-wise representation of the  terms of $\underline{\tt Free}_{\EuScript C}$ }\label{trrrre}

\indent We   introduce the syntax of parenthesized words generated by $P_{\underline{\EuScript C}}$, as 
\begin{center}
\mybox{
$w::=\underline{a}\,|\, (ww)$ }\end{center} where $a\in P_{\underline{\EuScript C}}$. We shall denote the collection of all terms obtained in this way by ${\tt PWords}_{\underline{\EuScript C}}$. \\[0.1cm]
\indent   For an unrooted tree ${\cal T}$, we next introduce the ${\cal T}$-{\em admissibility} predicate on ${\tt PWords}_{\underline{\EuScript C}}$. Intuitively, $w$  is ${\cal T}$-admissible if it represents a gradual composition of the corollas of  ${\cal T}$. Formally:
 
 \begin{itemize}
\item[$\diamond$] $\underline{a}$ is  ${\cal T}$-admissible if ${\it Cor}({\cal T})=\{a\}$,  and
\item[$\diamond$]  if ${\cal T}=\{{\cal T}_1\,(xy)\,{\cal T}_2\}$, ${w}_1$ is ${\cal T}_1$-admissible and ${w}_2$ is  ${\cal T}_2$-admissible,  then
 $({w}_1{w}_2)$  is ${\cal T}$-admissible.
\end{itemize}
We shall denote the set of all ${\cal T}$-admissible terms of ${\tt PWords}_{\underline{\EuScript C}}$ with $A({\cal T})$.
\begin{rem} 
Notice that, if ${w}$ is ${\cal T}$-admissible, then, since all the corollas of  ${\cal T}$ are mutually distinct,  ${w}$  does not contain repetitions of letters from  $P_{\underline{\EuScript C}}$. \\
\indent A parenthesized word   can   be admissible with respect to more than one unrooted tree. In the second clause   above, $(w_1w_2)$ is admissible with respect to {\em any} tree  formed by ${\cal T}_1$ and ${\cal T}_2$.
\end{rem}

\indent We introduce the syntax of {\em unrooted trees with grafting data induced by $\underline{\EuScript C}$}, denoted by $\underline{\tt T}^+_{\underline{\EuScript C}}$, as follows. The collection of object terms of $\underline{\tt T}^+_{\underline{\EuScript C}}$   is obtained by  combining the formalism $\underline{\tt T}_{\underline{\EuScript C}}$ and the syntax ${\tt PWords}_{\underline{\EuScript C}}$, by means of the ${\cal T}$-admissibility predicate: we take for object terms of $\underline{\tt T}^+_{\underline{\EuScript C}}$  all the pairs $({\cal T},{w})$, typed as 
\begin{center}
\mybox{
$\displaystyle\frac{{\cal T}\in \underline{\tt T}_{\underline{\EuScript C}}(X)\quad w\in {\tt PWords}_{\underline{\EuScript C}} \quad w\in A({\cal T})}{({\cal T},{w}):X}$ }
\end{center}
The raw arrow terms of $\underline{\tt T}^+_{\underline{\EuScript C}}$ are the following: 
\begin{center}\mybox{
$\varphi::=1_{({\cal T},w)}\,\,|\,\,\beta^{x,\underline{x};y,\underline{y}}_{({\cal T}_1,w_1),({\cal T}_2,w_2),({\cal T}_3,w_3)}\,\,|\,\,  {\beta_{({\cal T}_1,w_1),({\cal T}_2,w_2),({\cal T}_3,w_3)}^{x,\underline{x};y,\underline{y}\enspace ^{-1}}}\,\,|\,\, \gamma^{x,y}_{({\cal T}_1,w_1),({\cal T}_2,w_2)}\,\,|\,\, \varphi\circ\varphi\,\,|\,\, \varphi\,{_x\bo_y}\,\varphi$ }
\end{center}

\begin{lem}\label{delta} The object terms  of $\underline{\tt T}^+_{\underline{\EuScript C}}$ are in one-to-one type-preserving correspondence with the object  terms of $\underline{\tt{Free}}_{\underline{\EuScript C}}$.  
The raw arrow terms  of $\underline{\tt T}^+_{\underline{\EuScript C}}$ are in one-to-one correspondence with the raw arrow terms of $\underline{\tt{Free}}_{\underline{\EuScript C}}$.  
\end{lem}
\begin{proof} Straightforward (see Remark \ref{oj}).
\end{proof}

Let $\Delta:{\underline{\tt T}}^+_{\underline{\EuScript C}}\rightarrow \underline{\tt{Free}}_{\underline{\EuScript C}}$ denote the correspondence  from  Lemma \ref{delta}. The type of a raw arrow term $\varphi$ of $\underline{\tt T}^+_{\underline{\EuScript C}}$ is then determined as follows: if $\Delta(\varphi):{W}_1\rightarrow {W}_2$, then $\varphi:\Delta^{-1}({W}_1)\rightarrow \Delta^{-1}({W}_2)$.

\begin{rem}\label{oh}
Note that the source and the target object of each arrow term $\varphi$ of $\underline{\tt T}^+_{\underline{\EuScript C}}$ share the same underlying unrooted tree, i.e., that we always have $\varphi:({\cal T},w_1)\rightarrow({\cal T},w_2)$, for some ${\cal T},w_1,w_2$. For example, the type of the arrow term $\beta^{x,\underline{x};y,\underline{y}}_{({\cal T}_1,w_1),({\cal T}_2,w_2),({\cal T}_3,w_3)}$ is $$\beta^{x,\underline{x};y,\underline{y}}_{({\cal T}_1,w_1),({\cal T}_2,w_2),({\cal T}_3,w_3)}:({\cal T},(w_1w_2)w_3\,)\rightarrow ({\cal T},w_1(w_2w_3)),$$
where ${\cal T}=\{\{{\cal T}_1\,(x\underline x)\,{\cal T}_2\}\,(y\underline{y})\,{\cal T}_3\}$.
\end{rem}
 
We define the interpretation function $\lfloor-\rfloor_X:\underline{\tt T}^+_{\underline{\EuScript C}}(X)\rightarrow {\EuScript C}(X)$  to be the composition $[-]_{X}\circ\Delta_X$. The following lemma is an immediate consequence of the definition of $\lfloor-\rfloor_X$.
\begin{lem}\label{int_delta}
For   arbitrary object term $W$ and arrow term $\varphi$ of ${\underline{\tt{Free}}}_{\underline{\EuScript C}}(X)$, the  equalities   $$[W]_{X}=\lfloor \Delta_X^{-1}(W) \rfloor_X\enspace\mbox{ and }\enspace [\varphi]_{X}=\lfloor \Delta_X^{-1}(\varphi) \rfloor_X$$ hold in ${\EuScript C}(X)$.
\end{lem}

\subsubsection{``Rooting'' the syntax $\underline{\tt{T}}^+_{\underline{\EuScript C}}$}\label{rooting}
We now introduce the syntax of {\em rooted  trees with grafting data induced by $\underline{\EuScript C}$}, denoted by ${\tt r\underline{T}}^+_{\underline{\EuScript C}}$.\\[0.1cm]
\indent For a pair $({\cal T}, x)$, where ${\cal T}\in \underline{\tt{T}}_{\underline{\EuScript C}}(X)$ and $x\in X$,  we first introduce the $({\cal T}, x)$-{\em admissibility} predicate on ${\tt PWords}_{\underline{\EuScript C}}$:
\begin{itemize}
\item[$\diamond$] $\underline{a}$ is $({\cal T}, x)$-admissible if ${\it Cor}({\cal T})=\{a\}$, and
\item[$\diamond$] if  ${\cal T}=\{{\cal T}_1\,(zy)\,{\cal T}_2\}$,  $x\in {\it FV}({\cal T}_1)$ (resp. $x\in {\it FV}({\cal T}_2)$), $w_1$ is $({\cal T}_1,x)$-admissible (resp. $w_1$ is $({\cal T}_2,x)$-admissible) and $w_2$ is $({\cal T}_2,y)$-admissible (resp. $w_2$ is $({\cal T}_1,z)$-admissible), then $(w_1w_2)$  is $({\cal T},x)$-admissible. 
\end{itemize} 
We shall denote the set of all $({\cal T},x)$-admissible terms of ${\tt PWords}_{\underline{\EuScript C}}$ by $A({\cal T},x)$.\\[0.1cm] \indent Intuitively,  $w$ is $({\cal T}, x)$-admissible if it is ${\cal T}$-admissible and it is an   operadic word with respect to the  rooted tree determined by considering $x$ as the root of ${\cal T}$. As a matter of fact, $({\cal T},w)$ enjoys  the following normalization property, inherent to (formal terms which describe) {operadic} operations:  all $\beta^{-1}$-reduction sequences starting from   $({\cal T}, w)$ end with an object term $({\cal T},w')$, such that all  pairs of parentheses of $w'$ are   associated to the left. \\[0.1cm]
\indent The  object terms of ${\tt r\underline{T}}^+_{\underline{\EuScript C}}$ are  triplets $({\cal T},x,w)$, typed as  
\begin{center}
\mybox{
$\displaystyle\frac{{\cal T}\in {\underline{\tt T}}^+_{\underline{\EuScript C}}(X)\quad  \enspace x\in X\quad  \enspace w\in A({\cal T},x)}{({\cal T},x,w):X}$ }
\end{center}

\indent The class of arrow terms of ${\tt r\underline{T}}^+_{\underline{\EuScript C}}$ is obtained from raw terms  
 \begin{center}\mybox{
$\displaystyle \chi::=\left\{\begin{array}{l}
1_{({\cal T},x,w)}\,\,|\,\,\beta^{z;y}_{({\cal T}_1,x,w_1),({\cal T}_2,\underline{z},w_2),({\cal T}_3,{\underline{y}},w_3)}\,\,|\,\,  {\beta_{({\cal T}_1,x,w_1),({\cal T}_2,\underline{z},w_2),({\cal T}_3,{\underline{y}},w_3)}^{z;y\enspace ^{-1}}}   \\[0.4cm]  
 \theta^{z;y}_{({\cal T}_1,x,w_1),({\cal T}_2,\underline{z},w_2),({\cal T}_3,{\underline{y}},w_3)}\,\,|\,\, \chi\circ\chi\,\,|\,\, \chi\,{_x\bo_y}\,\chi\end{array} \right. $}\end{center}
by specifying typing rules such as:
\begin{center}
\mybox{
\begin{tabular}{c}

{\small $\displaystyle\frac{{\cal T}=\{\{{\cal T}_1\,(z\underline z)\,{\cal T}_2\}\,(y\underline{y})\,{\cal T}_3\}\quad y\in {\it FV}({\cal T}_1)\quad x\in X\cap {\it FV}({\cal T}_1)}{\theta^{z;y}_{({\cal T}_1,x,w_1),({\cal T}_2,\underline{z},w_2),({\cal T}_3,{\underline{y}},w_3)}:({\cal T},x,(w_1w_2)w_3\,)\rightarrow ({\cal T},x,(w_1w_3)w_2\,)}$}\\[0.7cm]

{\small $\displaystyle\frac{\chi_1:({\cal T}_1,x,w_1)\rightarrow ({\cal T}_1,x,{w'}_1)\enspace \chi_2:({\cal T}_2,y,w_2)\rightarrow ({\cal T}_2,y,{w'}_2)\enspace z\in{\it FV}({\cal T}_1) \enspace z\neq x}{\chi_1 \, {_z\bo_y}\, \chi_2: (\{{\cal T}_1\,(zy)\,{\cal T}_2\},x,(w_1w_2)) \rightarrow (\{{\cal T}_1\,(zy)\,{\cal T}_2\},x,({w'}_1{w'}_2))}$} 
\end{tabular}
}
\end{center}

\indent Notice that, for an object term  $({\cal T},x,w)$ of ${\tt r\underline{T}}^{+}_{\underline{\EuScript C}}(X)$, the choice of  $x\in X$   as the  root of   ${\cal T}$  determines the roots of all subtrees of ${\cal T}$, and, in particular, of all corollas of ${\cal T}$. In other words, this choice allows us to speak about the  inputs   and the output   of any subtree of ${\cal T}$.  Formally, for a subtree ${\cal S}$ of ${\cal T}$ and a variable $x\in {\it FV}({\cal T})$, we define 
\begin{itemize}
\item[$\diamond$]   ${\tt in}_{({\cal T},x)}({\cal T})={\it FV}({\cal T})\backslash\{x\}$ and ${\tt out}_{({\cal T},x)}({\cal T})=x$,
\item[$\diamond$] if  ${\cal S}\neq {\cal T}$, if $a\in {\it Cor}({\cal T})$ is such that $x\in {\it FV}(a)$,   if $c\in {\it Cor}({\cal S})$ is the corolla of ${\cal S}$ with the smallest distance from   $a$, and if $p$ is the sequence of half-edges from $c$ to $a$, then ${\tt in}_{({\cal T},x)}({\cal S})={\it FV}({\cal S})\backslash\{z\}$, where $z\in {\it FV}(c)\cap p$, and ${\tt out}_{({\cal T},x)}({\cal S})=z$.
\end{itemize}

\begin{example}\label{exj} For  the unrooted tree ${\cal T}$ from Example \ref{ex2}, the choice of, say,  $y_4\in X$,  turns ${\cal T}$ into a rooted tree, which can be visualized as 
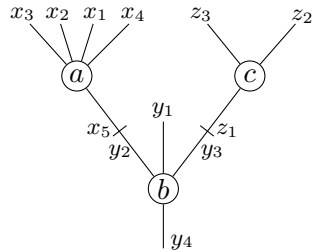
\begin{center}
\begin{tikzpicture}
 \node (f) [circle,fill=none,draw=black,minimum size=4mm,inner sep=0.1mm]  at (-1.15,0) {$a$};
\node (g) [circle,fill=none,draw=black,minimum size=4mm,inner sep=0.1mm]  at (0,-1.5) {$b$};
\node (h) [circle,fill=none,draw=black,minimum size=4mm,inner sep=0.1mm]  at (1.15,0) {$c$};
\node (a) [label={[xshift=0cm, yshift=-0.3cm]{\footnotesize $x_1$}},rectangle,fill=none,draw=none,minimum size=2mm,inner sep=0mm]  at (-0.9,0.8) {};
\node (b) [label={[xshift=0cm, yshift=-0.3cm,]{\footnotesize $x_2$}},rectangle,fill=none,draw=none,minimum size=2mm,inner sep=0mm]  at (-1.4,0.8) {};
\node (c) [label={[xshift=0cm, yshift=-0.3cm,]{\footnotesize $x_3$}},rectangle,fill=none,draw=none,minimum size=2mm,inner sep=0mm]  at (-1.87,0.8) {};
\node (d) [label={[xshift=0cm, yshift=-0.3cm,]{\footnotesize $y_1$}},circle,fill=none,draw=none,minimum size=2mm,inner sep=0mm]  at (0,-0.5) {};
\node (e) [label={[xshift=0.25cm, yshift=-0.15cm,]{\footnotesize $y_4$}},circle,fill=none,draw=none,minimum size=2mm,inner sep=0mm]  at (0,-2.4) {};
\node (i) [label={[xshift=0.05cm, yshift=-0.11cm,]{\footnotesize $x_4$}},rectangle,fill=none,draw=none,minimum size=0mm,inner sep=0mm]  at (-0.45,0.7) {};
\node (i1) [label={[xshift=-0.225cm, yshift=-0.3cm,]{\footnotesize $x_5$}},label={[xshift=0.05cm, yshift=-0.6cm,]{\footnotesize $y_2$}},circle,fill=none,draw=none,minimum size=0mm,inner sep=0mm]  at (-0.625,-0.65) {};
\node (i2) [label={[xshift=0.025cm, yshift=-0.6cm,]{\footnotesize $y_3$}},label={[xshift=0.225cm, yshift=-0.3cm,]{\footnotesize $z_1$}},circle,fill=none,draw=none,minimum size=0mm,inner sep=0mm]  at (0.625,-0.65) {};
\node (k) [label={[xshift=0.0cm, yshift=-0.335cm,]{\footnotesize $z_2$}},rectangle,fill=none,draw=none,minimum size=2mm,inner sep=0mm]  at (1.85,0.8) {};
\node (l) [label={[xshift=0cm, yshift=-0.31cm,]{\footnotesize $z_3$}},rectangle,fill=none,draw=none,minimum size=2mm,inner sep=0mm]  at (0.5,0.8) {};
\draw (h)--(k);
\draw (h)--(l);
\draw (f)--(g);
\draw (f)--(i);
\draw (g)--(h);
\draw (f)--(a);
\draw (f)--(b);
\draw (f)--(c);
\draw (g)--(d);
\draw (g)--(e);
\draw (0.67,-0.8)--(0.49,-0.66);
\draw (-0.67,-0.8)--(-0.49,-0.66);
\end{tikzpicture}
\end{center}
We have  
$$\begin{array}{ll}
{\tt in}_{({\cal T},y_4)}(b)=\{y_1,y_2,y_3\} & {\tt out}_{({\cal T},y_4)}(b)=y_4\\
{\tt in}_{({\cal T},y_4)}(a)=\{x_1,x_2,x_3,x_4\} & {\tt out}_{({\cal T},y_4)}(a)= x_5\ \\ 
{\tt in}_{({\cal T},y_4)}(c)=\{z_2,z_3\} & {\tt out}_{({\cal T},y_4)}(c)=z_1\:.
\end{array}$$
 Observe that, among all paranthesized words admissible with respect to ${\cal T}$, only $(\underline{b}\underline{a})\underline{c}
$ and $(\underline{b}\underline{c})\underline{a}$ are operadic, relative to the choice of $y_4$ as the root of ${\cal T}$. \hfill$\square$
\end{example}
\subsubsection{The interpretation of  ${\tt r\underline{T}}^+_{\underline{\EuScript C}}$ in ${\EuScript C}$}\label{int3}
\indent  The interpretation function $$\lceil-\rceil_X:{\tt r\underline{T}}^+_{\underline{\EuScript C}}(X)\rightarrow {\EuScript C}(X)$$   
is defined as the composition $\lfloor-\rfloor_X\circ F_{\underline{\EuScript C}}(X)$, where  $F_{\underline{\EuScript C}}(X):{\tt r\underline{T}}^+_{\underline{\EuScript C}}(X)\rightarrow {\tt \underline{T}}^+_{\underline{\EuScript C}}(X)$ forgets the root of object terms of ${\tt r\underline{T}}^+_{\underline{\EuScript C}}(X)$ and sends the arrow term $\theta^{z;y}_{({\cal T}_1,x,w_1),({\cal T}_2,\underline{z},w_2),({\cal T}_3,{\underline{y}},w_3)}$ to the  composition $$\gamma^{\underline{x},x}_{({\cal T}_2,w_2),(\{{\cal T}_1 (y\underline{y}){\cal T}_3\},w_1w_3)}\circ\beta^{\underline{x},x;y,\underline{y}}_{({\cal T}_2,w_2),({\cal T}_1,w_1),({\cal T}_3,w_3)}\circ(\gamma^{x,\underline{x}}_{({\cal T}_1,w_1),({\cal T}_2,w_2)}\, {_y\bo_{\underline{y}}}\, 1_{({\cal T}_3,w_3)}),$$ while being defined in the canonical  way on the remaining arrow terms.

\begin{rem}   
Notice that  $\lceil\chi\rceil_X$ is an arrow in ${\EuScript C}(X)$ all of whose instances of the isomorphism $\gamma$  are  ``hidden" by using explicitly the abbreviation $\vartheta$. In other words, the semantics of arrow terms of ${\tt r\underline{T}}^+_{\underline{\EuScript C}}$ is what distinguishes $\beta\vartheta$-arrows of ${\EuScript C}(X)$.
\end{rem}
 
\subsubsection{The second reduction}
\indent We define the familly of {second reduction functions} $${\tt{Red_2}}(X,x):{\tt \underline{T}}^+_{\underline{\EuScript C}}(X)\rightarrow {\tt r\underline{T}}^+_{\underline{\EuScript C}}(X),$$ where $x\in X$, as follows. For the object terms of ${\tt \underline{T}}^+_{\underline{\EuScript C}}(X)$, we set
$${\tt{Red_2}}(X,x)(({\cal T},w))=({\cal T},x,w^{\cdot x}),$$ 
\noindent where  $w^{\cdot x}$ is the $({\cal T},x)$-admissible parenthesized word defined recursively as follows:
\begin{itemize}
\item[$\diamond$] if $w=\underline{a}$, then $w^{\cdot x}=\underline{a}$,
\item[$\diamond$] if ${\cal T}=\{{\cal T}_1\,(x_1x_2)\,{\cal T}_2\}$, $w=(w_1w_2)$, $w_i\in A({\cal T}_i)$ and $({\cal T}_i,w_i):X_i$ for $i=1,2$, then
\begin{itemize}
\item if $x\in X_1$, then $w^{\cdot x}=(w_1^{\cdot x}w_2^{\cdot x_2})$, 
\item  if $x\in X_2$, then  $w^{\cdot x}=(w_2^{\cdot x}w_1^{\cdot x_1})$.
\end{itemize}
\end{itemize}
\indent Moreover, the successive commutations which transform $w$ into the operadic word $w^{\cdot x}$  are   witnessed in ${\tt \underline{T}}^+_{\underline{\EuScript C}}$    by the arrow term
 $$\kappa_{({{\cal T}},w,x)}:({\cal  T},w)\rightarrow ({\cal T}, w^{\cdot x}),$$ 
 \noindent
 defined recursively as follows:
  \begin{itemize}
\item[$\diamond$] if $w=\underline{a}$, then  $\kappa_{({{\cal T}},w,x)}=1_{({\cal T},w)}$,
\item[$\diamond$] if ${\cal T}=\{{\cal T}_1(x_1x_2){\cal T}_2\}$, $w=(w_1w_2)$, $w_i\in A({\cal T}_i)$ and $({\cal T}_i,w_i):X_i$ for $i=1,2$,  
then
 \begin{itemize}
 \item if $x\in X_1$, then $\kappa_{({{\cal T}},w,x)}=\kappa_{({{\cal T}_1},w_1,x)}\,{_{x_1}\!\bo_{x_2}}\,\kappa_{({{\cal T}_2},w_2,x_2)}$,
 \item  if $x\in X_2$, then $\kappa_{({{\cal T}},w,x)}=(\kappa_{({{\cal T}_2},w_2,x)}\,{_{x_2}\!\bo_{x_1}}\,\kappa_{({{\cal  T}_1},w_1,x_1)})\circ \gamma^{x_1,x_2}_{({\cal T}_1,w_1),({\cal T}_2,w_2)}$.
 \end{itemize}
 \end{itemize} 
 \begin{rem}\label{notasection}
Note that the second reduction functions ${\tt{Red_2}}(X,x)$ are not sections of  the forgetful functions $F_{\underline{\EuScript C}}(X)$ introduced in Section \ref{int3}. 
\end{rem}

Before we   rigorously define the second reduction of arrow terms, we illustrate the idea behind it with a toy example. 
\begin{example}\label{ex3} Consider the object term $({\cal T}, (\underline{a}\,\underline{b})\underline{c}):X$, where  ${\cal T}$ is defined as in Example \ref{ex2}. The arrow term $$\beta^{x_i,y_{j_1};y_{j_2},z_l}_{({\cal T}_1,\underline{a}),({\cal T}_2,\underline{b}),({\cal T}_3,\underline{c})}:({\cal T}, (\underline{a}\,\underline{b})\underline{c})\rightarrow ({\cal T}, \underline{a} (\underline{b}\,\underline{c}))$$
where ${\cal T}_1$,  ${\cal T}_2$ and ${\cal T}_3$ are the subtrees of ${\cal T}$ determined by corollas $a$, $b$ and $c$, respectively, is well-typed and, by choosing $y_4\in X$ (as we did in Example \ref{exj}), we have  ${\tt Red_2}(X,y_{4})(({\cal T},(\underline{a}\,\underline{b})\underline{c}))=({\cal T},y_{4},(\underline{b}\,\underline{a})\underline{c})$  and ${\tt Red_2}(X,y_{4})(({\cal T},\underline{a}(\underline{b}\,\underline{c})))=({\cal T},y_{4},(\underline{b}\,\underline{c})\underline{a})$. For the two reductions of object terms, the arrow term $$\theta^{y_2;y_3}_{({\cal T}_2,y_4,\underline{b}),({\cal T}_1,x_5,\underline{a}),({\cal T}_3,z_1,\underline{c})}:({\cal T},y_{4},(\underline{b}\,\underline{a})\underline{c})\rightarrow ({\cal T},y_{4},(\underline{b}\,\underline{c})\underline{a})$$ is well-typed and it will be exactly the second reduction of $\beta^{x_i,y_{j_1};y_{j_2},z_l}_{({\cal T}_1,\underline{a}),({\cal T}_2,\underline{b}),({\cal T}_3,\underline{c})}$.\hfill $\square$
\end{example}
 
Formally,  for an arrow term $\varphi:({\cal T},u)\rightarrow ({\cal T},v)$ of ${\tt \underline{T}}^+_{\underline{\EuScript C}}(X)$,    $${\tt Red_2}(X,x)(\varphi):{\tt Red_2}(X,x)(({\cal T},u))\rightarrow {\tt Red_2}(X,x)(({\cal T},v))$$ is defined recursively as follows:
\begin{itemize}
\item[$\diamond$] ${\tt{Red_2}}(X,x)(1_{({\cal T},w)})=1_{{\tt Red}_2(X,x)(({\cal T},w))}$, 
\item[$\diamond$] if $\varphi=\beta^{z,\underline{z};y,\underline{y}}_{({\cal T}_1,w),({\cal T}_2,w_2),({\cal T}_3,w_3)}$, where $({\cal T}_1,w_1):X_i$, and 
\begin{itemize}
\item if $x\in X_1$, then ${\tt{Red_2}}(X,x)(\varphi)=\beta^{z;y}_{{\tt{Red_2}}(X_1,x)(({\cal T}_1,w_1)),{\tt{Red_2}}(X_2,\underline{z})(({\cal T}_2,w_2)),{\tt{Red_2}}(X_3,\underline{y})(({\cal T}_3,w_3))},$
\item if $x\in X_2$, then 
${\tt{Red_2}}(X,x)(\varphi)=\theta^{\underline{z};y}_{{\tt{Red_2}}(X_2,x)(({\cal T}_2,w_2)),{\tt{Red_2}}(X_1,z)(({\cal T}_1,w_1)),{\tt{Red_2}}(X_3,\underline{y})(({\cal T}_3,w_3))},$
\item if $x\in X_3$, then ${\tt{Red_2}}(X,x)(\varphi)=\beta^{\underline{y};\underline{z}\, ^{_{-1}}}_{{\tt{Red_2}}(X_3,x)(({\cal T}_3,w_3)),{\tt{Red_2}}(X_2,y)(({\cal T}_2,w_2)),{\tt{Red_2}}(X_1,z)(({\cal T}_1,w_1))},$
\end{itemize}
\item[$\diamond$] if $\varphi=\beta^{z,\underline{z};y,\underline{y}\,^{_{-1}}}_{({\cal T}_1,w),({\cal T}_2,w_2),({\cal T}_3,w_3)}$, where $({\cal T}_1,w_1):X_i$, and 
 \begin{itemize}
\item if $x\in X_1$, then ${\tt{Red_2}}(X,x)(\varphi)=\beta^{z;y\,^{_{-1}}}_{{\tt{Red_2}}(X_1,x)(({\cal T}_1,w_1)),{\tt{Red_2}}(X_2,\underline{z})(({\cal T}_2,w_2)),{\tt{Red_2}}(X_3,\underline{y})(({\cal T}_3,w_3))},$
\item if $x\in X_2$, then 
${\tt{Red_2}}(X,x)(\varphi)=\theta^{y;\underline{z}}_{{\tt{Red_2}}(X_2,x)(({\cal T}_2,w_2)),{\tt{Red_2}}(X_3,\underline{y})(({\cal T}_3,w_3)),{\tt{Red_2}}(X_1,z)(({\cal T}_1,w_1))},$
\item if $x\in X_3$, then ${\tt{Red_2}}(X,x)(\varphi)=\beta^{\underline{y};\underline{z}}_{{\tt{Red_2}}(X_3,x)(({\cal T}_3,w_3)),{\tt{Red_2}}(X_2,y)(({\cal T}_2,w_2)),{\tt{Red_2}}(X_1,z)(({\cal T}_1,w_1))},$
\end{itemize}
\item[$\diamond$] if $\varphi=\gamma^{z,y}_{({\cal T}_1,w_1),({\cal T}_2,w_2)}$,  then  ${\tt{Red_2}}(X,x)(\varphi)=1_{{\tt Red}_2(X,x)((\{{\cal T}_1(zy){\cal T}_2\},w_1w_2))}$
\item[$\diamond$] if $\varphi=\varphi_2\circ\varphi_1$, then ${\tt Red_2}(X,x)(\varphi)={\tt{Red_2}}(X,x)(\varphi_2)\circ {\tt{Red_2}}(X,x)(\varphi_1)$,
\item[$\diamond$] if $\varphi={\varphi_1}\,{_z\bo_y}\,{\varphi_2}$, where $\varphi_1:({\cal T}_1,w_1)\rightarrow ({\cal T}'_1,w'_1)$, $\varphi_2:({\cal T}_2,w_2)\rightarrow ({\cal T}'_2,w'_2)$ and $({\cal T}_i,w_i):X_i$,   then
\begin{itemize}
\item if $x\in X_1$, then ${\tt{Red_2}}(X,x)(\varphi)={\tt{Red_2}}(X_1,x)(\varphi_1)\,{_z\bo_y}\,{\tt{Red_2}}(X_2,y)(\varphi_2)$,
\item if $x\in X_2$, then ${\tt{Red_2}}(X,x)(\varphi)={\tt{Red_2}}(X_2,x)(\varphi_2)\,{_y\bo_z}\,{\tt{Red_2}}(X_1,z)(\varphi_1)$.
\end{itemize}
\end{itemize}

\indent  The following theorem is the core of  the coherence theorem.  Intuitively, it says that, once the action  of the symmetric group has been removed,  the coherence of non-skeletal cyclic operads can be reduced to the coherence of  non-skeletal operads\footnote{The syntax ${\tt \underline{T}}^+_{\underline{\EuScript C}}$ does {\em not} encode   canonical diagrams of
non-symmetric categorified cyclic operads: non-symmetric
cyclic operads  contain  cyclic actions, while ${\tt \underline{T}}^+_{\underline{\EuScript C}}$ does not encode  any action of the symmetric group. For  definitions of  non-symmetric
cyclic operads, see \cite[Section 3.2]{cgr} and \cite[Sections 1,2,3]{mm}.}. 
As it will be clear from its proof,    \hyperlink{bg-hexagon}{{\tt ($\beta \gamma$-\texttt{hexagon})}} is the key coherence condition that makes this reduction possible. 
\begin{thm}\label{t2}
For an arbitrary   arrow term $\varphi:({\cal T},u)\rightarrow ({\cal T},v)$ of ${\tt \underline{T}}^+_{\underline{\EuScript C}}$, the following equality of interpretations holds: $$ \lfloor\kappa_{({\cal T},v,x)}\rfloor_X\circ\lfloor \varphi \rfloor_X=\lceil{\tt Red}_2(X,x)(\varphi)\rceil_X \circ \lfloor\kappa_{({\cal T},u,x)}\rfloor_X.$$
\end{thm}
\begin{proof} By the definition of   $\lfloor-\rfloor_X$,  the equality of interpretations  that we need to prove is \begin{equation}\label{je}[\Delta_X(\kappa_{({\cal T},v,x)})]_X \circ [\Delta_X({\varphi})]_X= \lceil {\tt{Red_2}}(X,x)(\varphi)\rceil_X\circ  [\Delta_X(\kappa_{({\cal T},u,x)})]_X .\end{equation} The proof goes by induction on the structure of ${\varphi}$. We only treat the most interesting case, i.e., 
$${\varphi}=\beta^{z,\underline{z};y,\underline{y}}_{({\cal T}_1,w_1),({\cal T}_2,w_2),({\cal T}_3,w_3)},$$ where $({\cal T}_i,w_i):X_i$.  We proceed by  case analysis relative to the ``origin'' of $x\in X$.
\begin{itemize}
\item  If $x\in X_1$, then  
 $$\begin{array}{l}{\kappa_{({\cal T},u,x)}}=(\kappa_{({\cal T}_1,w_1,x)}\,{_x\bo_{\underline{x}}}\, \kappa_{({\cal T}_2,w_2,\underline{z})})\,{_y\bo_{\underline{y}}}\, \kappa_{({\cal T}_3,w_3,\underline{y})}\\
  {\kappa_{({\cal T},v,x)}}=\kappa_{({\cal T}_1,w_1,x)}\,{_x\bo_{\underline{x}}}\, (\kappa_{({\cal T}_2,w_2,\underline{z})}\,{_y\bo_{\underline{y}}}\,\kappa_{({\cal T}_3,w_3,\underline{y})}).
  \end{array}$$ 
By introducing the abbreviations
{\footnotesize $$\begin{array}{lll}
\kappa_1=[\Delta_{X_1}(\kappa_{({\cal T}_1,w_1,x)})]_{X_1}& \kappa_2=[\Delta_{X_2}(\kappa_{({\cal T}_2,w_2,\underline{z})})]_{X_2}& \kappa_3=[\Delta_{X_3}(\kappa_{({\cal T}_3,w_3,\underline{y})})]_{X_3}\\[0.15cm]
f_1=[\Delta_{X_1}(({\cal T}_1,w_1))]_{X_1} & f_2=[\Delta_{X_2}(({\cal T}_2,w_2))]_{X_2} &  f_3=[\Delta_{X_3}(({\cal T}_3,w_3))]_{X_3}\\[0.1cm]
f^{\tiny{\bullet} }_1=\lceil {\tt Red_2}(X_1,x)(({\cal T}_1,w_1))\rceil_{X_1} & f^{\bullet}_2=\lceil {\tt Red_2}(X_2,\underline{z})({\cal T}_2,w_2)\rceil_{X_2} &  f^{\bullet}_3=\lceil {\tt Red_2}(X_3,\underline{y})({\cal T}_3,w_3))\rceil_{X_3}
\end{array}$$ }
\noindent
the equality \eqref{je}  translates to the equality  $$\beta^{z,\underline{z};y,\underline{y}}_{f_1^{\bullet },f_2^{\bullet  },f_3^{\bullet} }\circ ((\kappa_1\,{_z\circ_{\underline{z}}}\,\kappa_2)\,{_y\circ_{\underline{y}}}\, \kappa_3)=(\kappa_1\,{_z\circ_{\underline{z}}}\, (\kappa_2\,{_y\circ_{\underline{y}}}\,\kappa_3))\circ \beta^{z,\underline{z};y,\underline{y}}_{f_1,f_2,f_3},$$
which in turn holds by the  naturality of $\beta$.  
 \item If  $x\in X_2$, then     
 $$\begin{array}{l}{\kappa_{({\cal T},u,x)}}=(({\kappa_{({\cal T}_2,w_2,x)}}\,{_{\underline{z}}\bo_z}\, {\kappa_{({\cal T}_1,w_1,z)})}\,{_y\bo_{\underline{y}}}\,\kappa_{({\cal T}_3,w_3,\underline{y})})\circ (\gamma^{z,\underline{z}}_{({\cal T}_1,w_1),({\cal T}_2,w_2)}\,{_y\bo_{\underline{y}}}\, 1_{({\cal T}_3,w_3)})\\
 {\kappa_{({\cal T},v,x)}}=(({\kappa_{({\cal T}_2,w_2,x)}}\,{_y\bo_{\underline{y}}}\, {\kappa_{({\cal T}_3,w_3,\underline{y})}})\,{_{\underline{z}}\bo_z}\,{\kappa_{{({\cal T}_1,w_1,z)}}})\circ (\gamma^{z,\underline{z}}_{({\cal T}_1,w_1),(\{{\cal T}_2(y\underline{y}){\cal T}_3\},w_2w_3)})\:.
 \end{array}$$ 
By introducing the abbreviations
{\footnotesize $$\begin{array}{lll}
\kappa_1=[\Delta_{X_1}(\kappa_{({\cal T}_1,w_1,z)})]_{X_1}&\kappa_2=[\Delta_{X_2}(\kappa_{({\cal T}_2,w_2,x)})]_{X_2}& \kappa_3=[\Delta_{X_3}(\kappa_{({\cal T}_3,w_3,\underline{y})})]_{X_3}\\[0.1cm]
f_1=[\Delta_{X_1}(({\cal T}_1,w_1))]_{X_1} & f_2=[\Delta_{X_2}(({\cal T}_2,w_2))]_{X_2} &  f_3=[\Delta_{X_3}(({\cal T}_3,w_3))]_{X_3}\\[0.1cm]
f_1^{\bullet}=\lceil{\tt Red_2}(X_1,z)(({\cal T}_1,w_1))\rceil_{X_1}& f^{\bullet}_2=\lceil{\tt Red_2}(X_2,x)(({\cal T}_2,w_2))\rceil_{X_2}& f^{\bullet}_3=\lceil{\tt Red_3}(X_3,\underline{y})(({\cal T}_3,w_3))\rceil_{X_3}
\end{array}$$}
\noindent
the equality \eqref{je} can be read from the  (outer part of the) commuting diagram
\begin{center}
\begin{tikzpicture}
    \node (E) at (0,0) {\small $(f_1\,{_z\circ_{\underline{z}}}\,f_2)\,{_y\circ_{\underline{y}}}\,f_3$};
    \node (F) at (0,-2) {\small $(f_2\,{_{\underline{z}}\circ_{z}}\,f_1)\,{_y\circ_{\underline{y}}}\,f_3$};
    \node (A) at (7.5,-2) {\small $(f_2\,{_y\circ_{\underline{y}}}\,f_3)\,{_{\underline{z}}\circ_{z}}\,f_1$};
    \node (G) at (7.5,0) {\small $f_1\,{_z\circ_{\underline{z}}}\,(f_2\,{_y\circ_{\underline{y}}}\,f_3)$};
    \node (F1) at (0,-4) {\small $(f_2^{\bullet}\,{_{\underline{z}}\circ_{z}}\,f_1^{\bullet})\,{_y\circ_{\underline{y}}}\,f_3^{\bullet}$};
    \node (A1) at (7.5,-4) {\small $(f_2^{\bullet}\,{_y\circ_{\underline{y}}}\,f_3^{\bullet})\,{_{\underline{z}}\circ_{z}}\,f_1^{\bullet}$};
    \draw[->] (E)--(G) node [midway,above] {\footnotesize   $\beta^{z,\underline{z};y,\underline{y}}_{f_1,f_2,f_3}$};
    \draw[->] (G)--(A) node [midway,right] {\footnotesize    $\gamma^{z,\underline{z}}_{f_1,f_2\,{_y\circ_{\underline{y}}}\,f_3}$};
  \draw[->] (A)--(A1) node [midway,right] {\footnotesize     $(\kappa_2 \,{_{y}\circ_{\underline{y}}}\, \kappa_3)\,{_{\underline{z}}\circ_{z}}\,\kappa_1$};
    \draw[->] (E)--(F) node [midway,left] {\footnotesize   $\gamma^{z,\underline{z}}_{f_1,f_2} \,{_y\circ_{\underline{y}}}\,1_{f_3}$};
  \draw[->] (F)--(F1) node [midway,left] {\footnotesize    $(\kappa_2 \,{_{\underline{z}}\circ_z}\, \kappa_1)\,{_y\circ_{\underline{y}}}\,\kappa_3$};
  \draw[->] (F)--(A) node [midway,above] {\footnotesize    $\vartheta^{\underline{z},z;y,\underline{y}}_{f_2,f_1,f_3}$};
\draw[->] (F1)--(A1) node [midway,above] {\footnotesize    $\vartheta^{\underline{z},z;y,\underline{y}}_{f_2^{\bullet},f_1^{\bullet},f_3^{\bullet}}$};
   \end{tikzpicture}
\end{center}
in which the upper square commutes by the definition of the isomorphism $\vartheta$  (see \eqref{theta} on Page \pageref{theta}) and the lower square is a naturality diagram for $\vartheta$.
\item If $x\in X_3$, then  
$$\begin{array}{lll}{\kappa_{({\cal T},u,x)}} & = & ({\kappa_{({\cal T}_3,w_3,x)}}\,{_{\underline{y}}\bo_y}\,({\kappa_{({\cal T}_2,w_2,y)}}\,{_{\underline{z}}\bo_z}\,{\kappa_{({\cal T}_1,w_1,z)}}))\circ\\
&& (1_{({\cal T}_3,w_3)}\,{_{\underline{y}}\bo_y}\, \gamma^{z,\underline{z}}_{({\cal T}_1,w_1),({\cal T}_2,w_2)})\circ \gamma^{y,\underline{y}}_{(\{{\cal T}_1(z\underline{z}){\cal T}_2\},w_1w_2),({\cal T}_3,w_3)}\\[0.2cm]
{\kappa_{({\cal T},v,x)}} &= & (({\kappa_{({\cal T}_3,w_3,x)}}\,{_{\underline{y}}\bo_y}\, {\kappa_{({\cal T}_2,w_2,y)}})\,{_{\underline{z}}\bo_z}\, {\kappa_{({\cal T}_1,w_1,z)}})\circ 
\\ && (\gamma_{({\cal T}_2,w_2),({\cal T}_3,w_3)}\,{_{\underline{z}}\bo_z}\, 1_{({\cal T}_1,w_1)})\circ \gamma^{z,\underline{z}}_{({\cal T}_1,w_1),(\{{\cal T}_2(y\underline{y}){\cal T}_3\},w_2w_3)}\:.
\end{array}$$ 

By introducing the abbreviations
{\footnotesize $$\begin{array}{lll}
\kappa_1=[\Delta_{X_1}(\kappa_{({\cal T}_1,w_1,z)})]_{X_1}&\kappa_2=[\Delta_{X_2}(\kappa_{({\cal T}_2,w_2,y)})]_{X_2}& \kappa_3=[\Delta_{X_3}(\kappa_{({\cal T}_3,w_3,x)})]_{X_3}\\[0.1cm]
f_1=[\Delta_{X_1}(({\cal T}_1,w_1))]_{X_1} & f_2=[\Delta_{X_2}(({\cal T}_2,w_2))]_{X_2} &  f_3=[\Delta_{X_3}(({\cal T}_3,w_3))]_{X_3}\\[0.1cm]
f_1^{\bullet}=\lceil{\tt Red_2}(X_1,z)(({\cal T}_1,w_1))\rceil_{X_1}& f^{\bullet}_2=\lceil{\tt Red_2}(X_2,y)(({\cal T}_2,w_2))\rceil_{X_2}& f^{\bullet}_3=\lceil{\tt Red_2}(X_3,x)(({\cal T}_3,w_3))\rceil_{X_3}
\end{array}$$}

\noindent
the equality \eqref{je}   can be read from the  (outer part of the) commuting diagram
\begin{center}
\begin{tikzpicture}
    \node (E) at (0,0) {\small $(f_1\,{_z\circ_{\underline{z}}}\,f_2)\,{_y\circ_{\underline{y}}}\,f_3$};
    \node (F) at (0,-2) {\small $f_3\,{_{\underline{y}}\circ_y}\,(f_1\,{_z\circ_{\underline{z}}}\,f_2)$};
    \node (F2) at (0,-4) {\small $f_3\,{_{\underline{y}}\circ_y}\,(f_2\,{_{\underline{z}}\circ_z}\,f_1)$};
    \node (A) at (7.5,-2) {\small $(f_2\,{_y\circ_{\underline{y}}}\,f_3)\,{_{\underline{z}}\circ_z}\,f_1$};
    \node (A2) at (7.5,-4) {\small $(f_3\,{_{\underline{y}}\circ_y}\,f_2)\,{_{\underline{z}}\circ_z}\,f_1$};
    \node (G) at (7.5,0) {\small $f_1\,{_z\circ_{\underline{z}}}\,(f_2\,{_y\circ_{\underline{y}}}\,f_3)$};
    \node (F1) at (0,-6) {\small $f_3^{\bullet}\,{_{\underline{y}}\circ_y}\,(f_2^{\bullet}\,{_{\underline{z}}\circ_z}\,f_1^{\bullet})$};
    \node (A1) at (7.5,-6) {\small $(f_3^{\bullet}\,{_{\underline{y}}\circ_y}\,f_2^{\bullet})\,{_{\underline{z}}\circ_z}\,f_1^{\bullet}$};
    \draw[->] (E)--(G) node [midway,above] {\footnotesize   $\beta^{z,\underline{z};y,\underline{y}}_{f_1,f_2,f_3}$};
    \draw[->] (G)--(A) node [midway,right] {\footnotesize   $\gamma^{z,\underline{z}}_{f_1,f_2\,{_y\circ_{\underline{y}}}\,f_3}$};
  \draw[->] (A2)--(A1) node [midway,right] {\footnotesize $(\kappa_3\,{_{\underline{y}}\circ_y}\, \kappa_2)\,{_{\underline{z}}\circ_z}\,  \kappa_1$};
    \draw[->] (E)--(F) node [midway,left] {\footnotesize  $\gamma^{y,\underline{y}}_{f_1\,{_z\circ_{\underline{z}}}\,f_2,f_3}$};
  \draw[->] (F)--(F2) node [midway,left] {\footnotesize  $1_{f_3}\,{_{\underline{y}}\circ_y}\, \gamma^{z,\underline{z}}_{f_1,f_2}$};
  \draw[->] (A)--(A2) node [midway,right] {\footnotesize $\gamma^{y,\underline{y}}_{f_2,f_3}\,{_{\underline{z}}\circ_z}\,1_{f_1}$};
  \draw[->] (F2)--(F1) node [midway,left] {\footnotesize  $\kappa_3\,{_{\underline{y}}\circ_y}\, (\kappa_2\,{_{\underline{z}}\circ_z}\, \kappa_1)$};
  \draw[->] (F2)--(A2) node [midway,above] {\footnotesize $\beta^{\underline{y},y;\underline{z},z\enspace -1}_{f_3,f_2,f_1}$};
\draw[->] (F1)--(A1) node [midway,above] {\footnotesize  $\beta^{\underline{y},y;\underline{z},z \enspace -1}_{f^{\bullet}_3,f^{\bullet}_2,f^{\bullet}_1}$};
   \end{tikzpicture}
\end{center}
 in which the upper square commutes  as an instance of \hyperlink{bg-hexagon}{{\tt ($\beta \gamma$-\texttt{hexagon})}} and the bottom square commutes by the naturality of $\beta^{-1}$.
\end{itemize}

\end{proof}
The following result is a direct consequence of Theorem \ref{t2}.
\begin{cor}\label{roknucuse}
For     arrow terms $\varphi_1$ and $\varphi_2$ of the same type in ${\tt \underline{T}}_{\underline{\EuScript C}}^+(X)$, the equality  $\lfloor\varphi_1\rfloor_X=\lfloor\varphi_2\rfloor_X$ follows from the equality $\lceil {\tt Red_2}(X,x)(\varphi_1)\rceil_{X}=\lceil {\tt Red_2}(X,x)(\varphi_2)\rceil_{X}$.
\end{cor}
\subsection{The third reduction: establishing skeletality} Intuitively, in the third reduction we pass from the non-skeletal  to   the skeletal operadic  framework. This will reduce the  problem of commutation of all $\beta\vartheta$-diagrams of ${\EuScript C}(X)$ to the problem of commutation of all diagrams of canonical arrows of  the skeletal non-symmetric categorified operad ${\EuScript O}_{\EuScript C}$, constructed from ${\EuScript C}$ in the appropriate way.
\subsubsection{The skeletal non-symmetric categorified operad ${\EuScript O}_{\EuScript C}$}
\indent Starting from   ${\EuScript C}$, we first define a skeletal non-symmetric categorified operad ${\EuScript O}_{\EuScript C}=\{{\EuScript O}_{\EuScript C}(n)\}_{n\in\mathbb{N}}$, i.e., a weak Cat-operad in the sense of \cite{dp}, as follows.
\begin{itemize}
\item The objects of the category ${\EuScript O}_{\EuScript C}(n)$ are quadruples $(X,x,\sigma,f)$, where $|X|=n+1$, $x\in X$, $f\in {\EuScript C}(X)$  and $\sigma:[n]\rightarrow X\backslash\{x\}$ is a bijection (inducing a total order on $X\backslash\{x\}$). 
\item The morphisms of ${\EuScript O}_{\EuScript C}(n)[(X,x,\sigma,f),(X,x,\sigma,g)]$ are quadruples $(X,x,\sigma,\varphi)$, such that $\varphi$ is a morphism of ${\EuScript C}(X)[f,g]$ (in particular,  ${\EuScript O}_{\EuScript C}(n)[(X,x,\sigma,f),(Y,y,\tau,g)]$ is empty for $(X,x,\sigma)\neq (Y,y,\tau)$). The identity morphism for   $(X,x,\sigma,f)$ is $(X,x,\sigma,1_{f})$.
The composition of morphisms  is canonically induced from the composition of morphisms in ${\EuScript C}(X)$.
\item The composition operation $\circ_i:{\EuScript O}_{\EuScript C}(n)\times {\EuScript O}_{\EuScript C}(m)\rightarrow {\EuScript O}_{\EuScript C}(n+m-1)$ on objects is defined by $$(X,x,\sigma_1,f)\circ_i (Y,y,\sigma_2,g)=(X\cup Y\backslash\{y\},x,\sigma,f\,{_{\sigma_1(i)\circ_{y}}}\,g),$$ and on morphisms by $$(X,x,\sigma_1,\varphi)\circ_i (Y,y,\sigma_2,\psi)=(X\cup Y\backslash\{y\},x,\sigma,\varphi\,{_{\sigma_1(i)\circ_{y}}}\,\psi),$$ where  $\sigma:[n+m-1]\rightarrow X\backslash\{x\}\cup Y\backslash\{y\}$ is a bijection defined by \begin{equation} \label{sigma_skeletal}
\sigma(j) = 
     \begin{cases}
    \sigma_1(j)  &\quad\text{for  } j\in\{1,\dots,i-1\}\\
     \sigma_2(j-i+1) &\quad\text{for  } j\in\{i,\dots, i+m-1\}  \\
   \sigma_1(j-m)  &\quad\text{for  } j\in \{i+m,\dots,n+m-1\}.
     \end{cases}
\end{equation}
 
\item For $\tilde{f}=(X,x,\sigma_1,f)$, $\tilde{g}=(Y,y,\sigma_2,g)$ and $\tilde{h}=(Z,z,\sigma_3,h)$, where $\sigma_1:[n]\rightarrow X\backslash\{x\}$, $\sigma_2:[m]\rightarrow Y\backslash\{y\}$ and $\sigma_3:[k]\rightarrow Z\backslash\{z\}$,  the components $$\beta^{\,i;j}_{\tilde{f},\tilde{g},\tilde{h}}:(\tilde{f}\circ_i \tilde{g})\circ_j \tilde{h}\rightarrow \tilde{f}\circ_i (\tilde{g}\circ_j \tilde{h}) \enspace \mbox{ and }\enspace  \theta^{\,i;k}_{\tilde{f},\tilde{g},\tilde{h}}:(\tilde{f}\circ_i \tilde{g})\circ_k \tilde{h}\rightarrow (\tilde{f}\circ_k \tilde{h})\circ_i \tilde{h}$$ of natural isomorphisms $\beta$ and $\theta$ are distinguished among the morphisms of ${\EuScript O}_{\EuScript C}(n)$ as the quadruples arising  from the appropriate components of $\beta$ and $\vartheta$ of ${\EuScript C}$, as follows: $$\beta^{\,i;j}_{\tilde{f},\tilde{g},\tilde{h}}=(X\cup Y\backslash\{y\}\cup Z\backslash\{z\},x,\sigma,\beta^{\sigma_1(i),y;\sigma_{2}(j),z}_{f,g,h})$$  and
 $$\theta^{\,i;k}_{\tilde{f},\tilde{g},\tilde{h}}=(X\cup Y\backslash\{y\}\cup Z\backslash\{z\},x,\sigma',\vartheta^{\sigma_1(i),y;\sigma_{1}(k),z}_{f,g,h}),$$ 
where $\sigma$ and $\sigma'$ are the bijections induced in the appropriate way from $\sigma_1,\sigma_2$ and $\sigma_3$.
\end{itemize}
We now show that   ${\EuScript O}_{\EuScript C}$ indeed verifies the  axioms of weak Cat-operads from \cite[Section 7]{dp}. 
\begin{lem}
For an arbitrary $n\in\mathbb{N}$, the following equations hold in ${\EuScript O}_{\EuScript C}(n)$: \\[-0.55cm]
 \begin{itemize}
\item[1.] the categorical equations:$$\varphi\circ 1_{\tilde{f}}=\varphi=1_{\tilde{g}}\circ\varphi, \mbox{ for } \varphi:\tilde{f}\rightarrow\tilde{g}, \quad\quad  (\varphi\circ\phi)\circ\psi=\varphi\circ(\phi\circ\psi);$$
\item[2.] the bifunctoriality equations:
$$ 1_{\tilde{f}} \circ_i 1_{\tilde{g}}=1_{\tilde{f}\circ_i \tilde{g}},\quad\quad (\varphi_2\circ\varphi_1)\circ_i (\psi_2\circ\psi_1)=(\varphi_2\circ_i \psi_2)\circ(\varphi_1\circ_i \psi_1);$$
\item[3.] the naturality equations:\\[-0.55cm]
\begin{itemize}
\item[a)] $\beta^{i;j}_{\tilde{f}_2,\tilde{g}_2,\tilde{h}_2}\circ ((\varphi\circ_i\phi)\circ_j\psi)=(\varphi\circ_i(\phi\circ_j\psi)) \circ \beta^{i;j}_{\tilde{f}_1,\tilde{g}_1,\tilde{h}_1}$,\\[-0.45cm]
\item[b)] $\theta^{i;j}_{\tilde{f}_2,\tilde{g}_2,\tilde{h}_2}\circ ((\varphi\circ_i\phi)\circ_j\psi)=((\varphi\circ_j\psi)\circ_i\phi) \circ \theta^{i;j}_{\tilde{f}_1,\tilde{g}_1,\tilde{h}_1}$;
\end{itemize}
\item[4.] the equations concerning inverse isomorphisms:\\[-0.55cm]
\begin{itemize}
\item[a)] $\beta^{i;j\enspace ^{_{-1}}}_{\tilde{f},\tilde{g},\tilde{h}}\circ\beta^{i;j}_{\tilde{f},\tilde{g},\tilde{h}}=1_{(\tilde{f}\circ_i\,\tilde{g})\circ_j\tilde{h}}$,\quad $\beta^{i;j}_{\tilde{f},\tilde{g},\tilde{h}}\circ\beta^{i;j\enspace ^{_{-1}}}_{\tilde{f},\tilde{g},\tilde{h}}=1_{\tilde{f}\circ_i(\tilde{g}\,\circ_j\tilde{h})}$,
\item[b)] $\theta^{j;i}_{\tilde{f},\tilde{h},\tilde{g}}\circ \theta^{i;j}_{\tilde{f},\tilde{g},\tilde{h}}=1_{(\tilde{f}\circ_i\,\tilde{g})\circ_j\tilde{h}}$;
\end{itemize}
\item[5.] the coherence conditions:\\[-0.55cm]
\begin{itemize}
\item[a)] $(1_{\tilde{f}}\circ_{i} \beta^{j;l}_{\tilde{g},\tilde{h},\tilde{k}})\circ\beta^{i;l}_{\tilde{f},\tilde{g}\circ\!_j\tilde{h},\tilde{k}}\circ (\beta^{i;j}_{\tilde{f},\tilde{g},\tilde{h}}\circ_l 1_{\tilde{k}})=\beta^{i;j}_{\tilde{f},\tilde{g},\tilde{h}\circ_l\tilde{k}}\circ\beta^{j;l}_{\tilde{f}\circ_i \tilde{g},\tilde{h},\tilde{k}}$,
\item[b)] $(1_{\tilde{f}}\circ_i\theta^{j;l}_{\tilde{g},\tilde{h},\tilde{k}})\circ\beta^{i;l}_{\tilde{f},\tilde{g}\circ_j\tilde{h},\tilde{k}}\circ(\beta^{i;j}_{\tilde{f},\tilde{g},\tilde{h}}\circ_{l}1_{k})=\beta^{i;j}_{\tilde{f},\tilde{g}\circ_l\tilde{h},\tilde{k}}\circ(\beta^{i;l}_{\tilde{f},\tilde{g},\tilde{k}}\circ_j 1_{\tilde{h}})\circ \theta^{j;l}_{\tilde{f}\circ_i \tilde{g},\tilde{h},\tilde{k}}$,
\item[c)] $\theta^{i;l}_{\tilde{f},\tilde{g}\circ_j{\tilde{h}},\tilde{k}}\circ(\beta^{i;j}_{\tilde{f},\tilde{g},\tilde{h}}\circ_l 1_{\tilde{k}})=\beta^{i;j}_{\tilde{f}\circ_l\tilde{k},\tilde{g},\tilde{h}}\circ(\theta^{i;l}_{\tilde{f},\tilde{g},\tilde{k}}\circ_{j}1_{\tilde{h}})\circ\theta^{j;l}_{\tilde{f}\circ_i \tilde{g},\tilde{h},\tilde{k}}$,
\item[d)] $\theta^{i;j}_{\tilde{f}\circ_l\tilde{k},\tilde{g},\tilde{h}}\circ(\theta^{i;l}_{\tilde{f},\tilde{g},\tilde{k}}\circ_{j}1_{\tilde{h}})\circ\theta^{j;l}_{\tilde{f}\circ_i \tilde{g},\tilde{h},\tilde{k}}=(\theta^{j;l}_{\tilde{f},\tilde{h},\tilde{k}}\circ_{i}1_{\tilde{g}})\circ\theta^{i;l}_{\tilde{f}\circ_j \tilde{h},\tilde{g},\tilde{k}}\circ(\theta^{i;j}_{\tilde{f},\tilde{g},\tilde{h}}\circ_l 1_{\tilde{k}})$.
\end{itemize}
\end{itemize}
\end{lem}
\begin{proof}
The first two groups of equations and the equation 3(a) follow from the appropriate properties of ${\EuScript C}$ (Remark \ref{functoriality}), 3(b) holds by the naturality of $\vartheta$, 4(a)   by the analogous equations for ${\EuScript C}$, 4(b)   by Lemma \ref{thetainverse}.(1), 5(a)   by \hyperlink{b-pentagon}{{\tt ($\beta$-\texttt{pentagon})}}, 5(b) by \hyperlink{bg-decagon}{{\tt{($\beta \gamma$-decagon)}}} and Remark \ref{thetaappears}(2),   5(c)  by Lemma \ref{thetainverse}(2), and   5(d) by Lemma \ref{thetainverse}(3).
\end{proof}
\subsubsection{``Skeletalization'' of the syntax ${\tt r\underline{T}}^{+}_{\underline{\EuScript C}}$} 
In order to be in the position to  apply the coherence result of $\cite{dp}$, we introduce
 a  ``skeletalization'' of the syntax ${\tt r\underline{T}}^{+}_{\underline{\EuScript C}}$. \\[0.1cm]
\indent Let ${\cal T}$ be  an unrooted tree. Suppose that ${\it FV}({\cal T})=X$ and let $x\in X$. For a corolla $c\in {\it Cor}({\cal T})$, such that $|{\tt in}_{({\cal T},x)}(c)|=n$ (see the end of \S \ref{rooting}), we define the set of   {\em skeletalizations of} $c$ {\em (relative to  ${\cal T}$ and  $x$)} as
 $$\Sigma_{({\cal T},x)}(c)={\bf Bij}[\,[n], {\tt in}_{({\cal T},x)}(c)\,].$$ We set $$\Sigma({\cal T},x)=\prod_{c\in {\it Cor}({\cal T})}\Sigma_{({\cal T},x)}(c).$$ We shall  denote the elements of $\Sigma({\cal T},x)$ with $\overrightarrow{\sigma}$.
\begin{rem}
 Notice that $\overrightarrow{\sigma_1}\in \Sigma({\cal T}_1,x)$ and $\overrightarrow{\sigma_2}\in \Sigma({\cal T}_2,y)$ determine  ``by concatenation''  an element of $\overrightarrow{\sigma}\in\Sigma(\{{\cal T}_1\,(zy)\,{\cal T}_2\},x)$, and that, symmetrically, any $\overrightarrow{\sigma}\in \Sigma(\{{\cal T}_1\,(zy)\,{\cal T}_2\},x)$ can  be ``split'' into $\overrightarrow{\sigma_1}\in \Sigma({\cal T}_1,x)$ and $\overrightarrow{\sigma_2}\in \Sigma({\cal T}_2,y)$; we shall write
$\overrightarrow{\sigma}=\overrightarrow{\sigma_1\!\cdot\!\sigma_2}$. 
\end{rem}
\indent The {\em skeletalization of the syntax} ${\tt r\underline{T}}^{+}_{\underline{\EuScript C}}$ is the syntax ${\tt _{sk}r\underline{T}}^{+}_{\underline{\EuScript C}}$, obtained as follows.\\
\indent  The objects terms of ${\tt _{sk}{r\underline{T}}}^{+}_{\underline{\EuScript C}}$  are  quadruples  $({\cal T},x,\overrightarrow{\sigma},w)$, typed by the   rule
 \begin{center}
\mybox{$\displaystyle\frac{{\cal T}\in {\tt \underline{T}}^{+}_{\underline{\EuScript C}}(X)\quad x\in X\quad\overrightarrow{\sigma}\in \Sigma{({\cal T},x)}\quad w\in A({\cal T},x)}{({\cal T},x,\overrightarrow{\sigma},w):X\backslash\{x\}}$} \end{center}

The arrow terms of ${\tt _{sk}{r\underline{T}}}^{+}_{\underline{\EuScript C}}$ are obtained from raw terms  
\begin{center}\mybox{
$\displaystyle \chi::=\left\{\begin{array}{l}
1_{({\cal T},x,\overrightarrow{\sigma},w)}\,\,|\,\,\beta^{z;y}_{({\cal T}_1,x,\overrightarrow{\sigma_1},w_1),({\cal T}_2,\underline{z},\overrightarrow{\sigma_2},w_2),({\cal T}_3,{\underline{y}},\overrightarrow{\sigma_3},w_3)}\,\,|\,\,  {\beta_{({\cal T}_1,x,\overrightarrow{\sigma_1},w_1),({\cal T}_2,\underline{z},\overrightarrow{\sigma_2},w_2),({\cal T}_3,{\underline{y}},\overrightarrow{\sigma_3},w_3)}^{z;y\enspace ^{-1}}}   \\[0.4cm]  
\theta^{z;y}_{({\cal T}_1,x,\overrightarrow{\sigma_1},w_1),({\cal T}_2,\underline{z},\overrightarrow{\sigma_2},w_2),({\cal T}_3,{\underline{y}},\overrightarrow{\sigma_3},w_3)}\,\,|\,\, \chi\circ\chi\,\,|\,\, \chi\,{_z\bo_y}\,\chi\end{array} \right. $}\end{center}
by specifying typing rules such as:
\begin{center}
\mybox{ 
\begin{tabular}{c}\\[-0.3cm]

{\small $\displaystyle\frac{\substack{{\cal T}=\{\{{\cal T}_1\,(z\underline z)\,{\cal T}_2\}\,(y\underline{y})\,{\cal T}_3\}\quad y\in {\it FV}({\cal T}_2)\quad x\in X\cap {\it FV}({\cal T}_1)\\[0.1cm] \overrightarrow{\sigma_1}\in\Sigma({\cal T}_1,x)\quad \overrightarrow{\sigma_2}\in\Sigma({\cal T}_2,\underline{z})\quad \overrightarrow{\sigma_3}\in\Sigma({\cal T}_3,\underline{y})}}{\beta^{z;y}_{({\cal T}_1,x,\overrightarrow{\sigma_1},w_1),({\cal T}_2,\underline{z},\overrightarrow{\sigma_2},w_2),({\cal T}_3,{\underline{y}},\overrightarrow{\sigma_3},w_3)}:({\cal T},x,\overrightarrow{\sigma_1\!\cdot\!\sigma_2\!\cdot\!\sigma_3},(w_1w_2)w_3\,)\rightarrow ({\cal T},x,\overrightarrow{\sigma_1\!\cdot\!\sigma_2\!\cdot\!\sigma_3},w_1(w_2w_3)\,)}$}\\[0.9cm]

{\small $\displaystyle\frac{\chi_1:({\cal T}_1,x,\overrightarrow{\sigma_1},w_1)\rightarrow ({\cal T}_1,x,\overrightarrow{\sigma_1},{w'}_1)\enspace \chi_2:({\cal T}_2,y,\overrightarrow{\sigma_2},w_2)\rightarrow ({\cal T}_2,y,\overrightarrow{\sigma_2},{w'}_2)\enspace z\in{\it FV}({\cal T}_1) \enspace z\neq x}{\chi_1 \, {_z\bo_y}\, \chi_2: (\{{\cal T}_1\,(zy)\,{\cal T}_2\},x,\overrightarrow{\sigma_1\!\cdot\!\sigma_2},w_1w_2) \rightarrow (\{{\cal T}_1\,(zy)\,{\cal T}_2\},x,\overrightarrow{\sigma_1\!\cdot\!\sigma_2},{w'}_1{w'}_2)}$} 
\end{tabular}
 }
\end{center}
 \subsubsection{The interpretation of ${\tt _{sk}{r\underline{T}}}^{+}_{\underline{\EuScript C}}$ in  ${\EuScript O}_{\EuScript C}$}
In order to define the interpretation of ${\tt _{sk}{r\underline{T}}}^{+}_{\underline{\EuScript C}}$ in ${\EuScript O}_{\EuScript C}$, we first need to ``order the inputs'' of unrooted trees figuring in object terms $({\cal T},x,\overrightarrow{\sigma},w)$ of ${\tt _{sk}{r\underline{T}}}^{+}_{\underline{\EuScript C}}$.

For an unrooted tree ${\cal T}$, a variable $x\in {\it FV}({\cal T})$ and  an element $\overrightarrow{\sigma}=(\sigma_1,\dots,\sigma_n)\in \Sigma_{({\cal T},x)}$,  the {total order} $$\sigma:[\,|{\tt in}_{({\cal T},x)}({\cal T})|\,]\rightarrow {\tt in}_{({\cal T},x)}({\cal T})$$ {on the set of inputs of} ${\cal T}$ {(relative to $x$) induced by} $\overrightarrow{\sigma}$ is defined as follows:
\begin{itemize}
\item[$\diamond$] if $({\cal T},x)=(\{a(x_1,\dots,x_n);{\it id}_{X}\},x_i)$, then  $\sigma=\overrightarrow{\sigma}$,
\item[$\diamond$] if $({\cal T},x)=(\{{\cal T}_1\,(zy)\,{\cal T}_2\},x)$, $x\in {\it FV}({\cal T}_1)$, $|{\tt in}_{({\cal T}_1,x)}({\cal T}_1)|=n$,  $|{\tt in}_{({\cal T}_2,y)}({\cal T}_1)|=m$, $\sigma_1:[n]\rightarrow {\tt in}_{({\cal T}_1,x)}({\cal T}_1)$ is the total order induced by $\overrightarrow{\sigma_1}\in\Sigma_{({\cal T}_1,x)}$, $\sigma_2:[m]\rightarrow {\tt in}_{({\cal T}_2,y)}({\cal T}_2)$ is the total order induced by $\overrightarrow{\sigma_2}\in\Sigma_{({\cal T}_2,y)}$ and $\sigma_1(i)=z$,   then $$\sigma:[n+m-1]\rightarrow {\it FV}({\cal T})\backslash\{x\}$$ is defined by \eqref{sigma_skeletal}.
\end{itemize}

\indent  The interpretation function $$\lceil-\rceil_X^{\tt sk}:{\tt _{sk}{r\underline{T}}}^{+}_{\underline{\EuScript C}}(X)\rightarrow {\EuScript O}_{\EuScript C}(|X|)$$ is defined recursively in the obvious way, e.g.
$$\begin{array}{c}
\lceil\,(\{a(x_1,\dots,x_n);{\it id}_{X}\},x_i,\vec{\sigma},\underline{a})\,\rceil^{\tt sk}_{X\backslash\{x_i\}} = (\{x_1,\dots,x_n\},x_i,\sigma,a)\\[0.25cm]
\lceil(\{{\cal T}_1\,(zy)\,{\cal T}_2\},x,\overrightarrow{\sigma_1\!\cdot\!\sigma_2},w_1w_2)\rceil^{\tt sk}_{X\backslash\{x\}} =\lceil({\cal T}_1,x,\overrightarrow{\sigma_1},w_1)\rceil^{\tt sk}_{X_1\backslash\{x\}}{\circ_{\sigma_1^{-1}(z)}}\,\lceil({\cal T}_2,y,\overrightarrow{\sigma_2},w_2)\rceil^{\tt sk}_{X_2\backslash\{y\}}\\[0.25cm]
\lceil\chi_1 \,{_z\bo_y}\, \chi_2\rceil^{\tt sk}_{X\backslash\{x\}}=\lceil \chi_1 \rceil^{\tt sk}_{X_1\backslash\{x\}}\circ_{\sigma_1^{-1}(z)}\lceil \chi_2 \rceil^{\tt sk}_{X_2\backslash\{y\}}\;,
\end{array}$$ 
where it is assumed that every total order $\sigma$ (resp. $\sigma_i$) is induced by $\overrightarrow{\sigma}$ (resp. $\overrightarrow{\sigma_i}$).
\subsubsection{The third reduction}
In what follows, we shall denote with ${\tt r\underline{T}}_{\underline{\EuScript C}}^+(X, x, {\cal T})$ the subclass   of ${\tt r\underline{T}}_{\underline{\EuScript C}}^+(X)$ determined by the rooted tree $({\cal T},x)$ (i.e.,  by the object terms whose first two components are given by $({\cal T},x)$   and by the arrow terms among them). 
 We define the family of third reduction functions
$${\tt Red_3}(X,x,{\cal T},\overrightarrow{\sigma}):{\tt r\underline{T}}_{\underline{\EuScript C}}^+(X,x, {\cal T})\rightarrow {\tt _{sk}{r\underline{T}_{\underline{\EuScript C}}}}^+(X),$$ where  $x\in X$, ${\cal T}$ is an unrooted tree such that ${\it FV}({\cal T})=X$ and $\overrightarrow{\sigma}\in \Sigma_{({\cal T},x)}$, as follows.\\[0.1cm]
\indent For  object terms of ${\tt r\underline{T}}_{\underline{\EuScript C}}^+(X)$, we set$${\tt Red_3}(X,x,{\cal T},\overrightarrow{\sigma})(({\cal T},x,w))=({\cal T},x,\overrightarrow{\sigma},w).$$ For an arrow term  $\chi$ of ${\tt r\underline{T}}_{\underline{\EuScript C}}^+(X)$, ${\tt Red_3}(X,x,{\cal T},\overrightarrow{\sigma})(\chi)$ is defined recursively in the obvious way, e.g.:  
\begin{itemize}
\item[$\diamond$]  ${\tt Red_3}(X,x,\{\{{\cal T}_1\,(z\underline z)\,{\cal T}_2\}\,(y\underline{y})\,{\cal T}_3\},\overrightarrow{\sigma_1\!\cdot\!\sigma_2\!\cdot\!\sigma_3})(\beta^{z;y}_{({\cal T}_1,x,w_1),({\cal T}_2,\underline{z},w_2),({\cal T}_3,\underline{y},w_3)})=$\\[0.1cm]
\phantom{aaaaaaaaaaaaaaaaaaaaaaaaaaaaaaaaaaaaaa}$\beta^{\sigma_1^{-1}(z);\sigma_2^{-1}(y)}_{({\cal T}_1,x,\overrightarrow{\sigma_1},w_1),({\cal T}_2,\underline{z},\overrightarrow{\sigma_2},w_2),({\cal T}_3,\underline{y},\overrightarrow{\sigma_3},w_3)}$,
\item[$\diamond$]  if $\chi=\chi_1\, {_z\bo_y}\, \chi_2$, where $\chi_1:({\cal T}_1,x,w_1)\rightarrow ({\cal T}_1,x,{w'}_1)$ and $\chi_2:({\cal T}_2,y,w_2)\rightarrow ({\cal T}_2,y,{w'}_2)$, and if $\overrightarrow{\sigma_1}\in \Sigma_{({\cal T}_1,x)}$ and $\overrightarrow{\sigma_2}\in \Sigma_{({\cal T}_2,y)}$,  then    $${\tt Red_3}(X,x,\{{\cal T}_1\,(zy)\,{\cal T}_2\},\overrightarrow{\sigma_1\!\cdot\!\sigma_2})(\chi_1\,{_z\bo_y}\,\chi_2)=\phantom{aaaaaaasssssssssaaassssaa}$$ $$\phantom{aaaaassssaa}{\tt Red_3}(X_1,x,{\cal T}_1,\overrightarrow{\sigma_1})(\chi_1)\,{_{\sigma^{-1}_1(z)}\bo_{\sigma^{-1}_2(y)}}\,{\tt Red_3}(X_2,y,{\cal T}_2,\overrightarrow{\sigma_2})(\chi_2).$$
\end{itemize}
\begin{rem}
The type of the third reduction ${\tt Red_3}(X,x,{\cal T},\overrightarrow{\sigma})(\chi)$ of an arrow term $\chi:({\cal T},x,u)\rightarrow ({\cal T},x,v)$ is $${\tt Red_3}(X,x,{\cal T},\overrightarrow{\sigma})(\chi):{\tt Red_3}(X,x,{\cal T},\overrightarrow{\sigma})(({\cal T},x,u))\rightarrow {\tt Red_3}(X,x,{\cal T},\overrightarrow{\sigma})(({\cal T},x,v)).$$ Therefore, the third reduction of a pair of arrow terms of the {same type} in ${\tt r\underline{T}}^+_{\underline{\EuScript C}}(X)$ is a pair of arrow terms of the {same type} in ${\tt _{sk}{r\underline{T}_{\underline{\EuScript C}}}}^+(X)$.
\end{rem}
\begin{rem}
As opposed to second reduction functions ${\tt Red_2}(X,x)$ (see Remark \ref{notasection}), each third reduction function ${\tt Red_3}(X,x,{\cal T},\overrightarrow{\sigma})$ is a section of the function $F^{\tt sk}_{\underline{\EuScript C}}(X) :  {\tt _{sk}{r\underline{T}_{\underline{\EuScript C}}}}^+(X)\rightarrow {\tt r\underline{T}}_{\underline{\EuScript C}}^+(X,x, {\cal T})$,  defined in the obvious way by forgetting the skeletalization data.
\end{rem}
\begin{thm}\label{t3} For an arbitrary object term $({\cal T},x,w)$  and an arbitrary arrow term $\chi$ of ${\tt r\underline{T}}_{\underline{\EuScript C}}^+(X)$, the equalities  $$\lceil {\tt Red_3}(X,x,{\cal T},\overrightarrow{\sigma})(({\cal T},x,w)) \rceil^{\tt sk}_{X\backslash\{x\}}=(X,x,\sigma, \lceil ({\cal T},x,w) \rceil_X)$$ and  
$$\lceil {\tt Red_3}(X,x,{\cal T},\overrightarrow{\sigma})(\chi)\rceil^{\tt sk}_{X\backslash\{x\}}=(X,x,\sigma,\lceil \chi \rceil_X),$$ in which 
 the total order $\sigma$ is induced from $\overrightarrow{\sigma}$, 
  hold in ${\EuScript O}_{\EuScript C}(|X\backslash\{x\}|)$.
\end{thm}
\begin{proof}  Easy,  by induction on the proof of the $({\cal T},x)$-admissibility of $w$ (for the first equality), and   by induction on the structure of $\chi$ (for the second equality). 
\end{proof}
The following result is a direct consequence of Theorem \ref{t3}.
\begin{cor}\label{vazno3}
For     arrow terms $\chi_1$ and $\chi_2$ of the same type in ${\tt r\underline{T}}_{\underline{\EuScript C}}^+(X)$, the equality  $\lceil\chi_1\rceil_X=\lceil\chi_1\rceil_X$ follows from the equality $\lceil {\tt Red_3}(X,x,{\cal T},\overrightarrow{\sigma})(\chi_1)\rceil^{\tt sk}_{X\backslash\{x\}}=\lceil {\tt Red_3}(X,x,{\cal T},\overrightarrow{\sigma})(\chi_2)\rceil^{\tt sk}_{X\backslash\{x\}}$.
\end{cor}
 
\subsection{The proof of the coherence theorem}\label{cohe}
We finally assemble the three reductions, by using the two invariance properties common to all three reductions:  by reducing a pair  of arrow terms of the same type, 
\begin{itemize}
\item[1.]  the result is always   a pair of arrow terms of the same type, and
\item[2.] the equality of interpretations of the two resulting  arrow terms implies the equality of the interpretations of the respective starting  arrow terms.
\end{itemize}
{\it Proof of Theorem \ref{coherence-theorem}.} 
By Proposition \ref{t1} (first reduction), it is enough to prove the equality $$[{\tt Red_1}(\Phi)]_X=[{\tt Red_1}(\Psi)]_X.$$ By Lemma \ref{delta} and Lemma \ref{int_delta}, the problem translates to showing that  $$\lfloor\Delta_X^{-1}({\tt Red_1}(\Phi))\rfloor_X=\lfloor\Delta_X^{-1}({\tt Red_1}(\Psi))\rfloor_X.$$ By Corollary \ref{roknucuse} (second reduction), this equality follows from the equality $$\lceil {\tt Red_2}(X,x)(\Delta_X^{-1}({\tt Red_1}(\Phi)))\rceil_{X}=\lceil {\tt Red_2}(X,x)(\Delta_X^{-1}({\tt Red_1}(\Psi)))\rceil_{X},$$ where $x\in X$ is arbitrary. By Corollary \ref{vazno3} (third reduction), the above equality holds if, in ${\EuScript O}_{\EuScript C}$, we have  $$\lceil {\tt Red_3}(X,x,{\cal T},\overrightarrow{\sigma})({\tt Red_2}(X,x)(\Delta_X^{-1}({\tt Red_1}(\Phi))))\rceil^{\tt sk}_{X\backslash\{x\}}=\phantom{aaaaaaaaaaaaaaaaaaaaaaaaaa}$$ $$\phantom{aaaaaaaaaaaaaaaaaaaaaaaaaa}\lceil {\tt Red_3}(X,x,{\cal T},\overrightarrow{\sigma})({\tt Red_2}(X,x)(\Delta_X^{-1}({\tt Red_1}(\Psi))))\rceil^{\tt sk}_{X\backslash\{x\}},$$

\noindent where ${\cal T}$ is the unrooted tree figuring in $\Delta_X^{-1}({\tt Red_1}(W_s(\Phi)))$. Finally, the last equality holds by the coherence of ${\EuScript O}_{\EuScript C}$, established in \cite{dp}. $\quad\quad\quad\quad\quad\quad\quad\quad\quad\quad\quad\quad\quad\quad\quad\quad\quad\quad\quad\quad\quad\quad\blacksquare$
\section{Relaxing the equivariance and unit law}\label{s3}
In this section, we extend the notion of categorified entries-only cyclic operad, established in Section \ref{s2}, by additionally relaxing the equivariance axiom  \hyperlink{(EQ)}{\texttt{(EQ)}}, as well as by including units and relaxing the unit axiom. We do this in two stages: starting from Definition \ref{cat} (which does not incorporate units), we first relax \hyperlink{(EQ)}{\texttt{(EQ)}}, and then we endow the resulting structure with units and with two isomorphisms that account for the relaxed unit law. (Clearly, the unit law can also be relaxed while keeping the equivariance  strict; this notion of categorified cyclic operad can easily be deduced from ours.)
\subsection{Relaxing the equivariance}\label{eq_relax}
For bijections $\sigma_1:X'\rightarrow X$, $\sigma_2:Y'\rightarrow Y$ and $\sigma=\sigma_1|^{X\backslash\{x\}}\cup \sigma_2|^{Y\backslash\{y\}}$, operations $f\in{\EuScript C}(X)$ and $g\in {\EuScript C}(Y)$, and elements $x\in X$, $y\in Y$, $x'\in X$ and $y'\in Y$, such that $\sigma_1^{-1}(x)=x'$ and $\sigma^{-1}_2(y)=y'$,  \hyperlink{(EQ)}{\texttt{(EQ)}} is replaced by the existence of an isomorphim $\varepsilon$, called the {\em equivariance isomorphism}, whose respective components
$$\varepsilon^{x,y;x',y'}_{f,g;\sigma}:(f {_x\circ_y} \, g)^{\sigma}\rightarrow f^{\sigma_1}{_{{ x'}}\circ_{y'}}\, g^{\sigma_2}$$ 

\noindent
are natural in $f$ and $g$
 and are subject to the following coherence conditions:\\[0.2cm]
\indent   - \hypertarget{behex}{{\tt ($\beta\varepsilon$-\texttt{hexagon})}}
\begin{center}
\begin{tikzpicture}
\node(a) at (0,0) {\footnotesize $((f {_x\circ_{\underline{x}}} \, g)\, {_y\circ_{\underline{y}}} \, h)^{\sigma}$};
\node(b) at (4.5,0) {\footnotesize  $(f {_x\circ_{\underline{x}}} \, g)^{\sigma_1} {_{y'}\circ_{\underline{y'}}} \, h^{\sigma_2}$};
\node(c) at (11,0) {\footnotesize  $(f^{\sigma_{11}} {_{x'}\circ_{\underline{x'}}} \, g^{{\sigma_{12}}}) {_{y}\circ_{\underline{y'}}} \, h^{\sigma_2}$};
\node(1) at (0,-2) {\footnotesize  $(f {_x\circ_{\underline{x}}} \, (g\, {_y\circ_{\underline{y}}} \, h))^{\sigma}$};
\node(2) at (4.5,-2) {\footnotesize   $f^{\sigma_{11}} {_{x'}\circ_{\underline{x'}}} \, (g\, {_y\circ_{\underline{y}}} \, h)^{\sigma_3}$};
\node(3) at (11,-2) {\footnotesize   $f^{\sigma_{11}} {_{x'}\circ_{\underline{x'}}} \, (g^{{\sigma_{12}}} {_{y'}\circ_{\underline{y'}}} \, h^{\sigma_2})$};
\draw[->] (a)--(b) node [midway,above] {\scriptsize $\varepsilon^{y,\underline{y};y',\underline{y'}}_{f {_x\circ_{\underline{x}}} \, g,h;\sigma}$};
\draw[->] (b)--(c) node [midway,above] {\scriptsize $\varepsilon^{x,\underline{x};x',\underline{x'}}_{f,g;\sigma_1}\,{_{y'}\circ_{\underline{y'}}} \, 1_{h^{\sigma_2}}$};
\draw[->] (a)--(1) node [midway,left] {\scriptsize $(\beta^{x,\underline{x};y,\underline{y}}_{f,g,h})^{\sigma}$};
\draw[->] (1)--(2) node [midway,above] {\scriptsize $\varepsilon^{x,\underline{x};x',\underline{x'}}_{f , g{_y\circ_{\underline{y}}}h;\sigma}$};
\draw[->] (2)--(3) node [midway,above] {\scriptsize $1_{f^{\sigma_{11}}}\,{_{x'}\circ_{\underline{x'}}}\,\varepsilon^{y,\underline{y};y',\underline{y'}}_{g, h;\sigma_3}$};
\draw[->] (c)--(3) node [midway,right] {\scriptsize $\beta^{x,\underline{x};y,\underline{y}}_{f^{\sigma_{11}},g^{\sigma_{12}},h^{\sigma_{2}}}$};
\end{tikzpicture}
\end{center}

\indent   - \hypertarget{gehex}{{\tt ($\gamma\varepsilon$-\texttt{square})}}
\begin{center}
\begin{tikzpicture}
\node(a) at (0,0) {\footnotesize $(f {_x\circ_{y}} \, g)^{\sigma}$};
\node(b) at (4,0) {\footnotesize  $f^{\sigma_1} {_{x'}\circ_{y'}} \, g^{\sigma_2}$};
\node(1) at (0,-2) {\footnotesize  $(g\, {_y\circ_{x}} \, f)^{\sigma}$};
\node(2) at (4,-2) {\footnotesize   $g^{\sigma_2}{_{y'}\circ_{x'}} \, f^{\sigma_1} $};
\draw[->] (a)--(b) node [midway,above] {\scriptsize $\varepsilon^{x,y;x',{y'}}_{f ,g;\sigma}$};
\draw[->] (a)--(1) node [midway,left] {\scriptsize $(\gamma^{x,y}_{f,g})^{\sigma}$};
\draw[->] (1)--(2) node [midway,above] {\scriptsize $\varepsilon^{y,x;y',x'}_{g,f;\sigma}$};
\draw[->] (b)--(2) node [midway,right] {\scriptsize $\gamma^{x',y'}_{f^{\sigma_{1}},g^{\sigma_{2}}}$};
\end{tikzpicture}
\end{center}
The coherence conditions \hyperlink{behex}{{\tt ($\beta\varepsilon$-\texttt{hexagon})}} and \hyperlink{gehex}{{\tt ($\gamma\varepsilon$-\texttt{hexagon})}} replace the conditions  \hyperlink{bs}{{\tt ($\beta\sigma$)}} and  \hyperlink{gs}{{\tt ($\gamma\sigma$)}}  from Definition \ref{cat}, respectively. Finally, the naturality of the equivariance isomorphism replaces the condition   \hyperlink{(EQ-mor)}{{\tt(EQ-mor)}}.  \\[0.2cm]
\indent The   coherence theorem for this notion of categorified entries-only cyclic operad has the same  statement as the original coherence theorem (see \S \ref{coht}), except that the definition of the interpretation function $[[-]]:{\tt Free}_{\underline{\EuScript C}}\rightarrow {\EuScript C}(X)$ is adapted in the following way: we set $$[[{\varepsilon_4}^{x,y;x',y'}_{{\cal W}_1,{\cal W}_2;\sigma}]]_{X}=\varepsilon^{x,y;x',y'}_{[[{\cal W}_1]]_{X_1},[[{\cal W}_2]]_{X_2};\sigma}\quad \mbox{ and } \quad [[{\varepsilon_4}^{x,y;x',y'\,^{_{-1}}}_{{\cal W}_1,{\cal W}_2;\sigma}]]_{X}=\varepsilon^{x,y;x',y'\enspace ^{_{-1}}}_{[[{\cal W}_1]]_{X_1},[[{\cal W}_2]]_{X_2};\sigma}.$$

 The coherence proof   differs from the original coherence proof only in the first reduction. More precisely, although the    first reduction function will be defined in the same way as before, Lemma \ref{eq_int_obj} and Lemma \ref{kh}, and therefore also Proposition \ref{t1}, which relied on strict equivariance, will no longer hold. 
 Instead, we shall proceed like in the second reduction. 
 In order to take  relaxed equivariance into account,  we declare the syntax of {\em the first reduction arrow terms}: 
\begin{center}\vspace{0.1cm}\mybox{
$\displaystyle \Xi::=\left\{\begin{array}{l}
1_{\cal W}\\[0.25cm]  
 {\varepsilon_1}_{\underline{a}}^{\sigma}\,\,|\,\, {\varepsilon_1}_{\underline{a}}^{\sigma \, ^ {_{-1}}}\,\,|\,\, {{\varepsilon}_2}_{\cal W} \,\,|\,\, {{\varepsilon}_2}^{_{-1}}_{\cal W} \,\,|\,\, {\varepsilon_3}_{\cal W}^{\sigma,\tau}\,\,|\,\,{\varepsilon_3}_{\cal W}^{\sigma,\tau\,^{_{-1}}}\,\,|\,\,{\varepsilon_4}^{x,y;x',y'}_{{\cal W}_1,{\cal W}_2;\sigma}\,\,|\,\, {\varepsilon_4}^{x,y;x',y'\,^{_{-1}}}_{{\cal W}_1,{\cal W}_2;\sigma}\\[0.25cm]
\Phi\circ\Phi\,\,|\,\, \Phi \,{_x\bo_y}\,  \Phi,\end{array} \right. $}\vspace{0.1cm}\end{center} 
to which we assign types in the same way as   in  ${\tt Free}_{\underline{\EuScript C}}$. 
By the following two constructions, we express the reduction algorithm on object terms of ${\tt Free}_{\underline{\EuScript C}}$ from \S \ref{fr1} and the equivalence relation $\equiv$ from \S \ref{aux},  in terms of first reduction arrow terms.  
 
\paragraph{The reduction algorithm.} Observe first that, for an arbitrary object term ${\cal W}$ of ${\tt Free}_{\underline{\EuScript C}}$ and a  fixed ${\tt Red_1}({\cal W})$, there exists an arrow term  $\varepsilon_{\cal W}:{\cal W}\rightarrow {\tt Red_1}({\cal W})$ that belongs to $\Xi$ -- such an arrow term is induced from the   algorithm  from \S \ref{fr1}, as the algorithm becomes deterministic by fixing ${\tt Red_1}({\cal W})$. Note that such an arrow term is  not  unique, since, for example,  for $\underline{a}:X$, $\underline{b}:Y$, $x\in X$, $y\in Y$ and bijections $\sigma$, $\sigma_1$ and $\sigma_2$ which are the  identities on $X\backslash\{x\}\cup Y\backslash\{y\}$, $X$ and $Y$, respectively, the    reduction  
$$(\underline{a} \, {_x\bo_y}\, \underline{b})^{{\it id}_{X\backslash\{x\}\cup Y\backslash \{y\}}}\leadsto \underline{a} \, {_x\bo_y}\, \underline{b}$$ can be expressed as  $${\varepsilon_2}_{\underline{a} \, {_x\bo_y}\, \underline{b}},\quad ({\varepsilon_2}_{\underline{a}} \,{_x\bo_y}\, {\varepsilon_2}_{\underline{b}})\circ{\varepsilon_4}^{x,y;x,y}_{\underline{a},\underline{b},{\it id}_{X\backslash\{x\}\cup Y\backslash\{y\}}},\quad \mbox{and} \quad ({\varepsilon_1}^{{\it id}_X}_{\underline{a}} \,{_x\bo_y}\, {\varepsilon_1}^{{\it id}_Y}_{\underline{b}})\circ{\varepsilon_4}^{x,y;x,y}_{\underline{a},\underline{b},{\it id}_{X\backslash\{x\}\cup Y\backslash\{y\}}}.$$
 Nevertheless, we shall suppose that for each object term ${\cal W}$ of ${\tt Free}_{\underline{\EuScript C}}$ and a  fixed ${\tt Red_1}({\cal W})$, a choice of $\varepsilon_{\cal W}:{\cal W}\rightarrow {\tt Red_1}({\cal W})$ has also been fixed - for example, we can prioritize $\varepsilon_2$. We shall refer to $\varepsilon_{\cal W}$ as the first reduction arrow term of ${\cal W}$.

 \paragraph{The equivalence relation $\equiv$.} Suppose  that $U={W}_1\,{_x\bo_y}\,{W}_2$ and $V= {W}_1[\underline{a^{\tau_1}}/\underline{a}] \,{_{x'}\bo_{y'}}\,{W}_2[\underline{b^{\tau_2}}/\underline{b}]$ and suppose that $U\equiv V$ is obtained as in  \S \ref{aux}. We define  the arrow term $$\epsilon_{U,V}:{W}_1\,{_x\bo_y}\,{W}_2\rightarrow {W}_1[\underline{a^{\tau_1}}/\underline{a}] \,{_{x'}\bo_{y'}}\,{W}_2[\underline{b^{\tau_2}}/\underline{b}]$$  as follows\footnote{Attention: the symbol $\epsilon$ used in the specification of the arrow term $\epsilon_{U,V}$ is different than $\varepsilon$.}.  We  note that the codomain of $${\varepsilon_4}^{x,y;x',y'}_{W_1,W_2,{\it id}_{X\backslash\{x\}\cup Y\backslash\{y\}}}\circ {\varepsilon_2}_{{W}_1{_x\bo_y}{W}_2}^{-1}$$ is  of the form $W_1^{\sigma_1} \,{_{x'}\bo_{y'}}\,W_2^{\sigma_2 }$, with  $W_1^{\sigma_1}\leadsto {W}_1[\underline{a^{\tau_1}}/\underline{a}]$ and $W_2^{\sigma_2}\leadsto {W}_2[\underline{b^{\tau_2}}/\underline{b}]$. Then we set 
 $$\epsilon_{U,V}=
 (\varepsilon_{W_1^{\sigma_1}} \,{_{x'}\bo_{y'}}\,\varepsilon_{W_2^{\sigma_2}})\circ
 {\varepsilon_4}^{x,y;x',y'}_{W_1,W_2,{\it id}_{X\backslash\{x\}\cup Y\backslash\{y\}}}\circ {\varepsilon_2}_{{W}_1{_x\bo_y}{W}_2}^{-1}\:,$$
 where  $\varepsilon_{W_1^{\sigma_1}}$ and $\varepsilon_{W_2^{\sigma_2}}$ are the first reduction arrow terms of $W^{{\sigma}_1}_1$ and $W^{{\sigma}_2}_2$, tailored to the respective normal forms ${W}_1[\underline{a^{\tau_1}}/\underline{a}]$ and ${W}_2[\underline{b^{\tau_2}}/\underline{b}]$.
We say that $\epsilon_{U,V}$ is the witness of the equivalence $U\equiv V$. By the inductive character of the definition of $\equiv$, the definition of witnesses extends easily to  arbitrary pairs $U,V$ such that $U\equiv V$.  Note that, having fixed the first reduction arrow terms of all object terms of  ${\tt Free}_{\underline{\EuScript C}}$, the witness $\epsilon_{U,V}$ of an equivalence $U\equiv V$ is unique.  
\\[0.1cm]
\indent Proposition \ref{t1} of the first reduction is then adapted as follows:
\begin{thm}\label{t4}
 For an arbitrary   arrow term   $\Phi:{\cal U}\rightarrow {\cal V}$ of ${\tt Free}_{\underline{\EuScript C}}$,  the following equality of interpretations holds: $$[[\epsilon_{{\tt Red_1}({\cal V}),V}]]_X\circ[[\varepsilon_{\cal V}]]_X\circ[[\Phi]]_X=[{\tt Red_1}(\Phi)]_X\circ [[\varepsilon_{\cal U}]]_X,$$ 
  where  $V$ is the target   of ${\tt Red_1}(\Phi)$,
 $\epsilon: {\tt Red_1}({\cal V})\rightarrow V$ is the witness  of the equivalence ${\tt Red_1}({\cal V})\equiv V$, and $\varepsilon_{\cal V}$ and $\varepsilon_{\cal U}$ are the first reduction arrow terms of ${\cal V}$ and ${\cal U}$, respectively. 
\end{thm}
The soundness of the first reduction in the coherence proof   follows directly by Theorem \ref{t4}:
\begin{cor}\label{cor4}
For arrow terms $\Phi_1$ and $\Phi_2$ of the same type in ${\tt Free}_{\underline{\EuScript C}}$, the equality $[[\Phi_1]]_X=[[\Phi_2]]_X$ follows from the equality $[{\tt Red}_1(\Phi_1)]_X=[{\tt Red}_1(\Phi_2)]_X$.
\end{cor}
 
We can  weaken even further the axioms by turning the action of symmetric groups into a pseudo-action. In addition to $\varepsilon$, one requires the existence of natural isomorphisms $${\varepsilon_2}_{f}:f^{\it id}\rightarrow f \quad \mbox{ and }\quad {\varepsilon_3}^{\sigma,\tau}_{f}:(f^{\sigma})^{\tau}\rightarrow f^{\sigma\circ\tau},$$ satisfying the following coherence conditions:
$$\begin{array}{lll}

\varepsilon_3\circ\varepsilon_3 = \varepsilon_3\circ (\varepsilon_3)^\upsilon && (\mbox{starting from}\; ((f^\sigma)^\tau)^\upsilon)\\
\varepsilon_3 =(\varepsilon_2)^\sigma  && (\mbox{starting from}\; (f^{\it id})^\sigma)\\
\varepsilon_3 =\varepsilon_2  && (\mbox{starting from}\; (f^\sigma)^{\it id}),\\
\end{array}$$
 which are the well-known 
coherence conditions for pseudo-functors (see e.g. \cite[Definition 7.5.1]{borceux}).  The treatment above can be repeated without difficulty with ${\varepsilon_2}_{\cal W}$ and ${\varepsilon_3}^{\sigma,\tau}_{\cal W}$ now interpreted as these stipulated coherent isomorphisms.
\subsection{Adding the units and relaxing the unit law}
For categorified entries-only cyclic operads with units,  for each two-elements set $\{x,y\}$, we require the existence of a distinguished operation ${\it id}_{x,y}\in {\EuScript C}(\{x,y\})$, called the {\em identity} or {\em unit  indexed by $\{x,y\}$}. It is understood that ${\it id}_{y,x}={\it id}_{x,y}$. In accordance with this convention, we also require that ${\it id}^{\sigma_1}_{x,y}={\it id}^{\sigma_2}_{x,y}$, where $\sigma_1,\sigma_2:\{u,v\}\rightarrow\{x,y\}$ are defined by $\sigma_1(x)=u$, $\sigma_1(y)=v$, $\sigma_2(x)=v$ and $\sigma_2(y)=u$. Finally, we endow the structure with  the family of natural isomorphisms  $$\iota_{(-),{\it id}_{y,z}}^{x,y} :(-)\, {_x\circ _y} \, {\it id}_{y,z}\rightarrow (-)^{\sigma},$$ where $\sigma$ renames $x$ to $z$, and $$\nu^{u,v}_{x,y}: (-)^{\sigma}\circ{\bf 1}_{\{x,y\}} \rightarrow {\bf 1}_{\{u,v\}},$$ where $\sigma(x)=u$ and $\sigma(y)=v$ and the functor ${\bf 1}_{\{x,y\}}:{\bf 1}\rightarrow {\EuScript C}(\{x,y\})$ ``picks out'' the unit indexed by $\{x,y\}$, whose respective components 
$$\iota^{x,y}_{f,{\it id}_{y,z}}:f\, {_x\circ _y} \, {\it id}_{y,z}\rightarrow f^{\sigma}\quad\quad \mbox{and}\quad\quad \nu^{u,v}_{x,y}:{\it id}^{\sigma}_{x,y}\rightarrow {\it id}_{u,v}$$   satisfy the following coherence conditions: \\[0.2cm]

\noindent\begin{tabular}{    m{8cm}   m{8cm}  }
-  \hypertarget{begaihex}{{\tt ($\beta\gamma\iota\varepsilon$-\texttt{hexagon-1})}} & - \hypertarget{bedgaihex}{{\tt ($\beta\gamma\iota\varepsilon$-\texttt{hexagon-2})}}
\\
\resizebox{7.45cm}{!}{\begin{tikzpicture}
\node(a) at (0,0) {\footnotesize $({\it id}_{x,z}\, {_x\circ _{\underline{x}}} \, f)\, {_y\circ _{\underline{y}}} \, g$};
\node(b) at (4.5,0) {\footnotesize  ${\it id}_{x,z}\, {_x\circ _{\underline{x}}} \, (f\, {_y\circ _{\underline{y}}} \, g) $};
\node(1) at (0,-2.25) {\footnotesize  $(f\, {_{\underline{x}}\circ _{{x}}} \, {\it id}_{x,z})\, {_y\circ _{\underline{y}}} \, g$};
\node(3) at (0,-4.5) {\footnotesize  $f^{\sigma}\, {_y\circ _{\underline{y}}} \, g$};
\node(2) at (4.5,-2.25) {\footnotesize   $(f\, {_y\circ _{\underline{y}}} \, g) \, {_{\underline{x}}\circ _{x}}  \, {\it id}_{x,z}$};
\node(4) at (4.5,-4.5) {\footnotesize   $(f\, {_y\circ _{\underline{y}}} \, g)^{\tau}$};
\draw[->] (a)--(b) node [midway,above] {\scriptsize $\beta^{x,\underline{x};y,\underline{y}}_{{\it id}_{x,z},f,g}$};
\draw[->] (a)--(1) node [midway,left] {\scriptsize $\gamma^{x,\underline{x}}_{{\it id}_{x,z}, f}\, {_y\circ _{\underline{y}}} \, 1_{g}$};
\draw[->] (1)--(3) node [midway,left] {\scriptsize $\iota^{x,\underline{x}}_{f,{\it id}_{x,z}}\, {_y\circ _{\underline{y}}} \, 1_{g}$};
\draw[->] (4)--(3) node [midway,above] {\scriptsize ${\varepsilon^{y,\underline{y};y,\underline{y}}_{f,g;\tau}}$};
\draw[->] (b)--(2) node [midway,right] {\scriptsize $\gamma^{x,\underline{x}}_{{\it id}_{x,z},f\, {_y\circ _{\underline{y}}} \, g}$};
\draw[->] (2)--(4) node [midway,right] {\scriptsize $\iota^{\underline{x},x}_{f\, {_y\circ _{\underline{y}}} \, g,{\it id}_{x,z}}$};
\end{tikzpicture}} &  
\resizebox{7.6cm}{!}{\begin{tikzpicture}
\node(a) at (0,0) {\footnotesize $(f\, {_x\circ _{\underline{x}}} \, {\it id}_{\underline{x},y})\, {_y\circ _{\underline{y}}} \, g  $};
\node(b) at (4.5,0) {\footnotesize  $f\, {_x\circ _{\underline{x}}} \, ({\it id}_{\underline{x},y}\, {_y\circ _{\underline{y}}} \, g ) $};
\node(1) at (0,-2.25) {\footnotesize  $f^{\sigma} {_y\circ _{\underline{y}}} \, g$};
\node(3) at (0,-4.5) {\footnotesize  $f {_x\circ _{\underline{y}}} \, g$};
\node(4) at (4.5,-4.5) {\footnotesize  $f {_x\circ _{\underline{x}}} \, g^{\tau}$};
\node(2) at (4.5,-2.25) {\footnotesize   $f\, {_x\circ _{\underline{x}}} \,(g \, {_{\underline{y}}\circ _{y}} \, {\it id}_{\underline{x},y})$};
\draw[->] (a)--(b) node [midway,above] {\scriptsize $\beta^{x,\underline{x};y,\underline{y}}_{f,{\it id}_{\underline{x},y},g}$};
\draw[->] (3)--(4) node [midway,above] {\scriptsize $\varepsilon^{x,\underline{y};x,\underline{x}}_{f,g;{\it id}}$};
\draw[->] (a)--(1) node [midway,left] {\scriptsize $\iota^{x,\underline{x}}_{f,{\it id}_{\underline{x},y}}\, {_y\circ _{\underline{y}}} \, 1_{g}$};
\draw[->] (1)--(3) node [midway,left] {\scriptsize ${\varepsilon^{x,\underline{y};y,\underline{y}}_{f,g;{\it id}}}^{-1}$};
\draw[->] (b)--(2) node [midway,right] {\scriptsize $1_{f}\, {_x\circ _{\underline{x}}} \, \gamma^{y,\underline{y}}_{{\it id}_{\underline{x},y},g}$};
\draw[->] (2)--(4) node [midway,right] {\scriptsize $1_{f}\, {_x\circ _{\underline{x}}} \, \iota^{\underline{y},y}_{g,{\it id}_{\underline{x},y}}$};
\end{tikzpicture}}\end{tabular}

\vspace{0.3cm}
\noindent \begin{tabular}{    m{8cm}   m{8cm}  }
  - \hypertarget{begaihex}{{\tt ($\beta\iota\varepsilon$-\texttt{square})}} &  - \hypertarget{begaihdddex}{{\tt ($\gamma\iota$-\texttt{square})}}   \\  

\resizebox{7.5cm}{!}{\begin{tikzpicture}
\node(a) at (0,0) {\footnotesize $(f\, {_x\circ _{\underline{x}}} \, g)\, {_y\circ _{\underline{y}}} \, {\it id}_{\underline{y},z} $};
\node(b) at (4.5,0) {\footnotesize  $f\, {_x\circ _{\underline{x}}} \, (g\, {_y\circ _{\underline{y}}} \, {\it id}_{\underline{y},z}) $};
\node(1) at (0,-2.5) {\footnotesize  $(f\, {_x\circ _{\underline{x}}} \, g)^{\sigma}$};
\node(2) at (4.5,-2.5) {\footnotesize   $f\, {_x\circ _{\underline{x}}} \,g^{\tau}$};
\draw[->] (a)--(b) node [midway,above] {\scriptsize $\beta^{x,\underline{x};y,\underline{y}}_{f,g,{\it id}_{\underline{y},z}}$};
\draw[->] (a)--(1) node [midway,left] {\scriptsize $\iota^{y,\underline{y}}_{f\, {_x\circ _{\underline{x}}} \, g,{\it id_{\underline{y},z}}}$};
\draw[->] (1)--(2) node [midway,above] {\scriptsize ${\varepsilon^{x,\underline{x};x,\underline{x}}_{f,g;\sigma}}$};
\draw[->] (b)--(2) node [midway,right] {\scriptsize $1_{f}\, {_x\circ _{\underline{x}}} \, \iota^{y,\underline{y}}_{g,{\it id}_{\underline{y},z}}$};
\end{tikzpicture}} &  
\resizebox{7.5cm}{!}{\begin{tikzpicture}
\node(a) at (0,0) {\footnotesize $(f\, {_x\circ _{\underline{x}}} \, {\it id}_{\underline{x},y})\, {_y\circ _{\underline{y}}} \, g  $};
\node(b) at (4.5,0) {\footnotesize  $g \, {_{\underline{y}}\circ _{y}} \,  (f\, {_x\circ _{\underline{x}}} \, {\it id}_{\underline{x},y})$};
\node(1) at (0,-2.5) {\footnotesize  $f^{\sigma} {_y\circ _{\underline{y}}} \, g $};
\node(2) at (4.5,-2.5) {\footnotesize   $g \,{_{\underline{y}}\circ _{y}} \, f^{\sigma} $};
\draw[->] (a)--(b) node [midway,above] {\scriptsize $\gamma^{y,\underline{y}}_{f {_x\circ _{\underline{x}}} {\it id}_{\underline{x},y},g}$};
\draw[->] (a)--(1) node [midway,left] {\scriptsize $\iota^{y,\underline{y}}_{f\, {_x\circ _{\underline{x}}} \, g,{\it id_{\underline{y},z}}}$};
\draw[->] (1)--(2) node [midway,above] {\scriptsize $\gamma^{y,\underline{y}}_{f^{\sigma},g}$};
\draw[->] (b)--(2) node [midway,right] {\scriptsize $1_{g}\, {_{\underline{y}}\circ _{y}} \, \iota^{x,\underline{x}}_{f,{\it id}_{\underline{x},y}}$};
\end{tikzpicture} }
\end{tabular}

\vspace{0.3cm}
\noindent \begin{tabular}{    m{8cm}   m{8cm}  }
 - \hypertarget{begaihdddexxx}{{\tt ($\varepsilon\iota$-\texttt{square})}} & \indent  - \hypertarget{begaihex}{{\tt ($\iota\varepsilon\nu$-\texttt{square})}} \\
\resizebox{8.25cm}{!}{
\begin{tikzpicture}
\node (b') at (5,-9) {\footnotesize $({f}\, {_z\circ_u}\, {\it id}_{u,q})\,{_x\circ_{\underline{x}}}\, g^{\kappa_2}$};
\node (j) at (0,-9) {\footnotesize $(({f}\, {_z\circ_u}\, {\it id}_{u,q})\,{_x\circ_{\underline{x}}}\, g)^{\tau}$};
\node (k) at (0,-11.5) {\footnotesize $({f}^{\kappa_1}\,{_x\circ_{\underline{x}}}\, g)^{\tau}$};
\node (kk) at (5,-11.5) {\footnotesize ${f}^{\kappa_1}{_x\circ_{\underline{x}}}\, g^{\kappa_2}$};
\draw [->]   (j)--(b') node [midway,above] {\scriptsize $\varepsilon^{x,\underline{x};x,\underline{x}}_{{{f}\! {_z\circ_u}\!{\it id}_{u,q}},g;\tau}$};
\draw [->]   (b')--(kk) node [midway,left] {\scriptsize $\iota^{z,u}_{f,{\it id}_{u,q}}\,{_x\circ_{\underline{x}}}\, 1_{g^{\kappa_2}}$};
\draw [->] (j)--(k) node [midway,left] {\scriptsize $(\iota^{z,u}_{f,{\it id}_{u,q}} \,{_x\circ_{\underline{x}}}\, 1_{g})^{\tau} $};
\draw [->] (k)--(kk) node [midway,above] {\scriptsize ${\varepsilon}^{x,\underline{x};x,\underline{x}}_{f^{\kappa_1},g;\tau}$};
\end{tikzpicture}}
  & \resizebox{7cm}{!}{\begin{tikzpicture}
\node(a) at (0,0) {\footnotesize $(f\, {_x\circ _{y}} \, {\it id}_{y,z})^{\sigma}$};
\node(b) at (4.5,0) {\footnotesize  $f^{\sigma_1}\, {_{x'}\circ _{y'}} \, {\it id}_{y,z}^{\sigma_2} $};
\node(1) at (0,-2.5) {\footnotesize  $f^{\tau}$};
\node(2) at (4.5,-2.5) {\footnotesize   $f^{\sigma_1}\, {_{x'}\circ _{y'}} \, {\it id}_{y',\sigma_2^{-1}(z)}$};
\draw[->] (a)--(b) node [midway,above] {\scriptsize $\varepsilon^{x,y;x',y'}_{f,{\it id}_{y,z};\sigma}$};
\draw[->] (a)--(1) node [midway,left] {\scriptsize $(\iota^{x,y}_{f,{\it id}_{y,z}})^{\sigma}$};
\draw[->] (2)--(1) node [midway,above] {\scriptsize $\iota^{x,y}_{f^{\sigma_1},{\it id}_{y',\sigma_2^{-1}(z)}}$};
\draw[->] (b)--(2) node [midway,left] {\scriptsize $1_{f^{\sigma_1}}\, {_{x'}\circ _{y'}} \, \nu^{y',\sigma_2^{-1}(z)}_{y,z}$};
\end{tikzpicture}}
\end{tabular}

\vspace{0.3cm}
\noindent \begin{tabular}{    m{8cm}   m{8cm}  }
 - \hypertarget{gainupe}{{\tt ($\gamma\iota\nu$-\texttt{pentagon})}} & - \hypertarget{gainudpe}{{\tt ($\iota\nu$-\texttt{pentagon})}} \\
\resizebox{7.5cm}{!}{\begin{tikzpicture}
\node(a) at (0,0) {\footnotesize $  {\it id}_{{x},y}\, {_y\circ _{z}} \,   {\it id}_{{z},u}  $};
\node(b) at (4.5,0) {\footnotesize  $ {\it id}_{{z},u}\, {_z\circ _{y}} \,   {\it id}_{{x},y} $};
\node(1) at (0,-2.5) {\footnotesize  ${\it id}_{{x},y}^{\sigma}$};
\node(3) at (2.25,-4.1) {\footnotesize  ${\it id}_{{x},u}$};
\node(2) at (4.5,-2.5) {\footnotesize   $ {\it id}_{{z},u}^{\tau}$};
\draw[->] (a)--(b) node [midway,above] {\scriptsize $\gamma^{y,z}_{ {\it id}_{{x},y}, {\it id}_{{z},u} }$};
\draw[->] (a)--(1) node [midway,left] {\scriptsize $\iota^{y,z}_{{\it id}_{{x},y},{\it id}_{z,u}}$};
\draw[->] (1)--(3) node [midway,left,yshift=-0.1cm] {\scriptsize $\nu^{x,u}_{x,y}$};
\draw[->] (b)--(2) node [midway,right] {\scriptsize $\iota^{z,y}_{{\it id}_{{z},u},{\it id}_{x,y}}$};
\draw[->] (2)--(3) node [midway,right,yshift=-0.1cm] {\scriptsize $\nu^{z,u}_{x,u}$};
\end{tikzpicture}} & \resizebox{7.5cm}{!}{
\begin{tikzpicture}
\node(a) at (0,0) {\footnotesize $  {\it id}^{\sigma}_{{x},y}\, {_u\circ _{z}} \,   {\it id}_{{z},v}  $};
\node(b) at (4.5,0) {\footnotesize  $ {\it id}_{x,u}\, {_u\circ _{z}} \,   {\it id}_{{z},v}   $};
\node(1) at (0,-2.5) {\footnotesize  ${\it id}_{{x},y}^{\sigma\circ\tau}$};
\node(3) at (2.25,-4.1) {\footnotesize  ${\it id}_{{x},v}$};
\node(2) at (4.5,-2.5) {\footnotesize   $ {\it id}_{{x},u}^{\kappa}$};
\draw[->] (a)--(b) node [midway,above] {\scriptsize $\nu^{x,u}_{x,y} \, {_u\circ _{z}} \,   1_{{\it id}_{{z},v}} $};
\draw[->] (a)--(1) node [midway,left] {\scriptsize $\iota^{u,z}_{{\it id}^{\sigma}_{{x},y},{\it id}_{z,v}}$};
\draw[->] (1)--(3) node [midway,left,yshift=-0.1cm] {\scriptsize $\nu^{x,v}_{x,y}$};
\draw[->] (b)--(2) node [midway,right] {\scriptsize $\iota^{z,y}_{{\it id}_{{z},u},{\it id}_{x,y}}$};
\draw[->] (2)--(3) node [midway,right,yshift=-0.1cm] {\scriptsize $\nu^{x,v}_{x,u}$};
\end{tikzpicture}}
\end{tabular}

\vspace{0.3cm}
\noindent \begin{tabular}{   c }
  - \hypertarget{bek}{{\tt ($\nu\sigma$)}} \enspace $  (\nu^{u,v}_{x,y})^{\sigma}=\nu^{p,q}_{x,y}$, where $\sigma$ renames $u$ to $p$ and $v$ to $q$,   
\end{tabular}

 \vspace{0.3cm}
\noindent \begin{tabular}{   c }
   - \hypertarget{begaihex}{{\tt ($\nu$-\texttt{involution})}}  \enspace $\nu^{u,v}_{x,y}=\nu^{v,u}_{x,y}$.
\end{tabular}
 \vspace{0.3cm}

\indent The syntax ${\tt uFree}_{\underline{\EuScript C}}$, necessary for formulating the coherence theorem for this notion of categorified entries-only cyclic operad, is obtained  by extending  the syntax ${\tt Free}_{\underline{\EuScript C}}$ with object and arrow terms that formalize units and the corresponding isomorphisms. Accordingly, the interpretation $_{\tt u}[[-]]_X:{\tt uFree}_{\underline{\EuScript C}}(X)\rightarrow {\EuScript C}(X)$   is obtained by extending naturally the interpretation $[[-]]_X:{\tt Free}_{\underline{\EuScript C}}(X)\rightarrow {\EuScript C}(X)$. The statement of the coherence theorem is then the same as in \S \ref{eq_relax}, except that it concerns the extended interpretation function $_{\tt u}[[-]]_X$. \\[0.1cm]
\indent The proof of the coherence theorem now has {four} reductions: the last three are the same  reductions as in \S \ref{eq_relax}, while the first reduction, which we call the  reduction zero for consistency reasons,  eliminates units, as one might expect.\\[0.1cm]
\indent In the sequel, we shall write $\underline{\it id}_{x,y}$ for the object term of ${\tt uFree}_{\underline{\EuScript C}}$ which   encodes  the unit ${\it id}_{x,y}$.\\[0.1cm]
\indent Let ${\tt uFree}^{\it id}_{\underline{\EuScript C}}$ be the part of the syntax   ${\tt uFree}_{\underline{\EuScript C}}$  determined by taking for object terms only  the instances of $\underline{\it id}_{x,y}$, for $x,y\in V$. Observe that this leaves only the instances of $1_{\underline{\it id}_{x,y}}$ as arrow terms of ${\tt uFree}^{\it id}_{\underline{\EuScript C}}$. The reduction zero is materialized by the reduction function
 $${\tt Red_0}: {\tt uFree}_{\underline{\EuScript C}}\rightarrow {\tt Free}_{\underline{\EuScript C}}\cup {\tt uFree}^{\it id}_{\underline{\EuScript C}},$$ 
which systematically ``eliminates the units'' contained in object terms of ${\tt uFree}_{\underline{\EuScript C}}$, in order to reach either an object term of ${\tt Free}_{\underline{\EuScript C}}$ (if the starting object term contains at least one parameter of $P_{\underline{\EuScript C}}$), or an object term of ${\tt uFree}^{\it id}_{\underline{\EuScript C}}$ (if the starting object term contains no parameters of $P_{\underline{\EuScript C}}$).  Here is the formal definition:
\begin{itemize}
\item[$\diamond$] ${\tt Red_0}(\underline{\it id}_{x,y})=\underline{\it id}_{x,y}$, for all $x,y\in V$,
\item[$\diamond$]    ${\tt Red_0}(\underline{a})=\underline{a}$, for all $a\in P_{\underline{\EuScript C}}$,
\item[$\diamond$] ${\tt Red_0}({\cal W}_1\,{_x\bo_y}\,{\cal W}_2)=\begin{cases}
\underline{\it id}_{u,v}  & {\small \mbox{${\cal W}_1=\underline{\it id}_{x,u}$, ${\cal W}_2=\underline{\it id}_{y,v}$}} \\
{\tt Red_0}({\cal W}_1^{\sigma}) & {\small \mbox{${\cal W}_1\not\in {\tt uFree}^{\it id}_{\underline{\EuScript C}}$, ${\cal W}_2=\underline{\it id}_{y,v}$}}\\
 {\tt Red_0}({\cal W}_2^{\tau}) & {\small \mbox{${\cal W}_1=\underline{\it id}_{x,u}$, ${\cal W}_2\not\in {\tt uFree}^{\it id}_{\underline{\EuScript C}}$}}\\
{\tt Red_0}({\cal W}_1) \,{_x\bo_y}\, {\tt Red_0}({\cal W}_2) & {\small \mbox{${\cal W}_1, {\cal W}_2\not\in {\tt uFree}^{\it id}_{\underline{\EuScript C}}$}}\\
\end{cases}$ \\[0.2cm]
where, in the second case,    $\sigma$ renames $x$ to $v$, and, in the third case,  $\tau$ renames $y$ to $u$.
\item[$\diamond$] ${\tt Red_0}({\cal W}^{\sigma})={\tt Red_0}({\cal W})^{\sigma}$, if ${\cal W}\not\in {\tt uFree}^{\it id}_{\underline{\EuScript C}}$; otherwise, assuming that ${\cal W}=\underline{\it id}_{x,y}$ and that $\sigma:\{x,y\}\rightarrow \{u,v\}$,  ${\tt Red_0}({\underline{\it id}_{x,y}}^{\sigma})=\underline{\it id}_{u,v}$.
\end{itemize}

\indent The  definition of ${\tt Red_0}$ on arrow terms involves a  case analysis relative to the result of the reduction on appropriate  object terms. For
${\tt Red_0}(\beta^{x,\underline{x};y,\underline{y}}_{{\cal W}_1,{\cal W}_2,{\cal W}_3})$,
here are some of the possibilities:
\begin{itemize}
\item if ${\tt Red_0}({\cal W}_i)\not\in{\tt uFree}^{\it id}_{\underline{\EuScript C}}$, $i=1,2,3$, then   ${\tt Red_0}(\beta^{x,\underline{x};y,\underline{y}}_{{\cal W}_1,{\cal W}_2,{\cal W}_3})=\beta^{x,\underline{x};y,\underline{y}}_{{\tt Red_0}({\cal W}_1),{\tt Red_0}({\cal W}_2),{\tt Red_0}({\cal W}_3)},$ 
\item if ${\tt Red_0}({\cal W}_1),{\tt Red_0}({\cal W}_2)\not\in {\tt uFree}^{\it id}_{\underline{\EuScript C}}$ and    ${\tt Red_0}({\cal W}_3) =\underline{\it id}_{\underline{y},v}$, then  ${\tt Red_0}(\beta^{x,\underline{x};y,\underline{y}}_{{\cal W}_1,{\cal W}_2,{\cal W}_3})=\linebreak ({\varepsilon_2}_{{\tt Red_0}({\cal W}_1)}\,{_x\bo_{\underline{x}}}\, 1_{{\tt Red_0}({\cal W}_2)^{\tau}}) \circ {\varepsilon_4}^{x,\underline{x};x,\underline{x}}_{{\tt Red_0}({\cal W}_1),{\tt Red_0}({\cal W}_2);\sigma},$ 
where $\sigma$ and $\tau$ rename $y$ to $v$,
\item if ${\cal W}_1\not\in {\tt uFree}^{\it id}_{\underline{\EuScript C}}$,    ${\cal W}_2=\underline{\it id}_{\underline{x},y}$ and ${\cal W}_3 =\underline{\it id}_{\underline{y},v}$, then ${\tt Red_0}(\beta^{x,\underline{x};y,\underline{y}}_{{\cal W}_1,{\cal W}_2,{\cal W}_3})={\varepsilon_3}^{\tau,\sigma}_{{\tt Red_0}({\cal W}_1)},$ where $\sigma$  renames $y$ to $v$ and $\tau$ renames $x$ to $y$,
\item if ${\tt Red_0}({\cal W}_1)=\underline{\it id}_{x,u}$, ${\tt Red_0}({\cal W}_2)=\underline{\it id}_{\underline{x},y}$ and ${\tt Red_0}({\cal W}_3)=\underline{\it id}_{\underline{y},v}$, then ${\tt Red_0}(\beta^{x,\underline{x};y,\underline{y}}_{{\cal W}_1,{\cal W}_2,{\cal W}_3})=1_{\underline{\it id}_{u,v}}.$
\end{itemize}
The definition of  ${\tt Red_0}$ on the remaining arrow terms is obtained analogously. In particular, ${\tt Red_0}(\iota^{x,y}_{{\cal W},\underline{\it id}_{y,z}})=1_{{\tt Red_0}({\cal W} _x\bo_y {\it id}_{y,z})}$ and ${\tt Red_0}(\nu^{u,v}_{x,y})=1_{{\it id}_{u,v}}$. 
\begin{rem} \label{either-id-or-red} 
Note that, for an arrow-term $\Phi:{\cal U}\rightarrow{\cal V}$ of ${\tt uFree}_{\underline{\EuScript C}}$, we have either that  
${\tt Red}_0({\cal U}),\linebreak {\tt Red}_0({\cal V})\in {\tt Free}_{\underline{\EuScript C}}$, 
or that ${\tt Red}_0({\cal U})=\underline{\it id}_{x,y}={\tt Red}_0({\cal V})$ and ${\tt Red}_0(\Phi)= 1_{\underline{\it id}_{x,y}}$ (for some $x,y$).
\end{rem}
\begin{thm}
For arrow terms $\Phi$ and $\Psi$ of the same type in ${\tt uFree}_{\underline{\EuScript C}}$,  the equality ${_{\tt u}[[\Phi]]}_{X}={_{\tt u}[[\Psi]]}_{X}$ follows from the equality ${[[{\tt Red_0}(\Phi)]]_X}={[[{\tt Red_0}(\Psi)]]_X}$.
\end{thm}
\begin{proof}
The proof architecture is the same as for  Corollary \ref{roknucuse}  and Corollary \ref{cor4}: for all object terms ${\cal U}$ of ${\tt uFree}_{\underline{\EuScript C}}$, one should define arrow terms $\eta_{\cal U}:{\cal U}\rightarrow {\tt Red_0}({\cal U})$ and exhibit commuting diagrams of either of the following two  forms (cf. Remark \ref{either-id-or-red}):
\begin{center}
\begin{tikzpicture}
\node (a) at (0,0) {\footnotesize $_{\tt u}[[{\cal U}]]$};
\node (b) at (3.2,0) {\footnotesize $_{\tt u}[[{\cal V}]]$};
\node (c) at (3.2,-2) {\footnotesize $_{\tt u}[[\underline{\it id}_{x,y}]]$};
\node (d) at (0,-2) {\footnotesize $_{\tt u}[[\underline{\it id}_{x,y}]]$};
\draw [->] (a)--(b) node [midway,above] {\scriptsize $_{\tt u}[[\Phi]]$};
\draw [->]   (b)--(c) node [midway,right] {\scriptsize $_{\tt u}[[\eta_{\cal V}]]$};
\draw [->] (a)--(d) node [midway,left] {\scriptsize $_{\tt u}[[\eta_{\cal U}]]$};
\draw [->] (d)--(c) node [midway,below] {\scriptsize $1_{{\it id}_{x,y}}$};
\end{tikzpicture}
\quad\quad\quad\quad
\begin{tikzpicture}
\node (a) at (0,0) {\footnotesize $_{\tt u}[[{\cal U}]]$};
\node (b) at (3.2,0) {\footnotesize $_{\tt u}[[{\cal V}]]$};
\node (c) at (3.2,-2) {\footnotesize $[[{\tt Red_0}({\cal V})]]$};
\node (d) at (0,-2) {\footnotesize $[[{\tt Red_0}({\cal U})]]$};
\draw [->] (a)--(b) node [midway,above] {\scriptsize $_{\tt u}[[\Phi]]$};
\draw [->]   (b)--(c) node [midway,right] {\scriptsize $_{\tt u}[[\eta_{\cal V}]]$};
\draw [->] (a)--(d) node [midway,left] {\scriptsize $_{\tt u}[[\eta_{\cal U}]]$};
\draw [->] (d)--(c) node [midway,below] {\scriptsize $[[{\tt Red_0}(\Phi)]]$};
\end{tikzpicture}
\end{center}
We shall only
illustrate the proof in the second case, by considering  one instance of the case $\Phi=\beta^{x,\underline{x};y,\underline{y}}_{{\cal W}_1,{\cal W}_2,{\cal W}_3}$, with  ${\cal W}_1=\underline{a}\, {_z\bo_u}\, (\underline{\it id}_{u,v}\, {_v\bo_p}\, \underline{\it id}_{p,q})$,  ${\cal W}_2=\underline{b}$ and  ${\cal W}_3=\underline{\it id}_{\underline{y},z}$.
The corresponding diagram is then the (outer part of the) diagram below:
 \begin{center}
\begin{tikzpicture}
\node (a) at (0,0) {\footnotesize $(({a}\, {_z\circ_u}\, ({\it id}_{u,v}\, {_v\circ_p}\, {\it id}_{p,q}))\,{_x\circ_{\underline{x}}}\, b)\,{_y\circ_{\underline{y}}}\,{\it id}_{\underline{y},z}$};
\node (b) at (10,0) {\footnotesize $({a}\, {_z\circ_u}\, ({\it id}_{u,v}\, {_v\circ_p}\, {\it id}_{p,q}))\,{_x\circ_{\underline{x}}}\, (b\,{_y\circ_{\underline{y}}}\,{\it id}_{\underline{y},z})$};
\node (b') at (10,-8.5) {\footnotesize $({a}\, {_z\circ_u}\, {\it id}_{u,q})\,{_x\circ_{\underline{x}}}\, b^{\kappa_2}$};
\node (c) at (10,-3) {\footnotesize $({a}\, {_z\circ_u}\, {\it id}_{u,v}^{\sigma})\,{_x\circ_{\underline{x}}}\, (b\,{_y\circ_{\underline{y}}}\,{\it id}_{\underline{y},z})$};
\node (d) at (0,-3) {\footnotesize $(({a}\, {_z\circ_u}\, {\it id}_{u,v}^{\sigma})\,{_x\circ_{\underline{x}}}\, b)\,{_y\circ_{\underline{y}}}\,{\it id}_{\underline{y},z}$};
\node (e) at (0,-6) {\footnotesize $(({a}\, {_z\circ_u}\, {\it id}_{u,q})\,{_x\circ_{\underline{x}}}\, b)\,{_y\circ_{\underline{y}}}\,{\it id}_{\underline{y},z}$};
\node (j) at (0,-8.5) {\footnotesize $(({a}\, {_z\circ_u}\, {\it id}_{u,q})\,{_x\circ_{\underline{x}}}\, b)^{\tau}$};
\node (k) at (0,-11) {\footnotesize $({a}^{\kappa_1}\,{_x\circ_{\underline{x}}}\, b)^{\tau}$};
\node (kk) at (10,-11) {\footnotesize ${a}^{\kappa_1}{_x\circ_{\underline{x}}}\, b^{\kappa_2}$};
\node (f) at (10,-6) {\footnotesize $({a}\, {_z\circ_u}\, {\it id}_{u,q})\,{_x\circ_{\underline{x}}}\, (b\,{_y\circ_{\underline{y}}}\,{\it id}_{\underline{y},z})$};
\draw [->] (a)--(b) node [midway,above] {\scriptsize $\beta^{x,\underline{x};y,\underline{y}}_{{a}\, {_z\circ_u}\, ({\it id}_{u,v}\, {_v\circ_p}\, {\it id}_{p,q}),b,{\it id}_{\underline{y},z}}$};
\draw [->]   (b)--(c) node [midway,left,yshift=-0.1cm] {\scriptsize $(1_{a}\, {_z\circ_u}\, \iota_{{\it id}_{u,v},{\it id}_{p,q}} )\,{_x\circ_{\underline{x}}}\, (1_{b \,{_y\circ_{\underline{y}}}\, {\it id}_{\underline{y},z}})$};
\draw [->]   (f)--(b') node [midway,left,xshift=-0.3cm] {\scriptsize $1_{{a} {_z\circ_u}{\it id}_{u,q}}\,{_x\circ_{\underline{x}}}\, \iota^{y,\underline{y}}_{b,{\it id}_{\underline{y},z}}$};
\draw [->]   (j)--(b') node [midway,above] {\scriptsize $\varepsilon^{x,\underline{x};x,\underline{x}}_{{{a} {_z\circ_u}{\it id}_{u,q}},b;\tau}$};
\draw [->]   (b')--(kk) node [midway,left,xshift=-0.3cm] {\scriptsize $\iota^{z,u}_{a,{\it id}_{u,q}}\,{_x\circ_{\underline{x}}}\, 1_{b^{\kappa_2}}$};
\draw [->] (a)--(d) node [midway,right,yshift=0.4cm] {\scriptsize $((1_{a}\, {_z\circ_u}\, \iota_{{\it id}_{u,v},{\it id}_{p,q}} )\,{_x\circ_{\underline{x}}}\, 1_b) \,{_y\circ_{\underline{y}}}\, 1_{{\it id}_{\underline{y},z}}$};
\draw [->] (d)--(e) node [midway,right,yshift=0.4cm] {\scriptsize $((1_{a}\, {_z\circ_u}\, \nu^{u,q}_{u,v} )\,{_x\circ_{\underline{x}}}\, 1_b) \,{_y\circ_{\underline{y}}}\, 1_{{\it id}_{\underline{y},z}}$};
\draw [->] (e)--(j) node [midway,right] {\scriptsize $\iota^{y,\underline{y}}_{({a}\, {_z\circ_u}\, {\it id}_{u,q})\,{_x\circ_{\underline{x}}}\, b,{\it id}_{\underline{y},z}} $};
\draw [->] (j)--(k) node [midway,right] {\scriptsize $(\iota^{z,u}_{a,{\it id}_{u,q}} \,{_x\circ_{\underline{x}}}\, 1_{b})^{\tau} $};
\draw [->] (k)--(kk) node [midway,above] {\scriptsize ${\varepsilon}^{x,\underline{x};x,\underline{x}}_{a^{\kappa_1},b;\tau}$};
\draw [->] (c)--(f) node [midway,left,yshift=-0.1cm] {\scriptsize $(1_{a}\, {_z\circ_u}\, \nu^{u,q}_{u,v} )\,{_x\circ_{\underline{x}}}\, 1_{b \,{_y\circ_{\underline{y}}}\, {\it id}_{\underline{y},z}}$};
\draw [->] (d)--(c) node [midway,above] {\scriptsize $\beta^{x,\underline{x};y,\underline{y}}_{{a}\, {_z\circ_u}\, {\it id}_{u,v}^{\sigma},b,{\it id}_{\underline{y},z}}$};
\draw [->] (e)--(f) node [midway,above] {\scriptsize $\beta^{x,\underline{x};y,\underline{y}}_{{a}\, {_z\circ_u}\, {\it id}_{u,q},b,{\it id}_{\underline{y},z}}$};
\end{tikzpicture}
\end{center}

\noindent which is shown to commute by using the naturality of $\beta$ two times (the top two squares), \hyperlink{begaihex}{{\tt ($\beta\iota\varepsilon$-\texttt{square})}} (the third square) and  \hyperlink{begaihdddexxx}{{\tt ($\varepsilon\iota$-\texttt{square})}} (the bottom square).   
\end{proof}
\section{Alternative presentations of categorified cylic operads}\label{s4}
In this section, we indicate how to recast the definition of categorified cyclic operads of Section \ref{s2} in three alternative frameworks: { exchangeable-output non-skeletal}, { exchangeable-output skeletal} and { entries-only skeletal}.
\subsection{Categorified exchangeable-output  cyclic operads}
In \cite[Theorem 2]{mo}, the equivalence between Definition \ref{entriesonly} and Definition \ref{exoutput} has been worked out in detail. We shall ``categorify'' this equivalence and  translate Definition \ref{cat} to the exchange\-able-output formalism,  thus  synthesising a definition of  non-skeletal  categorified  exchangeable-output   cyclic operads. The skeletal version of this definition is then obtained  by ``categorifying'' the equivalence between non-skeletal and skeletal operads (\cite[Theorem 1.61]{mss}), extended naturally so that it also includes the cyclic structure.
\subsubsection{Non-skeletal  categorified  exchangeable-output   cyclic operads}
 The categorification of Definition \ref{exoutput} is made
 by enriching the structure of a {categorified non-skeletal symmetric operad} ${\EuScript O}$ by {endofunctors} $D_x:{\EuScript O}(X)\rightarrow {\EuScript O}(X)$, whose properties need to be such that the equivalence of \cite[Theorem 2]{mo} is not violated in the weakened setting. In other words, the decision about  whether some axiom of $D_x$ should be weakened or not  must respect the weakening made in passing from entries-only cyclic operads to their categorified version. \\[0.1cm]
\indent   Before we give  the resulting definition, we  adapt the  definition of categorified   { non-symmetric}, {skeletal} operads of \cite{dp} into the definition of { symmetric}, {non-skeletal}  categorified operads. As we did in Section \ref{s2},  we  keep the equivariance  axiom   strict.
\begin{defn}\label{p}
A {\em non-skeletal categorified symmetric operad} is a functor ${\EuScript O}:{\bf Bij}^{\it op}\rightarrow {\bf Cat}$, together with 
\begin{itemize}
\item a family of bifunctors $$\circ_x: {\EuScript O}(X)\times {\EuScript O}(Y)\rightarrow {\EuScript O}(X\backslash\{x\}\cup Y),$$ indexed by arbitrary non-empty finite sets $X$ and $Y$ and element $x\in X$ such that $X\backslash\{x\}\cap Y=\emptyset$, subject to the equivariance axiom: \\[0.1cm]
\indent \hypertarget{[EQ]}{{\texttt{[EQ]}}} for bijections $\sigma_1:X'\rightarrow X$ and $\sigma_2:Y'\rightarrow Y$, $$f^{\sigma_1}\circ_{\sigma_1^{-1}(x) }g^{\sigma_2}=(f\circ_x g)^{\sigma},$$ \phantom{\indent \texttt{[EQ]}} where $\sigma=\sigma_1|^{X\backslash\{x\}}\cup \sigma_2$,
\item two natural isomorphisms, $\beta$ and $\theta$, called {\em sequential associator} and {\em parallel associator}, respectively, whose respective components $$\beta^{x;y}_{f,g,h}:(f\circ_x g)\circ_y h\rightarrow f\circ_x (g\circ_y h)\quad \mbox{ and }\quad \theta^{x;y}_{f,g,h}:(f\circ_x g)\circ_y h\rightarrow (f\circ_y h)\circ_x g\,,$$ are natural in $f$, $g$ and $h$, and are subject to the following coherence conditions:
\begin{itemize}
\item \hypertarget{[b-pentagon]}{{\texttt{[$\beta$-pentagon]}}}  $(1_{{f}}\circ_{x} \beta^{y;z}_{{g},{h},{k}})\circ\beta^{x;z}_{{f},{g}\circ_y {h}, {k}}\circ (\beta^{x;y}_{{f},{g},{h}}\circ_z 1_{{k}})=\beta^{x;y}_{{f},{g},{h}\circ_z {k}}\circ\beta^{y;z}_{{f}\circ_x {g},{h},{k}}$,\\[-0.25cm]
\item \hypertarget{[bt-hexagon]}{{\texttt{[$\beta\theta$-hexagon]}}} $(1_{{f}}\circ_x\theta^{y;z}_{{g},{h},{k}})\circ\beta^{x;z}_{{f},{g}\circ_y {h}, {k}}\circ(\beta^{x;y}_{{f},{g},{h}}\circ_{z}1_{k})=\beta^{x;y}_{{f},{g}\circ_z {h},{k}}\circ(\beta^{x;z}_{{f},{g},{k}}\circ_y 1_{{h}})\circ \theta^{y;z}_{{f}\circ_x  {g}, {h},{k}}$,\\[-0.25cm]
\item  \hypertarget{[bt-pentagon]}{{\texttt{[$\beta\theta$-pentagon]}}}  $\theta^{x;z}_{{f},{g}\circ_y {{h}},{k}}\circ(\beta^{x;y}_{{f},{g},{h}}\circ_z 1_{{k}})=\beta^{x;y}_{{f}\circ_z {k},{g},{h}}\circ(\theta^{x;z}_{{f},{g},{k}}\circ_{y}1_{{h}})\circ\theta^{y;z}_{{f}\circ_x  {g},{h},{k}}$,\\[-0.25cm]
\item  \hypertarget{[t-hexagon]}{{\texttt{[$\theta$-hexagon]}}} $\theta^{x;y}_{{f}\circ_z {k}, {g}, {h}}\circ(\theta^{x;z}_{{f},{g},{k}}\circ_{y}1_{{h}})\circ\theta^{y;z}_{{f}\circ_x {g}, {h},{k}}=(\theta^{y;z}_{{f},{h},{k}}\circ_{x}1_{{g}})\circ\theta^{x;z}_{{f}\circ_y {h},{g},{k}}\circ(\theta^{x;y}_{{f},{g},{h}}\circ_z 1_{{k}})$,\\[-0.25cm]
\item \hypertarget{t-involution}{{\texttt{[$\theta$-involution]}}} $\theta^{y;x}_{{f},{h},{g}}\circ \theta^{x;y}_{{f},{g},{h}}=1_{({f}\circ_x\,{g})\circ_y {h}}$,\\[-0.25cm]
\item \hypertarget{[bs]}{{\texttt{[$\beta\sigma$]}}} if the equality $((f\circ_x g)\circ_y h)^{\sigma}=(f^{\sigma_1}\circ_{x'}g^{\sigma_2})\circ_{y'} h^{\sigma_3}$ holds by  \hyperlink{[EQ]}{{\texttt{[EQ]}}}, then \linebreak
\phantom{\texttt{[$\beta\sigma$]}} $(\beta^{x;y}_{f,g,h})^{\sigma}=\beta^{x';y'}_{f^{\sigma_1},g^{\sigma_2},h^{\sigma_3}}$,\\[-0.25cm]
\item \hypertarget{[ts]}{{\texttt{[$\theta\sigma$]}}} if the equality $((f\circ_x g)\circ_y h)^{\sigma}=(f^{\sigma_1}\circ_{x'}g^{\sigma_2})\circ_{y'} h^{\sigma_3}$ holds by  \hyperlink{[EQ]}{{\texttt{[EQ]}}}, then \linebreak
\phantom{\texttt{[$\theta\sigma$]}} $(\theta^{x;y}_{f,g,h})^{\sigma}=\theta^{x';y'}_{f^{\sigma_1},g^{\sigma_2},h^{\sigma_3}}$,\\[-0.25cm]
\item  \hypertarget{[EQ-mor]}{{\texttt{[EQ-mor]}}} if the equality $(f\circ_x g)^{\sigma}=f^{\sigma_1}\circ_{x'}g^{\sigma_2}$ holds by  \hyperlink{[EQ]}{{\texttt{[EQ]}}}, and if $\varphi:f\rightarrow f'$ and\linebreak
\phantom{\texttt{[EQ-mor]}} $\psi:g\rightarrow g'$, then $(\varphi\circ_x\psi)^{\sigma}=\varphi^{\sigma_1}\circ_{x'}\psi^{\sigma_2}$.\hfill $\square$
\end{itemize}
\end{itemize}
\end{defn}
 
\indent We now introduce the categorification of Definition \ref{exoutput}. Below, for $f\in {\EuScript O}(X)$, $x\in X$ and $y\not\in X\backslash\{x\}$, we write $D^{\EuScript O}_{xy}(f)$ for $D^{\EuScript O}_x(f)^{\sigma}$, where $\sigma$ renames $x$ to $y$. 
\begin{defn}\label{ex_out_cat_non-skeletal}
A {\em  non-skeletal  categorified  exchangeable-output   cyclic operad} is a non-skeletal categorified symmetric operad ${\EuScript O}$, together with 
\begin{itemize}
\item a family of endofunctors $$D_x:{\EuScript O}(X)\rightarrow {\EuScript O}(X),$$ indexed by arbitrary finite sets $X$ and elements $x\in X$, which are subject to the following axioms, in which $f$ and $g$ denote operadic operations  and $\varphi$ and $\psi$   morphisms between operadic operations:\\[0.2cm]
\indent  \hypertarget{[DIN]}{\texttt{[DIN]}} $D_{x}(D_{x}(f))=f$ and $D_{x}(D_{x}(\varphi))=\varphi$,\\[0.2cm]
\indent \hypertarget{[DEQ]}{\texttt{[DEQ]}} $ D_x(f)^{\sigma} = D_{\sigma^{-1}(x)}(f^{\sigma})$ and $ D_x(\varphi)^{\sigma} = D_{\sigma^{-1}(x)}(\varphi^{\sigma})$, where $\sigma:Y\rightarrow X$ is a  bijection,\\[0.2cm] 
\hypertarget{[DEX]}{\texttt{[DEX]}} $D_{x}(f)^{\sigma}=D_x(D_y(f))$ and $D_{x}(\varphi)^{\sigma}=D_x(D_y(\varphi))$, where $\sigma:X\rightarrow X$  exchanges $x$ \linebreak 
\phantom{\texttt{[DEX]}} and $y$, \\[0.2cm]
\indent  \hypertarget{[DC1]}{{\texttt{[DC1]}}} $D_{y}(f\circ_{x}g)=D_{y}(f)\circ_{x}g$ and $D_{y}(\varphi\circ_{x}\psi)=D_{y}(\varphi)\circ_{x}\psi $ , where   $y\in X\backslash\{x\}$, \\[0.2cm]
\indent \hypertarget{[Db]}{{\tt{[$D\beta$]}}} $D_z(\beta^{x;y}_{f,g,h})=\beta^{x;y}_{D_{z}(f),g,h}$, where $f\in {\EuScript O}(X)$, $g\in {\EuScript O}(y)$, $h\in {\EuScript O}(Z)$, $x,z\in X$   and $y\in Y$,  \\[0.2cm]
\indent \hypertarget{[Dt]}{{\tt{[$D\theta$]}}} $D_z(\theta^{x;y}_{f,g,h})=\theta^{x;y}_{D_{z}(f),g,h}$, where $f\in {\EuScript O}(X)$, $g\in {\EuScript O}(y)$, $h\in {\EuScript O}(Z)$ and $x,y,z\in X$,  

\item a natural isomorphism $\alpha$, called the {\em exchange}, whose components $$\alpha^{y,x;v}_{f,g}:D_{y}(f\circ_x g)\rightarrow D_{yv}(g)\circ_v D_{xy}(f),$$ are natural in $f$ and $g$, and are subject to the following coherence conditions: \\[0.1cm]
  - \hypertarget{abt-square}{{\tt{[$\alpha\beta\theta$-square]}}} for $f\in {\EuScript O}(X)$, $g\in {\EuScript O}(y)$, $h\in {\EuScript O}(Z)$, $x\in X$   and $y,z\in Y$, the following  \linebreak
\phantom{ii} diagram commutes
 \begin{center}
\begin{tikzpicture}
    \node (E) at (-3,0) {\small $D_z((f\circ_x g)\circ_y h)$};
    \node (F) at (-3,-2) {\small $D_z(f\circ_x (g\circ_y h))$};
    \node (dF) at (1.95,-2) {\small $D_{zv}(g\circ_y h)\circ_v D_{xz}(f)$};
    \node (A) at (0,0) {\small $D_{z}(f\circ_x g)\circ_y h$};
    \node (As) at (6,0) {\small $(D_{zv}(g)\circ_v D_{xz}(f))\circ_y h$};
    \node (Asubtr) at (6,-2) {\small $(D_{zv}(g)\circ_y h)\circ_v  D_{xz}(f)$};
\draw[double equal sign distance] (E)--(A) node[midway, above] {};
\draw [->] (A)--(As) node[midway, above] {\footnotesize $\alpha^{z,x;v}_{f,g} \circ_y 1_{h}$};
\draw [->] (F)--(dF) node[midway, above] {\footnotesize $\alpha^{z,x;v}_{f,g\circ_y h}$};
\draw[->] (E)--(F) node[midway, left] {\footnotesize $D_z(\beta^{x;y}_{f,g,h})$};
\draw[double equal sign distance] (dF)--(Asubtr) node[midway, above] {};
\draw[->] (As)--(Asubtr) node[midway,right] {\footnotesize $\theta^{v;y}_{D_{zv}(g),D_{xz}(f),h}$};
 \end{tikzpicture}
\end{center}
  -  \hypertarget{ab-hexagon}{{\tt{[$\alpha\beta$-hexagon]}}} for $f\in {\EuScript O}(X)$, $g\in {\EuScript O}(y)$, $h\in {\EuScript O}(Z)$, $x\in X$, $y\in Y$ and $z\in Z$, the following \linebreak \phantom{ii} diagram commutes
\begin{center}
\begin{tikzpicture}
    \node (E) at (-3.3,0) {\small $D_z((f\circ_x g)\circ_y h)$};
    \node (F) at (-4,-1.75) {\small $D_z(f\circ_x(g\circ_y h))$};
    \node (A) at (3.3,0) {\small $D_{zv}(h)\circ_{v}D_{yz}(f\circ_x g)$};
    \node (Asubt) at (-4,-3) {\small $D_{zv}(g\circ_y h)\circ_v D_{xz}(f)$};
    \node (Asubtr) at (3.3,-4.65) {\small $(D_{zv}(h)\circ_v D_{yv}(g))\circ_v {D_{xz}(f)}$};
 \node (Adsubtr) at (-3.3,-4.65) {\small $(D_{zv}(h)\circ_v D_{yz}(g))^{\sigma}\circ_v {D_{xz}(f)}$};
    \node (P4) at (4,-1.75) {\small $D_{zv}(h)\circ_{v} (D_{yv}(g)\circ_v {D_{xy}(f)})^{\sigma}$};
 \node (P5) at (4,-3) {\small $D_{zv}(h)\circ_{v} (D_{yv}(g)\circ_v {D_{xz}(f)}) $};
\draw[->] (E)--(A) node[midway, above] {\footnotesize $\alpha^{z,y;v}_{f\circ_x g,h}$};
\draw[->] (E)--(F) node[midway, left] {\footnotesize $D_z(\beta^{x;y}_{f,g,h})$};
\draw[->] (F)--(Asubt) node[midway, left] {\footnotesize $\alpha^{z,x;v}_{f,g\circ_y h}$};
\draw[->] (Asubt)--(Adsubtr) node[midway, left] {\footnotesize $(\alpha^{z,y;v}_{g,h})^{\sigma}\circ_v 1_{D_{xz}(f)}$};
\draw[->] (P5)--(Asubtr) node[midway,right] {\footnotesize $\beta^{v;x\, ^{_{-1}}}_{D_{zv}(h),D_{yz}(g),D_{xz}(f)}$};
\draw[->] (A)--(P4) node[midway, right] {\footnotesize $1_{D_{zv}(h)}\circ_{v} (\alpha^{y,x;v}_{f,g})^{\tau}$};
\draw [double equal sign distance] (P4)--(P5) node[midway, above] {};
\draw [double equal sign distance] (Adsubtr)--(Asubtr) node[midway, above] {};
   \end{tikzpicture}
\end{center}

\indent \phantom{a} where $\sigma$ renames $z$ to $v$ and $\tau$  renames $y$ to $z$,

  - \hypertarget{[Da]}{{\tt{[$D\alpha$]}}} $D_z(\alpha^{z,x;v}_{f,g})=\alpha^{z,v;u}_{D_{zv}(g),D_{xz}(f)}$, where $f\in {\EuScript O}(X)$, $g\in {\EuScript O}(Y)$ and $z\in Y$,\\[0.25cm]
  - \hypertarget{[as]}{{\tt{[$\alpha\sigma$]}}} if  the equality $(f\circ_x g)^{\sigma}=f^{\sigma_1}\circ_{\sigma_1^{-1}(x)} g^{\sigma_2}$ holds by  \hyperlink{[EQ]}{{\texttt{[EQ]}}}, then $$(\alpha^{z,x;v}_{f,g})^{\sigma}=\alpha^{\sigma^{-1}(z),\sigma_1^{-1}(x);w}_{f^{\sigma_1},g^{\sigma_2}},$$ \phantom{a} where $v\not\in X\backslash\{x\}\cup Y\backslash\{z\}$ and $w\not\in \sigma^{-1}[X\backslash\{x\}\cup Y\backslash\{z\}]$ are arbitrary variables.  \hfill$\square$
\end{itemize} 
\end{defn}

By comparing Definition \ref{ex_out_cat_non-skeletal} with Definition \ref{exoutput}, one sees that the only axiom of $D_{x}$ from Definition \ref{exoutput} that got weakened is  \hyperlink{[DC2]}{\texttt{[DC2]}}. Indeed, the proof of {\cite[Theorem 2]{mo}}  testifies that all the axioms of $D_{x}$, except \hyperlink{[DC2]}{\texttt{[DC2]}}, are proved by the functoriality and the equivariance of the corresponding entries-only cyclic operad, while the proof of \hyperlink{[DC2]}{\texttt{[DC2]}} requires the axiom \hyperlink{(CO)}{{\tt{(CO)}}}. Therefore, since  \hyperlink{(CO)}{{\tt{(CO)}}} gets weakened in passing from cyclic operads to categorified cyclic operads, \hyperlink{[DC2]}{\texttt{[DC2]}} has to be weakened too.  Henceforth, we shall restrict ourselves to constant-free cyclic operads (as required by \cite[Theorem 2]{mo}).
\begin{thm}\label{t5}
Definition \ref{cat} and Definition \ref{ex_out_cat_non-skeletal} are equivalent definitions of categorified cyclic operads.
\end{thm}
\begin{proof} The proof follows by ``categorifying'' the proof of  \cite[Theorem 2]{mo}.  In the table below, we show how  the coherence conditions of the entries-only definition  imply   the coherence conditions of the exchangeable-output definition, which, in particular, reveals the correspondence between the canonical isomorphisms of the two structures.

\begin{center}
{\small  \begin{tabular}{lr}  
    \toprule
   \textsc{entries-only} \quad\quad\quad\quad\quad\quad\quad\quad\quad $\Rightarrow$  &  \textsc{exchangeable-output}  \\
    \midrule
 \hyperlink{b-pentagon}{{\tt ($\beta$-\texttt{pentagon})}} &  \hyperlink{[b-pentagon]}{{\texttt{[$\beta$-pentagon]}}} \\[0.1cm]
   \hyperlink{bg-decagon}{\texttt{($\beta \gamma$-decagon)}} &   \hyperlink{[bt-hexagon]}{\texttt{[$\beta\theta$-hexagon]}}    \\[0.1cm]
\hyperlink{bvt-pentagon}{\texttt{($\beta\vartheta$-pentagon)}} (Lemma  \ref{thetainverse}.2) &  \hyperlink{[bt-pentagon]}{\texttt{[$\beta\theta$-pentagon]}}      \\[0.1cm]
\hyperlink{vt-hexagon}{\texttt{($\vartheta$-hexagon)}} (Lemma  \ref{thetainverse}.3) &   \hyperlink{[t-hexagon]}{{\texttt{[$\theta$-hexagon]}}}       \\[0.1cm]
\hyperlink{vt-involution}{\texttt{($\vartheta$-involution)}} (Lemma \ref{thetainverse}.1)  & \hyperlink{[t-involution]}{{\texttt{[$\theta$-involution]}}} \\[0.1cm]
 \hyperlink{bs}{\texttt{($\beta\sigma$)}} & \hyperlink{[bs]}{{\texttt{[$\beta\sigma$]}}} \\[0.1cm]
\hyperlink{bs}{\texttt{($\beta\sigma$)}},  \hyperlink{gs}{\texttt{($\gamma\sigma$)}}, \hyperlink{(EQ-mor)}{\texttt{(EQ-mor)}} & \hyperlink{[ts]}{{\texttt{[$\theta\sigma$]}}} \\[0.1cm]
\hyperlink{(EQ-mor)}{\texttt{(EQ-mor)}} & \hyperlink{[EQ-mor]}{{\texttt{[EQ-mor]}}} \\[0.1cm]
Remark \ref{eqdis}.5, \hyperlink{(EQ)}{\texttt{(EQ)}} & \hyperlink{[DIN]}{\texttt{[DIN]}},   \hyperlink{[DEQ]}{\texttt{[DEQ]}},   \hyperlink{[DEX]}{\texttt{[DEX]}},   \hyperlink{[DC1]}{\texttt{[DC1]}} \\[0.1cm]
\hyperlink{bs}{\texttt{($\beta\sigma$)}} &  \hyperlink{[Db]}{\tt{[$D\beta$]}} \\[0.1cm ]
\hyperlink{bs}{\texttt{($\beta\sigma$)}},  \hyperlink{gs}{\texttt{($\gamma\sigma$)}}, \hyperlink{(EQ-mor)}{\texttt{(EQ-mor)}} &  \hyperlink{[Dt]}{\tt{[$D\theta$]}} \\[0.1cm]
 \eqref{theta} & \hyperlink{abt-square}{{\tt{[$\alpha\beta\theta$-square]}}} \\[0.1cm]
\hyperlink{bg-hexagon}{\texttt{($\beta \gamma$-hexagon)}}, \hyperlink{(EQ)}{\texttt{(EQ)}}, \hyperlink{bs}{\texttt{($\beta\sigma$)}},  \hyperlink{gs}{\texttt{($\gamma\sigma$)}}, \hyperlink{(EQ-mor)}{\texttt{(EQ-mor)}} & \hyperlink{ab-hexagon}{{\tt{[$\alpha\beta$-hexagon]}}} \\[0.1cm]
\hyperlink{gs}{\texttt{($\gamma\sigma$)}}, \hyperlink{g-involution}{{\tt ($\gamma$-{\texttt{involution}})}} & \hyperlink{[Da]}{{\tt{[$D\alpha$]}}}\\[0.1cm]
\hyperlink{gs}{\texttt{($\gamma\sigma$)}} & \hyperlink{[as]}{{\tt{[$\alpha\sigma$]}}}\\[0.1cm]
 \bottomrule
  \end{tabular} }
\end{center}
\end{proof}
\begin{rem} It is easily seen that relaxing other axioms  of Definition \ref{exoutput} (besides \hyperlink{[DC2]}{\texttt{[DC2]}}) would correspond, in the entries-only formalism, to relaxing the equivariance and the action of the symmetric groups  (cf. Section \ref{eq_relax}).
\end{rem}
\subsubsection{Skeletal categorified  exchangeable-output    cyclic operads}\label{aaaa}

The categorification of  skeletal  exchangeable-output  cyclic operads (\cite[Proposition 42]{opsprops}) is obtained by  ``categorifying'' the equivalence of non-skeletal and skeletal operads (\cite[Theorem 1.61]{mss}), extended naturally so that it also includes the endofunctors $D_x:{\EuScript O}(X)\rightarrow {\EuScript O}(X)$ (for  non-skeletal operads) and $D_i:{\EuScript O}(n)\rightarrow {\EuScript O}(n)$ (for skeletal operads).  \\[0.1cm]
\indent     
It would be tempting to say that we could have replaced ``non-skeletal'' with ``skeletal'' throughout the paper, adjusting the proofs. But this is not the case: as we explain in the remark below, non-skeletality turns out to be crucial for the rewriting involved in our
proof of coherence in the presence of symmetries in Section \ref{s2}.
\begin{rem}\label{skeletal_problem}
In the non-skeletal setting of  (cyclic) operads, an action of the symmetric group can always be ``pushed" from the composite of two operations   to the operations themselves, by directing  \hyperlink{(EQ)}{\texttt{(EQ)}} from right to left.
 This was essential for the  first reduction made in \S \ref{frst}. For the skeletal setting of (cyclic) operads, this distribution of actions of the symmetric group doesn't work in general, as pointed to us  by Petri\' c.  For example, let ${\EuScript O}:\Sigma^{\it op}\rightarrow {\bf Set}$ be a (skeletal) operad and let $f,g\in {\EuScript O}(2)$. Consider the term $(f\circ_2 g)^{\sigma}$, where $\sigma:[3]\rightarrow[3]$  is defined by $\sigma(1)=2$, $\sigma(2)=1$  and $\sigma(3)=3$.  Clearly, it is not possible to distribute $\sigma$ on $f$ and $g$ in ${\EuScript O}(3)$.

\end{rem}
Remark \ref{skeletal_problem} shows that the part of our coherence proof technique that eliminates symmetries would not work in the skeletal setting. Nevertheless, the coherence in the skeletal setting holds by reduction to the non-skeletal setting, followed by all the reductions of Section \ref{s2}.
\subsection{Skeletal categorified entries-only cyclic operads}
We give below a definition of skeletal entries-only cyclic operads. Definition \ref{co_skeletal} is a variation on a similar definition which appears in  the unpublished manuscript \cite{djms}. We shall omit its equivalence with the non-skeletal entries-only definition  (whose categorification is  Definition \ref{cat} here), which parallels the equivalence between skeletal and non-skeletal versions of cyclic operads in the exchangeable-output setting that we exploited above.   Again, this definition can be categorified following the same approach as before. As an illustration, we shall describe  the translation   of \hyperlink{bg-hexagon}{{\tt ($\beta \gamma$-\texttt{hexagon})}} to the skeletal framework. 

In the remainder of this section, as well as throughout Section \ref{s5}, we shall denote  the set of   permutations on the set $[n]=\{1,\dots,n\}$ with ${\mathbb S}_{n}$. We shall use the notation ${\tt c}{\mathbb S}_{n}$ for the set of cyclic permutations on $[n]$.
 
 \subsubsection{Skeletal entries-only cyclic operads}

\begin{defn}\label{co_skeletal}
A  {\em skeletal entries-only cyclic operad} is a functor ${\EuScript C}:{\bf \Sigma}^{op}\rightarrow {\bf Set}$,    together with a family of functions
$${{_{i}\circ_{j}}}:{\EuScript C}(m)\times {\EuScript C}(n)\rightarrow {\EuScript C}(m+n-2) ,$$
defined for arbitrary $m,n>1$, $1\leq i\leq m$ and $1\leq j\leq n$.
These data must satisfy the  axioms given below.  \\[0.2cm]
{\em Sequential associativity.} For $f\in {\EuScript C}(m)$, $g\in {\EuScript C}(n)$, $h\in {\EuScript C}(p)$, $1\leq i\leq m$, $1\leq j\leq n$, $i\leq k\leq i+n-2$ and $1\leq l\leq p$, the following   equality holds:\\[0.25cm]
\indent \hypertarget{sA1}{\texttt{(sA1)}} $(f\, {_{i}\circ_{j}}\,\, g)\,\,{_{k}\circ_l}\, h =  \begin{cases} 
    f\,{{_{i}\circ_{j}}}\, (g\,{{_{k-i+j+1}\circ_{l}}}\, h),  &  i\leq k\leq i+n-j-1\\
    f\,{{_{i}\circ_{j+p-1}}}\, (g\,{{_{k-i-n+j+1}\circ_{l}}}\, h),   &  i+n-j\leq k\leq i+n-2 
   \end{cases}.
$  \\[0.25cm]
{\em Commutativity.} For  $f\in{\EuScript C}(m)$, $g\in {\EuScript C}(n)$, $1\leq i\leq m$ and $1\leq j\leq n$,  the following equality holds:\\[0.25cm]
\indent \hypertarget{sCO}{\texttt{(sCO)}} $f\, {_i\circ_j} \,\, g=(g\, {_j\circ_i} \,\, f)^{\sigma}$, \\[0.25cm] 
where $\sigma\in {\tt c}{\mathbb S}_{m+n-2}$ is   determined by $\sigma(1)=j+m-i$.\\[0.25cm]
{\em Equivariance.} For $\sigma_1\in{\mathbb S}_m$, $\sigma_2\in{\mathbb S}_n$,  $f\in{\EuScript C}(m)$, $g\in {\EuScript C}(n)$, $1\leq i\leq m$ and $1\leq j\leq n$, the following equality holds:\\[0.25cm] 
\indent \hypertarget{sEQ}{\texttt{(sEQ)}} $f^{\sigma_1}\,\,{_{{ {\sigma_1^{-1}}(i)}}\circ_{\sigma_2^{-1}(j)}}\,\, g^{\sigma_2}=(f\, {_i\circ_j} \, g)^{\sigma}$,\\[0.25cm]
where    $\sigma\in {\mathbb{S}}_{m+n-2}$ is defined by $$\sigma=\kappa^{m,n}_{\sigma_1^{-1}(i),\sigma_2^{-1}(j)}\circ (\sigma_1|_{[m]\backslash\{i\}}+\sigma_2|_{[n]\backslash\{j\}})\circ({\kappa^{m,n}_{i,j}})^{-1},$$ where
$$\kappa^{m,n}_{i,j}:[m]\backslash\{i\}+[n]\backslash\{j\}\rightarrow [m+n-2]$$
 is defined by
$$\kappa^{m,n}_{i,j}(k)= \begin{cases} 
    k,   &  1\leq k \leq  i-1 \\
  k+n-2, &  i+1\leq k \leq m \\
   \end{cases}$$ 
for $k\in [m]\backslash\{i\}$, and by 

 $$\kappa^{m,n}_{i,j}(l)= \begin{cases} 
    l+i+n-j-1, &  1\leq l \leq j-1 \\
    l+i-j-1,  &  j+1\leq l \leq n\\
   \end{cases}$$
for  $l\in [n]\backslash\{j\}$.
 \hfill$\square$
\end{defn}
\indent The bijection $\kappa^{m,n}_{i,j}$ provides a cyclic order of  the entries of the composite $f\, {_i\circ_j} \, g$ from the cyclic orders of the entries of  $f$ and $g$. 
The correspondence given by $\kappa^{m,n}_{i,j}$ can be read from the picture below: 
\begin{center}
\resizebox{6.25cm}{!}{\begin{tikzpicture}
\node (F) [circle, draw=black,inner sep=0.5mm,minimum size=0.6cm] at (0,0) {$f$};
\node (d1) [circle, draw=none,inner sep=0.5mm] at (0,1.4) {\footnotesize $\cdots$};
\node (d2) [circle, draw=none,inner sep=0.5mm] at (0,-1.4) {\footnotesize $\cdots$};
\node (i) [circle, draw=none,inner sep=0.5mm] at (1.75,0.2) {\tiny\color{red} $i$};
\node (j) [circle, draw=none,inner sep=0.5mm] at (2.25,0.2) {\tiny\color{blue} $j$};
\node (i-1) [circle, draw=none,inner sep=0.5mm] at (0.6,-0.35) {\tiny\color{red} $i\!-\!1$};
\node (j-1) [circle, draw=none,inner sep=0.5mm] at (3.4,0.35) {\tiny\color{blue} $j\!-\!1$};
\node (i+1) [circle, draw=none,inner sep=0.5mm] at (0.6,0.35) {\tiny\color{red} $i\!+\!1$};
\node (j+1) [circle, draw=none,inner sep=0.5mm] at (3.4,-0.35) {\tiny\color{blue} $j\!+\!1$};
\node (1) [circle, draw=none,inner sep=0.5mm] at (-0.15,-0.45) {\tiny\color{red} $1$};
\node (1') [circle, draw=none,inner sep=0.5mm] at (4.15,0.45) {\tiny\color{blue} $1$};
\node (m) [circle, draw=none,inner sep=0.5mm] at (-0.15,0.45) {\tiny\color{red} $m$};
\node (n) [circle, draw=none,inner sep=0.5mm] at (4.15,-0.45) {\tiny\color{blue} $n$};
\node (d3) [circle, draw=none,inner sep=0.5mm] at (4,1.4) {\footnotesize $\cdots$};
\node (d4) [circle, draw=none,inner sep=0.5mm] at (4,-1.4) {\footnotesize $\cdots$};
\node (F1) [rectangle, draw=none,inner sep=0.3mm] at (1,1.4) {\footnotesize $\enspace i\!+\!n\!-\!1$};
\node (F2) [rectangle, draw=none,inner sep=0.3mm] at (-1,1.4) {\footnotesize $m\!+\!n\!-\!2\enspace $};
\node (F3) [rectangle, draw=none,inner sep=0.3mm] at (-1,-1.4) {\footnotesize $1$};
\node (F4) [rectangle, draw=none,inner sep=0.3mm] at (1,-1.4) {\footnotesize $i\!-\!1$};
\node (G1) [rectangle, draw=none,inner sep=0.3mm] at (5,1.4) {\footnotesize $\enspace i\!+\!n\!-\!j$};
\node (G2) [rectangle, draw=none,inner sep=0.3mm] at (3.1,1.4) {\footnotesize $i\!+\!n\!-\!2\enspace $};
\node (G3) [rectangle, draw=none,inner sep=0.3mm] at (3.1,-1.4) {\footnotesize $i$};
\node (G4) [rectangle, draw=none,inner sep=0.3mm] at (5,-1.4) {\footnotesize $\enspace\quad i\!+\!n\!-\!j\!-\!1$};
\node (G) [circle, draw=black,inner sep=0.5mm,minimum size=0.6cm] at (4,0) {$g$};
\draw (F)--(G);
\draw (F)--(F1);
\draw (F)--(F2);
\draw (F)--(F3);
\draw (F)--(F4);
\draw (G)--(G1);
\draw (G)--(G2);
\draw (G)--(G3);
\draw (G)--(G4);
\draw (2,-0.15)--(2,0.15);
 \draw [dashed] (2,0) ellipse (2.85cm and 0.95cm);
\end{tikzpicture}}
\end{center}
\subsubsection{Skeletal version of {{\tt ($\beta \gamma$-\texttt{hexagon})}}} \label{s-hexa}
Suppose that, in Definition \ref{co_skeletal}, the category {\bf Set} has been replaced by {\bf Cat}, the functions ${_i\circ_j}$ by the appropriate bifunctors  and axioms \hyperlink{sA1}{\texttt{(sA1)}} and \hyperlink{sCO}{\texttt{(sCO)}} by the appropriate  isomorphisms.

\indent Let $f\in {\EuScript C}(m)$, $g\in {\EuScript C}(n)$, $h\in {\EuScript C}(p)$, $1\leq i\leq m$, $1\leq j\leq n$, $1\leq l\leq p$ and   $i\leq k\leq i+n-j-1$ (the case $i+n-j\leq k\leq i+n-2$ is treated analogously). Thanks to the equivariance axiom \hyperlink{sEQ}{\texttt{(sEQ)}}, \hyperlink{bg-hexagon}{{\tt ($\beta \gamma$-\texttt{hexagon})}} gets translated into the following diagram:\\[0.2cm]
\indent - \hypertarget{gfgf}{{\tt {(s-$\beta \gamma$-\texttt{hexagon})}}}
\begin{center}
\begin{tikzpicture}
    \node (E) at (0,0) {\small $(f\,{_i\circ _{j}}\,g)\,{_k\circ _{{l}}}\,h$};
    \node (G) at (4.65,0) {\small $f\,{_i\circ _{{j}}}\,(g\,{_{k-i+j+1}\circ _{{l}}}\,h)$};
\node (Gb) at (4.65,-3.8) {\small $(h\,{_{l}\circ _{\tau_2(k)}}\,(g\,{_{j}\circ _{i}}\,f))^{\kappa}$};
    \node (F) at (11,0) {\small $((g\,{_{k-i+j+1}\circ _{l}}\,h)\,{_{j}\circ _{i}}\,f)^{\sigma_1}$};
    \node (A) at (0,-1.9) {\small $(g\,{_{j}\circ _{i}}\,f)^{\tau_2}\,{_k\circ _{l}}\,h$};
    \node (Asubt) at (0,-3.8) {\small $(h\,{_{l}\circ _k}\,(g\,{_{j}\circ _{i}}\,f)^{\tau_2})^{\sigma_2}$};
 \node (Asubt1) at (11,-3.8) {\small $((h\,{_{{l}}\circ _{k-i+j+1}}\,g)\,{_{\tau_1(j)}\circ _{i}}\,f)^{\kappa}$};
    \node (P4) at (11,-1.9) {\small $((h\,{_{l}\circ _{k-i+j+1}}\,g)^{\tau_1}\,{_{j}\circ _{i}}\,f)^{\sigma_1}$};
    \draw[->] (E)--(G) node [midway,above] {\scriptsize     $\beta^{i,j;k,l}_{f,g,h}$};
    \draw[->] (G)--(F) node [midway,above] {\scriptsize    $\gamma^{i,j}_{f,g{_{k-i+j+1}\circ _{l}}h}$};
    \draw[->] (E)--(A) node [midway,left] {\scriptsize    $\gamma^{i,j}_{f,g}\,{_k\circ _{l}}\,{1_h}$};
 \draw[->] (F)--(P4) node [midway,right] {\scriptsize    $(\gamma^{k-i+j+1,l}_{g,h}\,{_{j}\circ _{i}}\,{1_f})^{\sigma_1}$};
    \draw[->] (A)--(Asubt) node [midway,left]  {\scriptsize    $\gamma^{k,l}_{(g{_{j}\circ _{i}}f)^{\tau_2},h}$};
 \draw[double equal sign distance] (Asubt)--(Gb) node [midway,left]  {};
    \draw[double equal sign distance] (P4)--(Asubt1) node [midway,right] {};
\draw[->] (Asubt1)--(Gb) node [midway,above] {\scriptsize    $(\beta^{l,k-i+j+1;\tau_1(j),i}_{h,g,f})^{\kappa}$};
   \end{tikzpicture}
\end{center}
in which
\begin{itemize}
\item $\sigma_1\in {\tt c}{\mathbb S}_{m+n+p-3}$ is determined by $\sigma_1(1)=j+m-i$,
\item $\tau_1\in {\tt c}{\mathbb S}_{n+p-2}$ is determined by $\tau_1(1)=l+n-k+i-j-1$,
\item $\tau_2\in {\tt c}{\mathbb S}_{m+n-2}$ is determined by $\tau_2(1)=j+m-i$, 
\item $\sigma_2\in {\tt c}{\mathbb S}_{m+n+p-3}$ is determined by $\sigma(1)=l+m+n-2-k$, and
\item $\kappa\in{\tt c}{\mathbb S}_{m+n+p-3}$ is defined by $\kappa = \kappa^{p,m+n-2}_{l,k}\circ({\it id}_{[p]\backslash\{l\}}+\tau_2|_{[m+n-2]\backslash\{\tau_2(k)\}})\circ\kappa^{p,m+n-2}_{l,\tau_2(k)}\circ\sigma_2$,   the composition defining $\kappa$ being equal to   $\kappa^{p+n-2,m}_{j,i}\circ(\tau_1|_{[p+n-2]\backslash\{\tau_1(j)\}}+{\it id}_{[m]\backslash\{i\}})\circ\kappa^{p+n-2,m}_{\tau_1(j),i}\circ\sigma_1.
$
\end{itemize}

Similarly, we can obtain the skeletal versions of the other coherence diagrams of Definition \ref{cat}, and  arrive to a formal full definition of skeletal entries-only categorified cyclic operad. We omit the details.

\section{Categorified cyclic operads ``in nature''}\label{s5}
In this section, we first give an example of a categorified entries-only cyclic operad  obtained by the principal categorification (Section \ref{s2}), in the form of an easy generalization
of the structure of profunctors of B\' enabou \cite{benabou2}. We then show how to exploit the coherence conditions of categorified entries-only cyclic operads for which all the axioms have been relaxed (Section \ref{s3}) in proving that the Feynman category for cyclic operads, introduced by Kaufmann and Ward  in  \cite{kauf}, admits an odd version.
\subsection{Generalized profunctors as categorified entries-only cyclic operads} 
In this example, the categorified  entries-only cyclic operad  structure arises because the cartesian product of sets (which figures in the definition of the composition of profunctors, and a fortiori of generalized profunctors) is neither associative nor commutative on the nose.  On the other hand  equivariance  holds strictly. It is more natural to work here in a skeletal setting.

\subsubsection{${\bf D}^{n}$-profunctors}
Recall that, for small categories ${\bf C}$ and ${\bf D}$, a {profunctor from ${\bf C}$ to ${\bf D}$} is a functor $F:{\bf D}^{\it op}\times {\bf C}\rightarrow {\bf Set}$, denoted usually by $F:{\bf C} \tobar {\bf D}$.
The composition of profunctors $F:{\bf C} \tobar {\bf D}$ and $G:{\bf D} \tobar {\bf E}$ is a profunctor $G\circ F:{\bf C} \tobar {\bf E}$ defined by $$G\circ F=\int^{d}F(d,-)\times G(-,d).$$
Small categories, profunctors (with the above composition) and natural transformations yield the bicategory  ${\bf Prof}$  of profunctors. Profunctors first appeared in the work \cite{benabou2} of  B\' enabou, under the name of { distributors}.

\indent Let ${\bf D}$ be a small category equipped with an isomorphism  $(-)^{\ast}:{\bf D}^{\it op}\rightarrow {\bf D}$, called a duality hereafter. 
Given that additional structure, a profunctor $F:{\bf D}^{\it op}\times {\bf D}\rightarrow {\bf Set}$ is canonically isomorphic to the profunctor $F'\circ ((-)^{\ast}\times 1_{D})$, where $F':{\bf D}\times {\bf D}\rightarrow {\bf Set}$ is defined by $F'(x,y)=F(x^{\ast},y)$. Therefore, in the presence of a duality, a profunctor $F:{\bf D}^{\it op}\times {\bf D}\rightarrow {\bf Set}$ can be  considered as a functor $F:{\bf D}\times {\bf D}\rightarrow {\bf Set}$. We shall call a functor of this type a ${\bf D}^{2}$-{\em profunctor}. More generally, we shall call a functor $F:{\bf D}^{n}\rightarrow {\bf Set}$, where $n\geq 2$, a ${\bf D}^{n}$-{\em profunctor}. Let us denote the functor category $[{\bf D}^{n},{\bf Set}]$ with ${\EuScript C}(n)$.\\[0.1cm]
\indent 
We  define the { partial compositions} on ${\bf D}^{n}$-profunctors as the family of bifunctors $${{_{i}\circ_{j}}}:{\EuScript C}(m)\times {\EuScript C}(n)\rightarrow {\EuScript C}(m+n-2),$$ defined for arbitrary $m,n>1$, $1\leq i\leq m$ and $1\leq j\leq n$, as follows: for $F\in {\EuScript C}(m)$ and $G\in {\EuScript C}(n)$,  
$$(F\,{{_{i}\circ_{j}}}\, G) (y_1,\dots,y_{m+n-2})=$$ 
$$\int^{u,v} F(x_1,\dots,x_{i-1},u,x_{i},\dots ,x_{m-1})\times G(x_{m},\dots,x_{m+j-2},v,x_{m+j-1},\dots,x_{m+n-2})\times {\bf D}[u,v^{\ast}],$$
where the definition of  $x_i$,   $1\leq i\leq m+n-2$, can be read from the following picture:
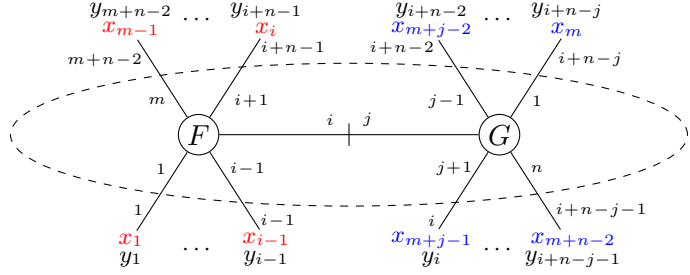
\begin{center}
\begin{tikzpicture}
\node (F) [circle, draw=black,inner sep=0.5mm] at (0,0) {$F$};
\node (d1) [circle, draw=none,inner sep=0.5mm] at (0,1.525) {\footnotesize $\cdots$};
\node (d2) [circle, draw=none,inner sep=0.5mm] at (0,-1.525) {\footnotesize $\cdots$};
\node (i) [circle, draw=none,inner sep=0.5mm] at (1.75,0.2) {\tiny $i$};
\node (j) [circle, draw=none,inner sep=0.5mm] at (2.25,0.2) {\tiny $j$};
\node (d3) [circle, draw=none,inner sep=0.5mm] at (4,1.525) {\footnotesize $\cdots$};
\node (d4) [circle, draw=none,inner sep=0.5mm] at (4,-1.525) {\footnotesize $\cdots$};
\node (F1) [rectangle, draw=none,inner sep=0.3mm] at (0.9,1.65) {\footnotesize $y_{i+n-1}$};
\node (F2) [rectangle, draw=none,inner sep=0.3mm] at (-0.9,1.65) {\footnotesize $y_{m+n-2}$};
\node (F3) [rectangle, draw=none,inner sep=0.3mm] at (-0.9,-1.65) {\footnotesize $y_1$};
\node (F3') [rectangle, draw=none,inner sep=0.3mm] at (-0.9,-1.4) {\footnotesize\color{red} $x_1$};
\node (F1') [rectangle, draw=none,inner sep=0.3mm] at (0.9,1.4) {\footnotesize\color{red} $x_i$};
\node (1) [rectangle, draw=none,inner sep=0.3mm] at (-0.8,-1) {\tiny $1$};
\node (1) [rectangle, draw=none,inner sep=0.3mm] at (-0.5,-0.45) {\tiny $1$};
\node (1) [rectangle, draw=none,inner sep=0.3mm] at (0.65,-0.45) {\tiny $i\!-\!1$};
\node (1) [rectangle, draw=none,inner sep=0.3mm] at (1.05,-1.15) {\tiny $i\!-\!1$};
\node (i) [rectangle, draw=none,inner sep=0.3mm] at (0.7,0.45) {\tiny $i\!+\!1$};
\node (i) [rectangle, draw=none,inner sep=0.3mm] at (1.25,1.15) {\tiny $i\!+\!n\!-\!1$};
\node (i) [rectangle, draw=none,inner sep=0.3mm] at (-0.55,0.45) {\tiny $m$};
\node (i) [rectangle, draw=none,inner sep=0.3mm] at (-1.25,1) {\tiny $m\!+\!n\!-\!2$};
\node (j) [rectangle, draw=none,inner sep=0.3mm] at (4.5,0.45) {\tiny $1$};
\node (j) [rectangle, draw=none,inner sep=0.3mm] at (5.2,1) {\tiny $i\!+\!n\!-\!j$};
\node (j) [rectangle, draw=none,inner sep=0.3mm] at (3.3,0.45) {\tiny $j\!-\!1$};
\node (j) [rectangle, draw=none,inner sep=0.3mm] at (2.7,1.15) {\tiny $i\!+\!n\!-\!2$};
\node (j) [rectangle, draw=none,inner sep=0.3mm] at (4.5,-0.45) {\tiny $n$};
\node (j) [rectangle, draw=none,inner sep=0.3mm] at (5.35,-1) {\tiny $i\!+\!n\!-\!j\!-\!1$};
\node (j) [rectangle, draw=none,inner sep=0.3mm] at (3.4,-0.45) {\tiny $j\!+\!1$};
\node (j) [rectangle, draw=none,inner sep=0.3mm] at (3.1,-1.15) {\tiny $i$};
\node (F2') [rectangle, draw=none,inner sep=0.3mm] at (-0.9,1.4) {\footnotesize\color{red}  $x_{m-1}$};
\node (F4) [rectangle, draw=none,inner sep=0.3mm] at (0.9,-1.65) {\footnotesize $y_{i-1}$};
\node (F4') [rectangle, draw=none,inner sep=0.3mm] at (0.9,-1.4) {\footnotesize\color{red} $x_{i-1}$};
\node (G1) [rectangle, draw=none,inner sep=0.3mm] at (4.9,1.65) {\footnotesize $y_{i+n-j}$};
\node (G2) [rectangle, draw=none,inner sep=0.3mm] at (3.1,1.65) {\footnotesize $y_{i+n-2}$};
\node (G3) [rectangle, draw=none,inner sep=0.3mm] at (3.1,-1.65) {\footnotesize $y_i$};
\node (G4) [rectangle, draw=none,inner sep=0.3mm] at (4.9,-1.65) {\footnotesize $\enspace y_{i+n-j-1}$};
\node (G1') [rectangle, draw=none,inner sep=0.3mm] at (4.9,1.4) {\footnotesize\color{blue} $x_m$};
\node (G2') [rectangle, draw=none,inner sep=0.3mm] at (3.1,1.4) {\footnotesize\color{blue}  $x_{m+j-2}$};
\node (G3') [rectangle, draw=none,inner sep=0.3mm] at (3.1,-1.4) {\footnotesize\color{blue} $x_{m+j-1}$};
\node (G4') [rectangle, draw=none,inner sep=0.3mm] at (4.9,-1.4) {\footnotesize\color{blue} $\enspace x_{m+n-2}$};
\node (G) [circle, draw=black,inner sep=0.5mm] at (4,0) {$G$};
\draw (F)--(G);
\draw (F)--(F1');
\draw (F)--(F2');
\draw (F)--(F3');
\draw (F)--(F4');
\draw (G)--(G1');
\draw (G)--(G2');
\draw (G)--(G3');
\draw (G)--(G4');
\draw (2,-0.15)--(2,0.15);
 \draw [dashed] (2,0) ellipse (4.5cm and 0.95cm);
\end{tikzpicture}
\end{center}
Formally, we have that  $x_i=y_{({\tau_{i,j}^{m,n}})_{y_1,\dots,y_{m+n-2}}(i)}$, where
$$(\tau^{m,n}_{i,j})_{y_1,\dots,y_{m+n-2}}: [m+n-2] \rightarrow [m+n-2]$$
 is the bijection defined by $$(\tau^{m,n}_{i,j})_{y_1,\dots,y_{m+n-2}}(k)= \begin{cases} 
     k,   &  1\leq k \leq  i-1 \\
      {k+n-1}, &  i\leq k \leq m-1 \\
     {k+n-m+i-j},   &   m\leq k \leq m+j-2 \\
  {k-m-j+i+1},   &   m+j-1\leq k \leq m+n-2. \\
   \end{cases}$$ 
\subsubsection{${\bf D}^{n}$-profunctors as categorified cyclic operads} \label{profunctor-categorified}
We  next give evidence that ${\bf D}^{n}$-profunctors, together with the family ${{_{i}\circ_{j}}}$, carry the structure of a skeletal categorified entries-only  cyclic operad.\footnote{${\bf D}^{n}$-profunctors actually carry the structure of a skeletal categorified entries-only  cyclic operad { with units (and relaxed unit law)},  the unit ${\it id}\in {\EuScript C}(2)$ being given by  ${\bf D}[-^{\ast},-]: {\bf D}\times {\bf D}\rightarrow {\bf Set}$.} 

\medskip
The family $\{{\EuScript C}(n)\}_{n\geq 2}$ extends to a functor ${\EuScript C}:{\bf{\Sigma}}^{\it op}\rightarrow {\bf Cat}$, as follows.
For $\sigma\in{\mathbb S}_{n}$ and $F:{\bf D}^{n}\rightarrow {\bf Set}$, the profunctor $F^\sigma$
is defined by $F^\sigma(y_1,\dots,y_n)=F(y_{\sigma^{-1}(1)},\dots,y_{\sigma^{-1}(n)})$. We define analogously the action of $\sigma$ on natural transformations.  
\smallskip

We next prove that the equivariance axiom \hyperlink{sEQ}{\texttt{(sEQ)}} 
holds on the nose.
\begin{lem}
For $\sigma_1\in{\mathbb S}_m$, $\sigma_2\in{\mathbb S}_n$,  $F\in{\EuScript C}(m)$, $G\in {\EuScript C}(n)$, $1\leq i\leq m$, $1\leq j\leq n$
 and $\sigma \in{\mathbb S}_{m+n-2}$ defined via $\sigma_1$ and $\sigma_2$ as in  \hyperlink{sEQ}{\em\texttt{(sEQ)}}, the following equality holds: 
\begin{equation}\label{eq_profunctors}(F^{\sigma_1}\,\,{_{{ {\sigma_1^{-1}}(i)}}\circ_{\sigma_2^{-1}(j)}}\,\, G^{\sigma_2})(y_1,\dots,y_{m+n-2})=(F\, {_i\circ_j} \, G)^{\sigma}(y_1,\dots,y_{m+n-2}).\end{equation}  
\end{lem}
\begin{proof}  
Let    $$(\delta^{m,n})_{y_1,\dots,y_{m+n-2}}:[m+n-2]\rightarrow \{y_1,\dots,y_{m+n-2}\}$$  and $$\pi^{m,n}_{i,j}: [m]\backslash\{i\}+[n]\backslash\{j\}\rightarrow [m+n-2]$$ be  functions defined by $(\delta^{m,n})_{y_1,\dots,y_m}(i)=y_i$, for all $1\leq i\leq m+n-2$, and 
$$\pi^{m,n}_{i,j}(k)= \begin{cases} 
     k,  &  1\leq k \leq  i-1 \\
     k-1, &  i+1\leq k \leq m \\
   \end{cases}$$ 
for $k\in [m]\backslash\{i\}$, and 

$$\pi^{m,n}_{i,j}(l)= \begin{cases} 
    l+m-1,   &   1\leq l \leq j-1 \\
  l+m-2,   &   j+1\leq l \leq n \\
   \end{cases}$$ 
for $l\in [n]\backslash\{j\}$, respectively. The equality \eqref{eq_profunctors} then follows by the definition of (the inverse of) $\sigma$, combined with the equality $$\kappa^{m,n}_{i,j}=\pi^{m,n}_{i,j}\circ (\tau^{m,n}_{i,j})^{-1}_{y_1,\dots,y_{m+n-1}}\circ (\delta^{m,n})_{y_1,\dots,y_{m+n-2}}.$$

\vspace{-0.8cm}
\end{proof}

For $F\in {\EuScript C}(m)$,  $G\in {\EuScript C}(n)$,  $1\leq i\leq m$ and $1\leq j\leq n$, we exhibit  $$\gamma_{F,G}^{i,j}: F\,{{_{i}\circ_{j}}}\, G \rightarrow (G\,{{_{j}\circ_{i}}}\, F)^{\sigma},$$ where $\sigma\in {\tt c}{\mathbb S}_{m+n-2}$ is  determined by $\sigma(1)=j+m-i$, as follows:
  $$(\gamma^{i,j}_{F,G})_{y_1,\dots,y_{m+n-2}}([(u,v,a,b,f)]_{\sim})=[(v,u,b,a,f^{\ast})]_{\sim}.$$ 
  It is easily seen that this correspondence is well-defined and bijective.

\medskip
We continue with sequential associativity. Let $F\in{\EuScript C}(m)$, $G\in {\EuScript C}(n)$, $H\in {\EuScript C}(p)$, $1\leq i\leq m$, $1\leq j\leq n$, $i\leq k\leq i+n-2$ and $1\leq l\leq p$. According to Definition \ref{co_skeletal}, we distinguish two cases:
\begin{itemize}
\item   $i\leq k\leq i+n-j-1$,  in which case we define a natural isomorphism 
$$\beta^{i,j;k,l}_{F,G,H}:(F\,{{_{i}\circ_{j}}}\, G)\,{{_{k}\circ_{l}}}\, H\rightarrow F\,{{_{i}\circ_{j}}}\, (G\,{{_{k-i+j+1}\circ_{l}}}\, H),$$
\item   $i+n-j\leq k\leq i+n-2$, in which case we define a natural isomorphism 
$$\beta^{i,j;k,l}_{F,G,H}:(F\,{{_{i}\circ_{j}}}\, G)\,{{_{k}\circ_{l}}}\, H\rightarrow F\,{{_{i}\circ_{j+p-1}}}\, (G\,{{_{k-i-n+j+1}\circ_{l}}}\, H).$$
\end{itemize}
\noindent
In both cases,
 the corresponding components  are defined by 
 $$(\beta^{i,j;k,l}_{F,G,H})_{y_1,\dots,y_{m+n+p-3}}([(u_1,v_1,(u_2,v_2,a,b,f),c,g)]_{\sim})=[(u_2,v_2,a,(u_1,v_1,b,c,g),f)]_{\sim}.$$
 
Again, this correspondence $(\beta^{i,j;k,l}_{F,G,H})_{y_1,\dots,y_{m+n+p-3}}$ is a bijection.
 
An easy diagram chase shows that $\beta^{i,j;k,l}_{F,G,H}$ and $\gamma^{i,j}_{F,G}$ satisfy  \hyperlink{gfgf}{\tt (s-$\beta \gamma$-\texttt{hexagon})}.
One would check the other (skeletal versions of) the coherence diagram similarly.

\medskip
Alternatively, at the price of using a slightly less standard presentation, but with the benefit of a simpler definition of composition, we can reformulate the whole example in our main  non-skeletal setting, as follows.  For each set $X$, considered as a discrete category, let ${\bf D}^X$ be the category of $X$ diagrams in ${\bf D}$, whose objects $\mathbb{X}$,  $\mathbb{Y},\ldots$ are thus named tuples of objects of ${\bf D}$. We set
${\EuScript C}(X)=[{\bf D}^X,{\bf Set}]$. In this setting, composition is defined straightforwardly:
$$(F\,{{_{x}\circ_{y}}}\, G) (\mathbb{Z})=
\int^{u,v} F(\mathbb{Z'}[x\leftarrow u])\times G(\mathbb{Z}''[y\leftarrow v])\times {\bf D}[u,v^{\ast}],$$
where $\mathbb{Z}'$ (resp. $\mathbb{Z}''$) is  the restriction of $\mathbb{Z}$ to $X\backslash\{x\}$ (resp. $Y\backslash\{y\}$), and where, say,  $\mathbb{Z}'[x\leftarrow u]$ is the extension of $\mathbb{Z}'$ to $X$ that maps $x$ to $u$. We leave it to the reader to check that all conditions of Definition \ref{cat} are verified.

\subsection{Feynman category for anticyclic operads}
This section is a development around Kaufmann and Ward's Feynman categories and their odd versions, introduced in \cite{kauf}. Feynman categories are monoidal categories with some additional structure, whose representations are operad-like notions.  

Our goal  is to illustrate the utility of categorified entries-only cyclic operads in  proving that the Feynman category 
for cyclic operads admits an odd version, which is, in turn, precisely  the Feynman category for anticyclic operads.  In \cite{kauf}, the authors gave a sketch of  the   underlying 2-categorical constructions. We recall here the relevant definitions, in particular, that of an anticyclic operad and of an ordered presentation. The latter embodies a coherence condition, which we make explicit.  We complete the work in  \cite{kauf} by proving this coherence condition in the cyclic case, as a consequence of our coherence result for categorified cyclic operads.

\subsubsection{Anticyclic operads}
Anticyclic operads were defined by Getzler and Kapranov in \cite[\S 2.10]{Getzler:1994pn},  as operads  with simultaneous composition equipped with an action of the cycles $\tau_n$ whose compatibility with operadic composition differs from the one of cyclic operads only in signs involved in the equations. In \cite{chapoton}, Chapoton gave and equivalent definition, based on operads with partial composition.   The entries-only version of his structure follows  from the equivalence of \cite[Theorem 2]{mo}. It is obtained from Definition \ref{entriesonly} by replacing the axiom \hyperlink{CO}{\tt{(CO)}} by the equality $$f\,{_x\circ_y}\, g = - g \,{_y\circ_x}\, f .$$ All other axioms  remain the same.

\subsubsection{Ordered presentations of categories} \label{ordered-presentation-section}
We begin by recasting the notion of a (small) category presented by generators and relations, in the language of polygraphs (or computads)~\cite{burroni,street} which we now briefly introduce. A {\em 1-polygraph} is given by sets $\Sigma_0$ and $\Sigma_1$ (the $0$-cells and the $1$-generators), and a source and a target map  from $\Sigma_1$ to $\Sigma_0$. A 1-polygraph freely generates a category whose objects are the $0$-cells and whose  collection  $\Sigma_1^*$ of morphisms   consists of the ``well-typed'' sequences of $1$-generators.  A {\em 2-polygraph} is given by 
a 1-polygraph plus a set  $\Sigma_2$  of  $2$-generators, together 
with a source and a target map  from $\Sigma_2$ to $\Sigma_1^*$, such  that the source (resp. the target) of the source of any generating $2$-cell is the source (resp. the target) of its target. A 2-polygraph freely generates a 2-category  with objects and 1-morphisms   as above,   whose collection  $\Sigma_2^*$ of $2$-morphisms  consists of the pasting diagrams (called ``polygons'' in \cite{kauf}) that can be built out from the generating $2$-cells (and their inverses).

In the language of polygraphs, a category   presented by generators and relations   can be defined as a quotient of a $2$-polygraph $\Sigma$ whose $0$-cells are the objects of the presented category, whose $1$-generators are its generators, and whose $2$-generators are the relations $s=t$, represented as (invertible) $2$-cells from $s$ to $t$.  This polygraph generates a $2$-category $\Sigma^*$, as indicated above, from which we can extract the category
$\pcat{\Sigma}$
whose objects are the   $0$-cells, and whose morphisms are the equivalence classes of $1$-cells for the relation defined by $s\equiv t$ if and only if there exists a $2$-cell in $\Sigma_2^*$  whose source and target are respectively $s$ and $t$.   We refer to  the survey \cite{Guiraud-Malbos}  and to \cite{Curien-Mimram} for more details on polygraphs, and on presentations of categories, respectively.

We now recall the notion of ordered presentation  \cite[Definition 5.2.3]{kauf}, expressed in the language of polygraphs. An ordered polygraph is a 2-polygraph $\Sigma$  equipped with a map $\nu:\Sigma_2\rightarrow\{+,-\}$, in such a way that the following {\em sign-coherence} condition holds: the sign assignment $\nu$ extends multiplicatively to every $2$-cell  in $\Sigma_2^*$ (giving the same  sign to a $2$-generator  and to its inverse).  We require that every two parallel $2$-cells    receive the same sign, so that this sign can actually be attributed unambiguously to every pair of parallel $1$-cells.  A presentation is called ordered when its associated $2$-polygraph can be ordered.

  From an ordered presentation $(\Sigma,\nu)$, we can extract the $ {\bf Ab}$-enriched category 
$\pcat{(\Sigma,\nu)}^{\it odd}$,  whose objects are  the  $0$-cells $A,B,\ldots$, and whose   homsets  are defined as follows: $\pcat{(\Sigma,\nu)}^{\it odd}[A,B]$  is   the abelian group freely generated by  all parallel $1$-cells with $A$ as source and $B$ as target, quotiented   by the subgroup generated by  $s-t$  (resp.  $s+t$) for all pairs $(s,t)$ of sign $+$ (resp. of sign $-$).

\subsubsection{Feynman category for cyclic operads}
The definition of the Feynman category for cyclic operads uses the formalism of graphs introduced in \S \ref{section_trees}, adapted by forgetting the decorations of corollas. Therefore, in what follows, by a graph ${\cal G}$, we shall mean a set of  corollas of the form $\bullet(x,y,z,\ldots)$, where $\bullet$ stands for ``unlabelled corolla'',   together with an involution $\kappa$ on the set $V({\cal G})$ of all half-edges of ${\cal G}$.
  
An aggregate is  a  totally disconnected graph, i.e., a graph whose involution is the identity. We define morphisms between aggregates as follows:
every pair $({\cal G},\tau)$ of a  graph ${\cal G}$ and a bijection $\tau$ from some set $Y$ to ${\it FV}({\cal G})$ gives rise to a morphism from  the aggregate formed by the corollas of ${\cal G}$   (obtained by forgetting the involution $\kappa$  that specifies the edges of ${\cal G}$) to  an aggregate which has as many corollas as there are connected components in ${\cal G}$, each corolla being of the form $\bullet(\tau^{-1}(x),\tau^{-1}(y),\tau^{-1}(z),\ldots)$, where $x,y,z,\ldots$ are the fixpoints of $\kappa$ 
in the corresponding connected component. Furthermore, all morphisms are defined in this way.
\begin{example}
For
$$ {\cal G}=\{\bullet(x,y,z)\:,\:\bullet(\underline{x},u,v) \:,\:\bullet(a,b,c);\kappa\},$$ where $\kappa(x)=\underline{x}$ and $\kappa(w)=w$ for all  $w\in\{y,z,u,v,a,b,c\}$, and $\tau:\{y',z',u',v',a',b',c'\}\rightarrow \{y,z,u,v,a,b,c\}$ defined by $\tau(w')=w$ for all $w\in\{y,z,u,v,a,b,c\}$, we have $$({\cal G},\tau): \{\bullet(x,y,z)\:,\:\bullet(\underline{x},u,v) \:,\:\bullet(a,b,c)\}\rightarrow \{\bullet(y',z',u',v') \:,\:\bullet(a',b',c')\}.\vspace{-0.6cm}$$ \hfill$\square$
\end{example}

The composition of morphisms of aggregates is given by $({\cal G}',\sigma') \circ ({\cal G},\sigma)=({\cal G}'', \sigma\circ\sigma')$, 
where ${\cal G}''$ is obtained by inserting at every node  of ${\cal G}' $  the corresponding connected  component of ${\cal G}$, using $\sigma$ to connect its free half-edges with those of the corolla that it replaces.

Aggregates and their morphisms  form a monoidal category, the tensor product being the juxtaposition of aggregates and graphs. We are interested in its monoidal subcategory 
${\it Cyc}$  which has  aggregates as objects  and whose morphisms are pairs $(F,\sigma)$, where $F$ is a forest, i.e., a disjoint union of unrooted trees.

\smallskip
The monoidal category ${\it Cyc}$ admits the following presentation.  We take  aggregates  as  $0$-cells.  The $1$-generators are  of  two kinds:
\begin{itemize}
\item[a)] { edge contractions}
 \begin{center}
\resizebox{10cm}{!}{\begin{tabular}{c c c}
\begin{tikzpicture}
 \node (d1) [circle,minimum size=2mm,inner sep=0.1mm,fill=none,draw=none] at (-1.25,0) {$\dots$};
 \node (d2) [circle,minimum size=2mm,inner sep=0.1mm,fill=none,draw=none] at (3.25,0) {$\dots$};
  \node (a) [circle,minimum size=2mm,inner sep=0.1mm,fill=black,draw=black] at (0,0) {};
\node (x) [circle, minimum size=1.5mm,inner sep=0.1mm,fill=none,draw=none] at (0.65,-0.375) {\footnotesize $x$};
 \node (b) [circle,minimum size=2mm,inner sep=0.1mm,fill=black,draw=black] at (2,0) {};
\node (y) [circle, minimum size=1.5mm,inner sep=0.1mm,fill=none,draw=none] at (1.35,-0.375) {\footnotesize $y$};
\node (z) [circle, minimum size=1.5mm,inner sep=0.1mm,fill=none,draw=none] at (1.35,-0.8) {};
 \draw (a) -- (0,0.75);
 \draw (a) -- (-0.55,-0.55);
 \draw (a) -- (0.55,-0.55);
 \draw (b) -- (2.55,0.55);
 \draw (b) -- (1.45,0.55);
 \draw (b) -- (1.45,-0.5);
 \draw (b) -- (2.55,-0.5);
\end{tikzpicture}
&
\enspace
\begin{tikzpicture}
 \node at (1,0.5) {$_x\circ_y$};
 \node at (1,-0.4) {};
\draw[|->] (0.1,0.3)--(1.9,0.3);

\end{tikzpicture}
\enspace &

\begin{tikzpicture} 
 \node (d1) [circle,minimum size=2mm,inner sep=0.1mm,fill=none,draw=none] at (-1.25,0) {$\dots$};
 \node (d2) [circle,minimum size=2mm,inner sep=0.1mm,fill=none,draw=none] at (1.25,0) {$\dots$};
\node (a) [circle, minimum size=2mm,inner sep=0.1mm,fill=black,draw=black] at (0,0) {};
\node (z) [circle, minimum size=1.5mm,inner sep=0.1mm,fill=none,draw=none] at (0,-0.8) {};
 \draw (a) -- (0,0.75);
 \draw (a) -- (-0.45,-0.65);
 \draw (a) -- (0.45,-0.65);
\draw (a) -- (-0.75,0.2);
\draw (a) -- (0.75,0.2);
\node (z) [circle, minimum size=1.5mm,inner sep=0.1mm,fill=none,draw=none] at (1.35,-0.6) {};
\end{tikzpicture}
\end{tabular}}
\end{center}
where  $_x\circ_y$ stands for  $(F,{\it id})$, with \,$F=$\, \resizebox{1.5cm}{!}{\begin{tikzpicture}
 \node (a) [circle, minimum size=2mm,inner sep=0.1mm,fill=black,draw=black] at (0.55,-0.2) {};
 \node (b) [circle, minimum size=2mm,inner sep=0.1mm,fill=black,draw=black] at (1.45,-0.2) {};
\draw (a)--(b);
\draw (a)--(0.15,0.1);
\draw (a)--(0.15,-0.5);
\draw (b)--(1.45,0.2);
\draw (b)--(1.45,-0.6);
\draw (b)--(1.9,-0.2);
\node (x) [circle, minimum size=1.5mm,inner sep=0.1mm,fill=none,draw=none] at (0.85,-0.35) {\tiny $x$};
\node (y) [circle, minimum size=1.5mm,inner sep=0.1mm,fill=none,draw=none] at (1.15,-0.35) {\tiny $y$};
\draw (1,-0.1)--(1,-0.3);
\end{tikzpicture}}\, , and
\item[b)] {relabelling} isomorphisms  
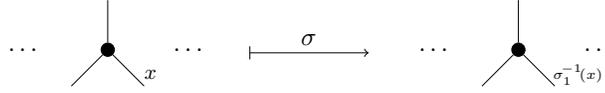
\begin{center}
\resizebox{8.5cm}{!}{\begin{tabular}{c c c}
\begin{tikzpicture}
 \node (d1) [circle,minimum size=2mm,inner sep=0.1mm,fill=none,draw=none] at (-1.25,0) {$\dots$};
 \node (d2) [circle,minimum size=2mm,inner sep=0.1mm,fill=none,draw=none] at (1.25,0) {$\dots$};
  \node (a) [circle, minimum size=2mm,inner sep=0.1mm,fill=black,draw=black] at (0,0) {};
\node (x) [circle, minimum size=1.5mm,inner sep=0.1mm,fill=none,draw=none] at (0.65,-0.375) {\footnotesize $x$};
\node (z) [circle, minimum size=1.5mm,inner sep=0.1mm,fill=none,draw=none] at (0.65,-0.6) {};
 \draw (a) -- (0,0.75);
 \draw (a) -- (-0.55,-0.55);
 \draw (a) -- (0.55,-0.55);
\end{tikzpicture}
&
\enspace\begin{tikzpicture}
 \node at (1,0.5) {$(C,{\sigma})$};
 \node at (1,-0.2) {};
\draw[|->] (0.1,0.3)--(1.9,0.3);

\end{tikzpicture} \enspace&
\begin{tikzpicture}
 \node (d1) [circle,minimum size=2mm,inner sep=0.1mm,fill=none,draw=none] at (-1.25,0) {$\dots$};
 \node (d2) [circle,minimum size=2mm,inner sep=0.1mm,fill=none,draw=none] at (1.25,0) {$\dots$};
  \node (a) [circle, minimum size=2mm,inner sep=0.1mm,fill=black,draw=black] at (0,0) {};
\node (x) [rectangle, minimum size=1.5mm,inner sep=0.1mm,fill=none,draw=none] at (0.9,-0.375) {\tiny $\sigma_1^{\!-1}\!(x)$};
\node (z) [circle, minimum size=1.5mm,inner sep=0.1mm,fill=none,draw=none] at (0.9,-0.6) {};
 \draw (a) -- (0,0.75);
 \draw (a) -- (-0.55,-0.55);
 \draw (a) -- (0.55,-0.55);
\end{tikzpicture}
\end{tabular}}
\end{center}
where $C$ is the domain corolla. We shall write $\sigma$ for $(C,\sigma)$.
\end{itemize}
The 2-generators express that:
\begin{itemize}
\item[1.]the  two possible edge contractions of the same edge are equal, i.e.,  ${_x\circ_y}={_y\circ_x}$,
\item[2.] { edge  contractions commute}, i.e., ${_x\circ_y}\, {_z\circ_u}= {_z\circ_u}\, {_x\circ_y}$ (with $y,z$ attached to the same corolla),
\item[3.] relabellings commute with edge contractions, i.e., ${_{\sigma_1^{-1}(x)}\circ_{\sigma_2^{-1}(y)}} (\sigma_1\times \sigma_2)=\sigma {_x\circ_y} $, where $\sigma$ is induced in the obvious way from $\sigma_1$ and $\sigma_2$, and
\item[4.] the obvious ``action'' laws regarding labellings, forcing the 1-generator ${\it id}$ to be the identity and  the composition of two $1$-generators $\tau$ and $\sigma$ to be the $1$-generator $\sigma\circ\tau$, hold.
\item[5.] 1-generators acting on disjoint parts of an aggregate commute.
\end{itemize}

\noindent
The $2$-generators corresponding to (1)-(3)  look like this:
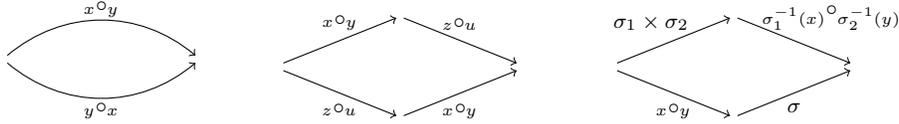
\begin{center}
\begin{tikzpicture}
\draw [->]   (0,0) to[out=40,in=140] (2.5,0) node [midway,above,xshift=1.25cm,yshift=0.375cm] {\scriptsize     ${_x\circ_y}$};
\draw [->]   (0,-0.1) to[out=-40,in=220] (2.5,-0.1) node [midway,below,xshift=1.25cm,yshift=-0.5cm] {\scriptsize     ${_y\circ_x}$};
\end{tikzpicture}\quad\quad\quad
\begin{tikzpicture}
\draw [->] (0,0)--(1.5,0.6) node [midway,above] {\scriptsize     ${_x\circ_y}$};
\draw [->] (1.6,0.6)--(3.1,0) node [midway,above] {\scriptsize     ${_z\circ_u}$};
\draw [->] (0,-0.1)--(1.5,-0.7) node [midway,below] {\scriptsize     ${_z\circ_u}$};
\draw [->] (1.6,-0.7)--(3.1,-0.1) node [midway,below] {\scriptsize     ${_x\circ_y}$};
%\draw [->,double,double equal sign distance,-implies] (1.56,0.25)--(1.56,-0.25);
\end{tikzpicture}\quad\quad\quad
\begin{tikzpicture}
\draw [->] (0,0)--(1.5,0.6) node [midway,above,xshift=-0.3cm] {\scriptsize     $\sigma_1\times\sigma_2$};
\draw [->] (1.6,0.6)--(3.1,0) node [midway,above,xshift=0.5cm] {\scriptsize     ${_{\sigma_1^{-1}(x)}\circ_{\sigma_2^{-1}(y)}}$};
\draw [->] (0,-0.1)--(1.5,-0.7) node [midway,below] {\scriptsize     ${_x\circ_y}$};
\draw [->] (1.6,-0.7)--(3.1,-0.1) node [midway,below] {\scriptsize     $\sigma$};
%\draw [->,double,double equal sign distance,-implies] (1.56,0.25)--(1.56,-0.25);
\end{tikzpicture}
\end{center}

Let us call $\Sigma_{\it Cyc}$ the 2-polygraph described above.
It follows from the results of \cite{co} (see Remark 2.10 in that paper, in particular) that the category presented by  $\Sigma_{\it Cyc}$ is isomorphic to ${\it Cyc}$, and that
 the representations  (in ${\bf Set}$) of ${\it Cyc}$ (endowed with its Feynman category structure) are exactly  the cyclic operads of Definition \ref{entriesonly}.  
  
We are now ready to state our result.
 
\begin{thm}\label{fy} $\Sigma_{\it Cyc}$ can be turned into an  ordered presentation  by assigning  $-$ to the relation ${_x\circ_y}= {_y\circ_x}$ and $+$ to all other relations.
\end{thm}
\begin{proof}
The coherence diagrams of categorified entries-only cyclic operads with relaxed equivariance  (and pseudo-actions) can be taken as $3$-generators in a $3$-polygraph $\Sigma'_{\it Cyc}$ extending $\Sigma_{\it Cyc}$.  In this setting, we can read
 our coherence theorem  as saying that the space between any parallel $2$-cells can be filled with a combination of  these $3$-generators (and this is where we need the isomorphisms $\varepsilon, \varepsilon_2,\varepsilon_3$ of Section  \ref{eq_relax} as explicit witnesses, or polygons).  This can be made formal by building a categorified cyclic operad ${\EuScript C}$ out of   $\Sigma'_{\it Cyc}$: one defines ${\EuScript C}(\set{x_1,\ldots,x_n})$ as the category whose objects are the  $1$-cells with  the corolla $\bullet(x_1,\ldots,x_n)$ as codomain and whose morphisms are the $2$-cells between them.
 
 It follows that it suffices to check sign coherence  for  the pairs of cells forming the border of the $3$-generators.   This in turn can be checked easily by reading the sequential associator, the commutator, the  equivariance and pseudo-action isomorphisms as 
$$\beta^{x,\underline{x};y,\underline{y}}_{f,g,h}:(f\,{_x\circ _{\underline{x}}}\,g)\,{_y\circ _{\underline{y}}}\,h\rightarrow f\,{_x\circ _{\underline{x}}}\,(g\,{_y\circ _{\underline{y}}}\,h),\quad  \gamma^{x,y}_{f,g}:f\,{_x\circ _y}\,g\rightarrow -g\,{_y\circ _x}\,f,$$
 $$\varepsilon^{x,y;x',y'}_{f,g;\sigma}:(f {_x\circ_y} \, g)^{\sigma}\rightarrow f^{\sigma_1}{_{{ x'}}\circ_{y'}}\, g^{\sigma_2},
 \quad {\varepsilon_2}_f: f^{\it id}\rightarrow f, \quad  {\varepsilon_3}_f^{\sigma,\tau}:(f^\sigma)^\tau\rightarrow f^{\sigma\circ\tau}.$$
and by verifying that  adding the sign information to the   coherence diagrams leaves them well-defined. This follows from the observation that each of these diagrams 
involves an even number of instances of the commutator.  
\end{proof}
Thus we can define  an ${\bf Ab}$-enriched Feynman category $\pcat{({\it Cyc},\nu)}^{\it odd}$, where $\nu$ is the sign assignment from Theorem \ref{fy}, whose representations are  anti-cyclic operads.  \\[0.1cm]
\indent  Knowing that odd versions of Feynman categories arise in this way allows us to benefit from general constructions.  Indeed,
the importance of Feynman categories which admit an odd version  is reflected in the fact that they also admit the bar and cobar costructions, as well as the Feynman transform. For details, we refer to \cite[\S 7.4]{kauf}.

\section*{Conclusion and further study}
An overview of the categorifications  established    in this paper is given in the table below.
 \begin{center}
{\footnotesize \begin{tabular}{rccc}  
    \toprule
  &  &   \multicolumn{2}{c}{{\textsc{{Categorified cyclic operads}}} }   \\
  \midrule
&  & {\textsc{{Entries-only}}} &  \textsc{{Exchangeable-output }}    \\
\cline{3-4}\\[-0.12cm]
    {\textsc{{Definitions}}}        & & \small{ Definition \ref{cat}}  & \small{Definition \ref{ex_out_cat_non-skeletal}}      \\[0.2cm]
     \textsc{{Coherence proof}}   & &\small{ \S \ref{cohe}}  & \small{Theorem \ref{t5}}       \\
    \bottomrule
  \end{tabular} }
\end{center}

\indent As future work, we hope to use the categorification methods of this paper  in order to categorify  other variations of cyclic operads, primarily    non-symmetric cyclic operads of \cite{cgr} and \cite{mm}, as well as  modular operads of \cite{mm}. Finally, inspired by the combinatorial approach to operadic polytopes made   in \cite{DP-HP} and further developed in  \cite{cyc2},  we hope to characterize polytopes which describe coherences and higher coherences of categorified cyclic operads.
\bibliographystyle{alpha}

\end{document}